\documentclass[11pt,a4paper,reqno,final]{article}
\usepackage{amsmath,amsfonts,amssymb,titlesec}
\usepackage{amsthm}
\usepackage{cite}
\usepackage{graphicx}
\usepackage{stix}
\usepackage{mathtools}
\usepackage{esint}
\usepackage{multirow}
\usepackage{authblk}
\usepackage{tikz}
\usepackage{verbatim}
\usepackage{subcaption}
\usepackage[left=1.5cm,right=1.5cm,top=1.5cm, bottom = 2cm]{geometry}
\usepackage[english]{babel}
\usepackage{booktabs}
\usepackage{caption}
\usepackage{cancel}
\usepackage[normalem]{ulem}

\titlespacing*{\subsection}
{0pt}      
{0.5ex}    
{0.5ex}    
\titlespacing*{\subsection}
{0pt}{2ex}{2ex}

\titlespacing*{\section}
{0pt}      
{0.5ex}    
{0.5ex}    

\newcounter{constcount}
\newcommand{\const}[1]{%
  \refstepcounter{constcount}%
  \expandafter\xdef\csname constlabel@#1\endcsname{\theconstcount}%
  C_{\theconstcount}%
}
\newcommand{\constref}[1]{C_{\csname constlabel@#1\endcsname}}
\usepackage[colorlinks=true,breaklinks=true,linkcolor=lightblue,citecolor=lightgreen,urlcolor=lightblue]{hyperref}
\newcounter{cnstcnt}

\numberwithin{figure}{section}
\numberwithin{table}{section}
\numberwithin{equation}{section}
\definecolor{lightblue}{rgb}{0.22,0.45,0.70}
\definecolor{lightgreen}{rgb}{0.22,0.50,0.25}

\newcommand{\trinl}{\ensuremath{|\!|\!|}}
\newcommand{\trinr}{\ensuremath{|\!|\!|}}
\newcommand{\mycomment}[1]{}
\newtheorem{thm}{Theorem}[section]

\newtheorem{lemma}[thm]{Lemma} 
 
\newtheorem{rem}[thm]{Remark}

\newcommand\norm[1]{\lVert#1\rVert}
\newcommand{\cred}[1]{{\leavevmode\color{red}{#1}}}
\newcommand{\cblue}[1]{{\leavevmode\color{blue}{#1}}}

\newcommand{\bx}{{\boldsymbol{x}}}

\newcommand{\dx}{\,\mathrm{d}\bx}
\newcommand{\pw}{\mathrm{pw}}

\newcommand{\m}{\mathrm{M}}
\newcommand{\ds}{\,\mathrm{d}s}
\newcommand{\dt}{\,\mathrm{d}t}

\allowdisplaybreaks

\title{\bf Unconditionally stable and energy conserving discretization \\of the dynamic von K\'arm\'an equations}
\author{Carsten Carstensen\thanks{Institut für Mathematik, Humboldt-Universität zu Berlin, 10117 Berlin, Germany and Distinguished Visiting Professor, Department of Mathematics, Indian Institute of Technology Bombay, Powai, Mumbai, Maharashtra ({\tt  cc@math.hu-berlin.de}).} \qquad 
Neela Nataraj\thanks{Department of Mathematics, Indian Institute of Technology Bombay, Powai, Mumbai, Maharashtra, India ({\tt neela@math.iitb.ac.in}).} \qquad 
Ricardo Ruiz-Baier\thanks{School of Mathematics, Monash University, 9 Rainforest Walk, 3800 Clayton VIC, Australia
({\tt ricardo.ruizbaier@monash.edu}).} \qquad 
Aamir Yousuf\thanks{IITB–Monash Research Academy, Indian Institute of Technology Bombay, Powai, Mumbai, Maharashtra, India ({\tt aamir72@iitb.ac.in}).}} 
\begin{document}
\maketitle
\begin{abstract}
\noindent
A fully discrete approximation of the dynamic von K\'arm\'an equations combines nonconforming Morley finite element methods  for spatial discretization with an energy conserving {\it modified}  unconditionally stable Newmark second-order time-stepping scheme. Brouwer’s fixed-point  theorem establishes existence of a solution to the fully discrete scheme and further uniqueness and stability estimates follow for small loads. Optimal order a priori error estimates in the piecewise energy norm with quadratic convergence in time are derived for the fully discrete scheme. The results of the  numerical experiments validate the theoretical error bounds.

\end{abstract}
\section{Introduction}
\subsection{Mathematical Model}
The von K\'arm\'an equations provide a fundamental mathematical framework for describing the nonlinear bending behavior of thin elastic plates. These equations, originally developed in the context of classical elasticity, account for large deflections while retaining the assumption of small strains, making them particularly useful in modeling geometrically nonlinear deformations.
This paper analyzes a discretization scheme for the {dynamic von K\'arm\'an system}
\begin{subequations}\label{strong-form}
\begin{align}
&u_{tt} + \Delta^2 u -[u,v] = f(\bx,t) \text{ at } (\bx,t) \in \Omega \times (0,T ],\label{stron-1}\\
& \quad\;\;\Delta^2 v +\frac{1}{2}[u,u]= 0 \text{ at }(\bx,t) \in \Omega \times [0,T]
\label{stron-2}
\end{align}
with initial and clamped boundary  conditions  
\begin{equation}
u(\bx,0)=u_0(\bx), \quad u_t(\bx,0)=u_1(\bx) \text{ in }\Omega;\quad    u=\frac{\partial u}{\partial  {\bf n}} =0, v=\frac{\partial v}{\partial  {\bf n}} =0 \text{ on }\partial \Omega \times (0,T].
\label{P1 strong_icbc}
\end{equation}
\end{subequations}
Here and throughout the paper, $[\bullet,\bullet]$ is the Monge--Amp\`ere form operator or von K\'arm\'an bracket,
$$[u,v]:=\frac{\partial^2 u}{\partial
				x^2}\frac{\partial^2 v}{\partial
				y^2}+\frac{\partial^2 u}{\partial
				y^2}\frac{\partial^2 v}{\partial
				x^2}-2\frac{\partial^2 u}{\partial
				x\partial
				y}\frac{\partial^2 v}{\partial
				x\partial
				y},$$ 
while $\Omega$ is a bounded polygonal Lipschitz domain in $\mathbb{R}^2$ with outward-pointing unit normal $ {\bf n}$ to the boundary $\partial \Omega$, $\Delta^2 w:= {\partial^4 w}/{\partial	x^4} + 2 {\partial^4 w}/{\partial x^2 \partial y^2}	+{\partial^4 w}/{\partial y^4}$  denotes the biharmonic operator, $\partial u / \partial {\bf n} = \nabla u \cdot {\bf n} $ (resp. $\partial v / \partial  {\bf n} = \nabla v \cdot {\bf n} )$ is the outer normal derivative of $u$ (resp. $v$) on $\partial \Omega$, and  $u_{t} $, $u_{tt} $ denote the first- and second-order  derivatives with respect to time, respectively. The system of equations \eqref{strong-form} describes the transversal displacement $u$ and the Airy-stress function $v$ of a vibrating plate, whose boundary is clamped in transversal direction but free in horizontal direction. For a derivation of the equations from physical principles, we refer to \cite{MR569597} and the references therein.

\subsection{Literature overview}
Several contributions in the literature discuss the discretization of the static von K\'arm\'an equations \cite{MR4229191,MR3554363,MR4230429,mallik2016nonconforming}; \cite{MR229978,MR220472,MR1477663} and the references therein study existence of solutions, regularity, and bifurcation phenomena. The
 weak solutions $u,v\in H^2_
 0(\Omega )$ to the static von K\'arm\'an equations belong to $H^2_0(\Omega)\cap H^{2+\sigma}(\Omega)$ with the index of elliptic regularity $\sigma>1/2$ determined by the interior
 angles of the polygonal boundary $\partial \Omega$ with $\sigma=1$ for convex domains \cite{MR595625}. In \cite{MR3554363} and \cite{mallik2016nonconforming}, the authors analyze conforming and nonconforming finite elements for the static von K\'arm\'an model and derive the optimal-order error estimates in the energy, $H^1$, and $L^2$  norms. A unified analysis  for fourth-order quadratic semilinear problems in  \cite{MR4630544}  applies to Morley, discontinuous Galerkin, $C^0$ interior penalty, and the weakly over-penalized symmetric interior penalty schemes.  
A conforming $C^1$ virtual element method in \cite{MR4229191} approximates  isolated solutions to the von K\'arm\'an equations.

Despite the rich literature on the numerical analysis of the von K\'arm\'an static model, not many results are available for its challenging time-dependent counterpart. The underlying biharmonic wave equation has been the subject of several numerical studies: \cite{MR4444402} considers the  $C^1$-conforming Bogner--Fox--Schmit elements in space coupled with Galerkin and collocation methods in time and \cite{MR4971619} analyzes lowest-order nonstandard finite element methods (Morley, discontinuous Galerkin, and $C^0$ interior penalty) combined with explicit and implicit time-stepping schemes.
One reference for the dynamic von K\'arm\'an equation is \cite{MR3403715} with  a semidiscrete (in time) energy conserving finite difference schemes. The dynamic von K\'arm\'an equation \eqref{strong-form} over rectangular domains in \cite{yosibash2004collocation} employs  Chebyshev and Legendre-collocation methods for the spatial discretization and the implicit Newmark-$\beta$ scheme combined with a nonlinear fixed-point algorithm in each time step. 
Bilbao et al.\cite{MR2371355,MR3403715,bilbao2023explicit} have developed a line of finite-difference schemes for the dynamic von Kármán plate, from closed-form energy-conserving but conditionally stable schemes.

\subsection{Contributions}
This paper presents a rigorous analysis and implementation of a {\it novel} fully discrete scheme for the dynamic von K\'arm\'an equations with clamped boundary conditions. A nonconforming  Morley finite element method for the spatial discretization and a {\it modified} Newmark second-order time discretization scheme in this article lead to an {\it unconditionally stable} scheme. 
The main contributions of this work are as follows.

\begin{itemize}
\item 
An implicit modified Crank--Nicolson scheme leads to an unconditionally stable quadratic scheme with
optimal convergence rates for the initial discretization.  
\item 
A classical implicit time-integration approach for plate dynamics is the Newmark constant-average acceleration method \cite{MR3395524}
that is unconditionally stable for linear problems.   
\textit{It 
is 
observed that the Newmark constant-average 
acceleration method loses its unconditional stability 
in the nonlinear problem} \cite[Page no. 9]{akay1980dynamic}. Since the Courant--Friedrichs--Lewy (CFL) condition $k =\mathcal{O}(h^2)$ \cite[eqn. (36)]{akay1980dynamic} appears too restrictive, we
  propose a {\it new} scheme which  preserves the  \textit{energy conservation} property  at discrete level and  exploits the advantage of unconditional stability of implicit schemes.

\item 
The smallness assumption on the data $u_0$, $u_1$, and the time step ensures the well-posedness of the fully discrete scheme. 
Optimal-order a priori error estimates in the energy norm are obtained for the fully discrete solution under additional smallness assumptions on $f$: We obtain quadratic  convergence in time  and optimal   convergence rates in space.

\item Numerical experiments  validate the theoretical analysis and confirm the predicted spatial and temporal convergence rates in Example 1 of Section~\ref{num-sec}. Example 2 in  Section~\ref{num-sec} demonstrates that the von K\'arm\'an thin   plate model is effective. 
The findings indicate that as the plate thickness decreases, the two-dimensional simulations closely approximate the results from three-dimensional modeling.
\end{itemize}
\subsection{Outline of the paper}
The remaining parts of this section introduce  overall notation. 
Section~\ref{sec-main results} presents the weak formulation of the problem,  describes the Morley finite element space and the fully discrete scheme, and states the main results of the paper.  Proofs follow in Sections~\ref{sec-fully-ini}–\ref{sect-fully-full}. Section~\ref{Sec-aux} introduces a companion operator, a Ritz projection, and  other auxiliary results. Section~\ref{sec-fully-ini} establishes the well-posedness and a priori error estimates for the initial time steps of the fully discrete scheme, while Section~\ref{sect-fully-full} addresses the remaining time levels. Finally, Section~\ref{num-sec} presents the Newton--Raphson procedure together with the results of  numerical experiments.

\subsection{Notation}\label{sub-notation}
Standard notations on Lebesgue and Sobolev spaces \cite{MR2597943} apply  throughout the paper. We denote the $L^2$ scalar product by $(\bullet, \bullet)_{L^2(\Omega)}$ and its induced norm by $\| \bullet \|$; the same  notation applies to the vector and tensor-valued functions. 
Throughout this paper, $\mathcal{T}$ denotes a shape-regular triangulation of the bounded polygonal  Lipschitz domain $\Omega$ into triangles, $H^m({\cal T})$ denotes the Hilbert space $ {\prod}_{{{K \in {\cal T}}}}H^m (K)$, and $P_r({\cal T}) \subset L^2(\Omega)$ is the space of  piecewise  polynomials of degree at most $r$. The notation \(\trinl \bullet \trinr\) denotes the global seminorm
\(|\bullet|_{H^2(\Omega)}\) and  \(\trinl \bullet \trinr_{\mathrm{pw}}\) is the piecewise seminorm \(|\bullet|_{H^2(\mathcal T)}\) for ${\mathcal T}-$piecewise \(H^2\)-functions.
The operators \(D^2\) and \(D^2_{\mathrm{pw}}\) denote the global and piecewise Hessian.   The notation $a \lesssim b$ abbreviates $a \le Cb$ with a  generic constant $C$, that is independent of the discretization parameters. The  weighted arithmetic-geometric mean  inequality $ab \le \frac{\epsilon}{2}a^2 + \frac{1}{2\epsilon}b^2$ for  positive real numbers $a ,b ,\epsilon$ applies throughout.

\medskip \noindent 
Let $(X,\|\bullet\|_X)$ be a Hilbert space  and  $0\le a < b \le T$, the Bochner space $L^p(a,b;X)$ consists of all strongly measurable functions $\varphi:(a,b) \to X$ \cite{MR2597943} with finite norm
$$\displaystyle \norm{g}_{L^p(a,b;X)}:=\Big(\int_a^b \norm{g(t)}_X^p \dt\Big)^{1/p} \text{ for }\ 1\le p< \infty \quad \text{ and else }\quad 
 \norm{g}_{L^\infty(a,b;X)}:=\underset{a \le t \le b}{\text{ess sup}} \norm{g(t)}_X.$$ 
 For $m \in \mathbb{N}$, the Sobolev-Bochner space $H^m(a,b;X)$ consists of all functions $g \in L^2(a,b;X)$ such that $ \frac{d^k g}{dt^k} \in L^2(a,b;X)$ $\text{ for } k = 0, 1, \ldots, m$ and is 
equipped with the norm
$$\|g\|_{H^m(a,b;X)} := \left(\sum_{k=0}^{m} \left\|\frac{d^k g}{dt^k}\right\|_{L^2(a,b;X)}^2\right)^{1/2} = \left(\sum_{k=0}^{m} \int_a^b \left\|\frac{d^k g(t)}{dt^k}\right\|_X^2 \, dt\right)^{1/2}.$$
The space  
$C^k([a,b];X)$ denote all $k-$times continuously differentiable  functions $g :[a,b] \rightarrow X$ with 
$$\norm {g}_{C^k([a,b];X)}:=\underset{0 \le i \le k}{\sum} \underset{a \le t\le b}{\max} \left\|\frac{\partial^i g}{\partial t^i}\right\|_X < \infty.$$

\section{Main results}\label{sec-main results}
This section presents the weak formulation of \eqref{strong-form}, summarizes  its well-posedness results from the literature, and derives energy conservation identities in Subsection~\ref{subsec-weak formultion}. Subsection~\ref{subsec Preliminaries} discusses the spatial and temporal discretization framework. Subsection~\ref{motivation} motivates an unconditionally stable modified Newmark fully discrete scheme \eqref{P4fully}  from the energy conservation identity \eqref{energy-c0} at the discrete level and  Subsection~\ref{sub-main-full}  proposes this scheme to approximate the solution at higher time steps. Subsection~\ref{P1 fully_discrete_section} introduces the initial approximations at the first two time steps using a modified Crank–Nicolson scheme. The stability and  error estimates for the fully-discrete scheme are stated in Subsection~\ref{key}.
\subsection{Weak formulation}\label{subsec-weak formultion}
Let $a(\bullet,\bullet):H^2_0(\Omega) \times H^2_0(\Omega) \rightarrow {\mathbb R}$ denote the energy scalar product in $H^2_0(\Omega)$ defined by 
$$ \displaystyle a(\varphi,\psi):=\int_\Omega D^2\varphi:D^2\psi \dx \text{ for all } \varphi, \psi \in H^2_0(\Omega).$$
Let  $b(\bullet,\bullet,\bullet):H^2_0(\Omega) \times H^2_0(\Omega)\times H^2_0(\Omega) \rightarrow {\mathbb R}$  be the  symmetric trilinear form given by
$$ \displaystyle b(\varphi,\phi,\psi):=-\int_\Omega [\varphi,\phi]\psi \dx \text{ for all }\varphi,\phi,\psi \in H^2_0(\Omega).$$
{The global Sobolev embedding   $H^2_0(\Omega)  \hookrightarrow L^\infty({\Omega})$ \cite{MR450957} and the Friedrichs inequality reveal $\norm{\varphi}_{L^\infty({\Omega})} \le C_{\rm S}\trinl \varphi \trinr$ for all  $\varphi \in H^2_0(\Omega)$   and  $C_{\rm S}>0$. H\"older and Cauchy–Schwarz inequalities lead to
\begin{align}
  |b(\varphi,\chi,\psi)|\le \norm{\psi}_{L^\infty({\Omega})}\int_\Omega |[\varphi,\chi]| \dx\le C_{\rm S}\trinr \psi \trinr\int_\Omega |[\varphi,\chi]| \dx \le C_{\rm S} \trinl \varphi \trinr \trinl \chi \trinr\trinl\psi \trinr.\label{b-bdd-c}
\end{align}
\medskip \noindent
The weak formulation of \eqref{strong-form}  seeks   $(u,v) \in \big(L^\infty(0,T;H_0^2(\Omega))\big)^2$ with $u_t,u_{tt}\in L^\infty(0,T;L^2(\Omega))$ such that, at a.e  $0< t <T$, 
\begin{subequations}\label{P1 weak_form}
\begin{align}
        &(u_{tt},\varphi)_{L^2(\Omega)}+a(u,\varphi)+b(u,\varphi,v)= (f,\varphi)_{L^2(\Omega)}\quad  \text{ for all }\varphi \in H_0^2(\Omega), \label{P1 weak_form1}\\  
   &\qquad\qquad a(v,\psi)-\frac{1}{2}b(u,u,\psi)=0 \qquad\qquad\;\; \text{ for all }\psi \in H_0^2(\Omega),\label{P1 weak_form2}\\
    &\qquad\qquad u(0)=u_0 \text{ and } u_t(0)=u_1.\label{P1 weak_form3}
\end{align}
\end{subequations}
The dependence of functions on $t$ will be skipped whenever there is no risk of confusion, e.g., $u:=u(t)$ and $v=v(t)$ in \eqref{P1 weak_form}.

\medskip
\noindent\textbf{Existence, uniqueness, and regularity results}

\medskip \noindent 
The well-posedness of the dynamic von K\'arm\'an equations depends on the regularity of the given  data and the smoothness properties of the domain. In Table \ref{tab:von_karman_summary}, we present the main existence and uniqueness results from the literature when $\Omega$ is a smooth bounded domain in $\mathbb{R}^2$. These results  highlight the assumptions on given data and the corresponding regularity of solution.
%
Moreover, under the assumptions $u_0 \in H^2_0(\Omega)$, $u_1\in L^2(\Omega)$, and $f \in L^2(0,T;L^2(\Omega))$ (see \cite{MR1621722} in Table~\ref{tab:von_karman_summary} below), the Faedo--Galerkin technique (cf. \cite[Eqn.~4.3]{MR4279861}) leads to the following  {\it energy conservation} identity
\begin{align}
     \|u_t\|^2+\|D^2u\|^2+\|D^2v\|^2=\|u_1\|^2+\|D^2u_0\|^2+\|D^2v_0\|^2+2\int_0^t(f,u_t)_{L^2(\Omega)}\dt \quad \text{at  any }t \in [0,T].\label{energy-c0}
\end{align}
\begin{table}[htbp]
\centering
\caption{Summary of existence, uniqueness, and regularity results for \eqref{strong-form} ($\Omega \subset \mathbb{R}^2$ is smooth and bounded)}
\label{tab:von_karman_summary}
\small
\renewcommand{\arraystretch}{1.5}
\begin{tabular}{@{}p{1cm}lp{13.7cm}@{}}
\toprule
\textbf{Ref.} & \multicolumn{2}{l}{\textbf{Details}} \\
\midrule
\cite{MR906214} & \textit{Assumptions:} & $u_0 \in H^{5+\delta}_{\rm D}(\Omega):=\{\varphi \in H^{5+\delta}\cap H^4(\Omega)\cap H^2_0(\Omega):\Delta^2 \varphi=0 \text{ on }\partial \Omega\}$,\\
&&$u_1 \in H^{3+\delta}(\Omega)\cap H^2_0(\Omega)$ for $0<\delta<\frac{1}{2}$; $f = 0$ \\[0.3em]
& \textit{Result:} & {Local} unique solution on maximal interval $I^* = [T_-, T_+)$ containing zero\\[0.3em]
& \textit{Regularity:} & $u \in C^0(I^*, H^{5+\delta}_{\rm D}(\Omega)) \cap C^1(I^*, H^{3+\delta}(\Omega)\cap H^2_0(\Omega)) \cap C^2(I^*, H^{1+\delta}(\Omega)\cap H^1_0(\Omega))$ (Classical solution)\\
\midrule
\cite{MR1233041} & \textit{Assumptions:} & \textbf{Case 1:} $u_0 \in H^{4}(\Omega)\cap H^2_0(\Omega)$, $u_1 \in H^2_0(\Omega)$, $f=0$ \\
& & \textbf{Case 2:} $u_0 \in H^{5+\delta}_{\rm D}(\Omega)$, $u_1 \in H^{3+\delta}(\Omega)\cap H^2_0(\Omega)$  for $0 < \delta < \frac{1}{2}$, $f=0$\\[0.3em]
& \textit{Result:} & {Global existence and uniqueness} for both cases \\[0.3em]
& \textit{Regularity:} & \textbf{Case 1:} $u \in C^0(\mathbb{R}, H^{4}(\Omega)\cap H^2_0(\Omega)) \cap C^1(\mathbb{R}, H^{2}(\Omega)) \cap C^2(\mathbb{R}, L^2(\Omega))$ \\
& & \textbf{Case 2:} $u \in C^0(\mathbb{R}, H^{5+\delta}_{\rm D}(\Omega)) \cap C^1(\mathbb{R}, H^{3+\delta}(\Omega)) \cap C^2(\mathbb{R}, H^{1+\delta}(\Omega))$ (Classical solution)\\
\midrule
\cite{MR1621722} & \textit{Assumptions:} & $u_0 \in H^2_0(\Omega)$, $u_1 \in L^2(\Omega)$, $f \in L^2(0,T; L^2(\Omega))$ \\[0.3em]
& \textit{Result:} & {Existence and uniqueness} of weak solution on $[0,T]$ \\[0.3em]
& \textit{Regularity:} & $u \in L^\infty(0,T; H^2_0(\Omega))$, $ u_t \in L^\infty(0,T; L^2(\Omega))$, $u_{tt} \in L^\infty(0,T; H^{-2}(\Omega))$, $v \in L^\infty(0,T; H^2_0(\Omega))$ \\
\bottomrule
\end{tabular}
\end{table}
Consider  $\eqref{P1 weak_form2}$ at $t=0$ and test it against $v_0$. Then, from \eqref{b-bdd-c} we   obtain the inequality $|v_0|_{H^2(\Omega)} \le C_{\rm S}  \trinl u_0\trinl^2$. This, 
Cauchy-Schwarz and Young's inequalities reveal $2(f,u_t)_{L^2(\Omega)} \le 2\|f\|\|u_t\| \le 
{T}\|f\|^2+\|u_t\|^2/T$ and lead to 
\begin{align*}
\|u_t\|^2+|u|^2_{H^2(\Omega)}+|v|_{H^2(\Omega)}^2&\le \trinl u_0\trinl^2(1+C_{\rm S} ^2 \trinl u_0\trinl^2)+\|u_1\|^2+T\|f\|^2_{L^2(0,T;L^2(\Omega))} +\frac{1}{T}\int_0^t\|u_t\|^2\dt\\
&=:{\rm M}'_{1}(u_0,u_1,f)+\frac{1}{T}\int_0^t\|u_t\|^2\dt
\end{align*}
 {with $\int_0^t\|f\|^2\dt \le \|f\|^2_{L^2(0,T;L^2(\Omega))}$ in the first inequality. Gronwall's inequality \cite[Lemma 2.1]{MR1638130}  leads to stability of the continuous solution
\begin{align}    \|u_t\|^2+|u|^2_{H^2(\Omega)}+|v|_{H^2(\Omega)}^2&\le
\left(1+  \frac{1}{T} \int_0^t \exp\big({\frac{t-s}{T}}\big) \ds\right) {\rm M}'_{1}(u_0,u_1,f)
 \le e {\rm M}'_{1}(u_0,u_1,f)=:\text{\v{M}}_{1}(u_0,u_1,f)
\label{energy-c2}
\end{align}
with the Euler number $ e=2.71$.
\subsection{Preliminaries}\label{subsec Preliminaries}
 For a triangle $K \in {\cal T}$ 
with  diameter $h_K$, area $|K|$, and the outward unit normal  $n_K$   along $\partial K$,  let $h=h_{\max}:= \underset {K \in \mathcal{T}}{{\max }}\;h_K $. Let ${\mathcal V}(\Omega)$ (resp. ${\mathcal V}(\partial \Omega))$ denote the set of all interior (resp. boundary) vertices of  ${\cal T}$ and let ${\mathcal V}= {\mathcal V}(\Omega) \cup {\mathcal V}(\partial \Omega)$.  Let ${\mathcal E}(\Omega)$ (resp. ${\mathcal E}(\partial \Omega)$) denote the set of all interior (resp. boundary) edges of  ${\cal T}$ and let ${\mathcal E}= {\mathcal E}(\Omega) \cup {\mathcal E}(\partial \Omega)$. 
 The nonconforming Morley FE space ${{\rm M}({\cal T})}$ \cite{CC} reads (see Figure~\ref{fig:mor})
 
\begin{minipage}{0.65\linewidth}
\begin{align*}
 {\rm M}'({\cal T}) &:= \big\{ \varphi_{\rm M} \in P_2({\cal T}) : \varphi_{\rm M} \text{ is continuous at interior  vertices}\\
&\qquad\qquad\qquad\qquad\text{  and its normal derivatives are continuous}\\
&\qquad\qquad\qquad\qquad\text{  at the midpoints of interior edges} \big\},\\[4pt]
{\rm M}({\cal T}) &:= \big\{ \varphi_{\rm M} \in {\rm M}'({\cal T}) : \varphi_{\rm M} \text{ vanishes at the vertices of } \partial\Omega\\
&\qquad\qquad\qquad\qquad\text{ and its normal derivatives vanish}\\
&\qquad\qquad\qquad\qquad\text{ at the midpoints of boundary edges of } \partial\Omega \big\}.
\end{align*}
\end{minipage}%
\hfill%
\begin{minipage}{0.29\linewidth}
\centering
\includegraphics[width=0.75\linewidth]{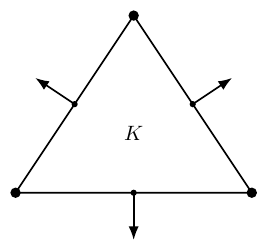}
\captionof{figure}{Morley finite element.}
\label{fig:mor}
\end{minipage}

\noindent
\medskip
Define a discrete bilinear form  $a_{\rm pw}(\bullet, \bullet):(H^2_0(\Omega)+{\rm M}({\cal T})) \times (H^2_0(\Omega)+ {\rm  M}({\cal T})) \rightarrow \mathbb{R}$  by 
$$a_{{\rm pw}}(\varphi,\psi):= \int_\Omega D^2_{{\rm pw}}\varphi:D^2_{\rm{pw}}\psi \dx \text{ for all } \varphi, \psi \in H^2_0(\Omega)+ {\rm M}({\cal T}).$$
This is a semi-scalar product and induces a semi-norm on $H^2(\mathcal{T})$ given by 
\begin{equation}
\trinl\varphi\trinl_{\pw}^2:=\sum_{K\in \mathcal{T}}|\varphi|_{H^2(K)}^2=a_{{\rm pw}}(\varphi,\varphi) \quad \text{ for all } \varphi \in H^2(\mathcal{T}). \label{P1 mesh_norm} 
\end{equation}
It is known that $a_{\rm pw}(\bullet, \bullet)$ is a scalar product  on $ H^2_0(\Omega)+ {\rm M}({\cal T})$ \cite{MR3407244,MR4235819} {and $\big(H^2_0(\Omega)+ {\rm M}({\cal T}),a_{\rm pw}(\bullet, \bullet)\big)$ is a Hilbert space  with the  induced norm $\trinl \bullet \trinl_\pw$}. \\
\noindent 
 The piecewise trilinear form  $b_{\pw}(\bullet,\bullet,\bullet) : (H^2_0(\Omega)+{\rm M}({\cal T})) \times (H^2_0(\Omega)+ {\rm  M}({\cal T})) \times (H^2_0(\Omega)+ {\rm  M}({\cal T})) \rightarrow \mathbb{R}$ is defined by
\begin{equation*}
     b_{\pw}(\varphi,\chi,\psi):=-\sum_{K \in {\mathcal{T}}}\int_K [\varphi, \chi ] \psi \dx\,{=:}- \int_\Omega [\varphi, \chi ]_{\rm{pw}} \psi \dx\quad \text{ for all  } \varphi,\chi,\psi \in H^2_0(\Omega)+ {\rm  M}({\cal T}).
\end{equation*}
The trilinear form $b_\pw$ is  merely symmetric with respect to the first two arguments, that is,   $b_{\pw}(\varphi,\chi,\psi) =b_{\rm pw}(\chi, \varphi, \psi)$.
The definition of $b_{\pw}(\bullet,\bullet,\bullet)$ from the last displayed expression, a H\"older's inequality, and a discrete Sobolev embedding 
\begin{align} \label{Sobolev}
\|\psi\|_{L^\infty(\Omega)} &\le C_{\rm dS} \trinl\psi\trinl_{\pw}  
\end{align} (cf. \cite[Lemma~4.7]{MR4205055} and \cite[Lemma~2.5]{MR4230429} for proofs) reveal
    \begin{equation}
    |b_{\pw}(\varphi,\chi,\psi)|\le C_{\rm dS}\trinl\varphi\trinl_{\pw}\trinl \chi\trinl_{\pw}\trinl\psi\trinl_{\pw}.\label{Bh-bdd}
    \end{equation}
    The time discretization on a uniform partition $ 0=t_0 < t_1<t_2< \cdots<t_N=T$ of the interval $[0,T]$ into $N \in\mathbb{N}$  
congruent intervals 
 with $t_n=nk$, focuses on a uniform  time step $k=T/N$. 
The following abbreviations apply  to any function $\varphi(\bx,t)$:
\begin{subequations}
\begin{align}
& \varphi^n  := \varphi(\bx,t_n)= \varphi(t_n), \quad \varphi^{n+1/2}:= \frac{1}{2}\left(\varphi^{n+1}+\varphi^n \right),
\quad \varphi^{n,1/4} :=\frac{1}{4} \left(\varphi^{n+1}+2\varphi^n+  \varphi^{n-1}\right),\label{varphi1}\\
& \bar{\partial}_t \varphi^{n+1/2} :=\frac{\varphi^{n+1}-\varphi^{n}}{k} ,\quad 
{\bar{\partial}}^2_t \varphi^n := \frac{\varphi^{n+1}-2\varphi^n+\varphi^{n-1}}{k^2} ,\quad 
\delta_t \varphi^n := \frac{\varphi^{n+1}-\varphi^{n-1}}{2k}. \label{varphi2}
\end{align}\end{subequations}
The identities below result from \eqref{varphi1}-\eqref{varphi2} and are straightforward.
\begin{subequations}
 \begin{align}
   &k  \bar{\partial}_t ^2 \varphi^n= \bar{\partial}_t \varphi^{n+1/2}-\bar{\partial}_t \varphi^{n-1/2},\;
    2 \varphi^{n,1/4} =\varphi^{n+1/2}+\varphi^{n-1/2},\label{un14}\\
     &\varphi^{n+1}-\varphi^{n-1}=2(\varphi^{n+1/2}-\varphi^{n-1/2})=k(\bar{\partial}_t \varphi^{n+1/2}+\bar{\partial}_t \varphi^{n-1/2})=2k\delta_t\varphi^n.\label{dtu}
 \end{align}\end{subequations}
\subsection{Motivation}\label{motivation} 
The algorithmic energy momentum method by Simo and Tarnow's \cite{MR1187632,MR1284833}  first established that midpoint or trapezoidal time discretizations of nonlinear elastodynamics fail to conserve energy exactly and could be numerically unstable in a related but different application to shells. They introduced a stress evaluation built from a discrete gradient that restores exact conservation of energy and momentum. Bilbao et al.  \cite{MR2371355,MR3403715} later adapted this discrete gradient framework to the  von K\'arm\'an bracket and derived closed form energy conserving finite difference schemes that are {\it stable  conditionally}, later generalized to the full system with in plane inertia. 
While \cite{bilbao2023explicit} establish unconditional stability in general, their treatment of the plate case falls back to a conditionally stable variant; to the best of our knowledge {\it no} fully-discrete scheme with rigorously proven unconditional stability is available for the nonlinear, coupled dynamic von Kármán system on a nonconforming spatial discretization with particular attention to the choice for the initial approximations.


Suppose that $U^n$ denotes the approximation in a semi-discrete problem at a time $t_n= nk$ and consider $U(t)$ as a piecewise linear spline in a uniform time discretization with time-step size $k>0$. 
Let  $U^{n+1/2} := (U^{n+1}+U^n)/2$ be  the approximation at the midpoint $t_{n+1/2} = (n+1/2)k$ and $\bar{\partial}_t{U}^{n+1/2} := (U^{n+1}-U^n)/k$
be the approximation of the velocity in the time interval $(t_{n-1},t_n)$. The discretization of the time derivative $u_{tt}$ at time $t_n$ by a central difference quotient $k^{-2}(U^{n+1}-2U^n+U^{n-1}) =k^{-1}(U^{n+1/2}-U^{n-1/2})$ is the most natural choice and thereafter we need an approximation of $\Delta^2 u$ at time $t_n$ by an average of values of $U^{n-1}, U^n, U^{n+1} $ to allow energy conservation. It is known from \cite{MR4971619} that $(U^{n+1/2}+U^{n-1/2})/2$ allows for the following arguments. The resulting discrete problem for the wave equation and $f \equiv 0$ reads
\[ k^{-1} (\bar{\partial}_t{U}^{n+1/2}-\bar{\partial}_t{U}^{n-1/2}, \varphi)_{L^2(\Omega)}
+\frac{1}{2}a_\pw(U^{n+1/2}+U^{n-1/2}, \varphi)=0.
\]
The test with $\varphi= k(\bar{\partial}_t{U}^{n+1/2}+\bar{\partial}_t{U}^{n-1/2})=U^{n+1}-U^{n-1} =2(U^{n+1/2} -U^{n-1/2})$ leads to telescoping terms 
\[ \| \bar{\partial}_t{U}^{n+1/2} \|^2 - \|\bar{\partial}_t{U}^{n-1/2}\|^2 +  \trinl {U}^{n+1/2}\trinr_\pw^2 -  \trinl {U}^{n-1/2}\trinr_\pw^2=0 \]
that (up to initial values for $U^0,U^1,\bar{\partial}_t{U}^{1/2}$) provides energy conservation (for the wave equation) viz.
\[E'_{n+1/2}:=  \| \bar{\partial}_t{U}^{n+1/2} \|^2 +   \trinl {U}^{n+1/2}\trinr_\pw^2=E'_{n-1/2} \text{ for } n \in {\mathbb N}.
\]

\medskip \noindent Adopt the notation $ V^n, V^{n+1/2}, V^{n+1}$ for piecewise linear approximations for the variable $V(t)$ that approximates $v$
in \eqref{stron-2}. This preliminary consideration eventually reveals a technique for energy conservation and its repeated application favours the subsequent scheme 
\[ k^{-1} (\bar{\partial}_t{U}^{n+1/2}-\bar{\partial}_t{U}^{n-1/2}, \varphi)_{L^2(\Omega)}
+\frac{1}{2}a_\pw(U^{n+1/2}+U^{n-1/2}, \varphi) +\frac{1}{4}b_{\pw}(U^{n+1/2}+U^{n-1/2}, \varphi, V^{n+1/2}+V^{n-1/2})=0.
\]
The test with $\varphi= k(\bar{\partial}_t{U}^{n+1/2}+\bar{\partial}_t{U}^{n-1/2})=U^{n+1}-U^{n-1} =2(U^{n+1/2} -U^{n-1/2})$ and the above calculations lead to 
\[
2E_{n-1/2}'-2 E_{n+1/2}'=b_\pw(U^{n+1/2}, U^{n+1/2},V^{n+1/2}+V^{n-1/2} ) -
b_\pw(U^{n-1/2}, U^{n-1/2},V^{n+1/2}+V^{n-1/2} )
\]
provided the trilinear form $b_\pw$ is symmetric in the first two components. The second discrete equation may be evaluated in an averaged sense with 
\[ b_\pw(U^{n+1/2}, U^{n+1/2},\psi) =2a_\pw(V^{n+1/2},\psi). \]
For the test function $\psi= V^{n+1/2}+V^{n-1/2}$, this leads to 
\[ E_{n-1/2}'- E_{n+1/2}' =
a_\pw(V^{n+1/2}-V^{n-1/2}, 
V^{n+1/2}+V^{n-1/2}) =  
\trinl {V}^{n+1/2}\trinr_\pw^2 
-\trinl {V}^{n-1/2}\trinr_\pw^2\]
and so to energy conservation in the sense of 
\[E_{n+1/2} := E'_{n+1/2} +\trinl{V^{n+1/2}} \trinr^2_{\pw} = E_{n-1/2}.\]
This novel motivation for the semi-discrete scheme is examined for the fully-discrete scheme below in this paper. 
\subsection{Fully-discrete scheme}\label{sub-main-full}
\medskip \noindent
This subsection presents the fully discrete scheme for the von Kármán equations under the assumption that discrete solutions at the first two time steps are available. Initial discretization computes the discrete solutions at the first two time steps in Subsection~\ref{P1 fully_discrete_section} below. We state the main well-posedness and error estimate results for the scheme here, while the detailed proofs follow in Section~\ref{sect-fully-full}.

\medskip
\noindent
Let $(u_{\m}^n,v_{\m}^n)$ denote the Morley approximation of $(u(t),v(t))$ at time $t_n$ for $n=0,1,\ldots,N$. Assume that $(u_{\m}^0,v_{\m}^0)$ and $(u_{\m}^1,v_{\m}^1) \in {{\bf M}({\cal T})}:=\m(\cal{T}) \times \m(\cal{T})$ be approximations of the continuous solution $(u(t),v(t))$ at times $t_0$ and $t_1$, respectively. For $n=1,2,\ldots, N-1$, the fully discrete scheme motivated in Subsection~\ref{motivation} seeks the solution $(u_{\m}^{n+1}, v_{\m}^{n+1})\in {{\bf M}({\cal T})}$ to
\begin{subequations}\label{P4fully}
\begin{align}
    (\bar{\partial}_t ^2 u_{\m}^n,\varphi_{\m})_{L^2(\Omega)} + a_{\pw} (u_{\m}^{n,1/4},\varphi_{\m} )+b_{\pw}(u_{\m}^{n,1/4},\varphi_{\m},v_{\m}^{n,1/4})&= (f^{n,1/4},\varphi_{\m} )_{L^2(\Omega)}\quad\text{ for all } \varphi_{\m} \in  \m(\cal{T}),\label{P4 fully_discrete1}\\
 a_{\pw}(v_{\m}^{n+1/2},\psi_{\m}) -\frac{1}{2}{b_{\pw}(u_{\m}^{n+1/2},u_{\m}^{n+1/2},\psi_{\m})}&=0 \quad\text{ for all }\psi_{\m}\in \m(\cal{T}) .\label{P4 fully_discrete2}
     \end{align}
     \end{subequations}
    With reference to \eqref{P4 fully_discrete1}, a Newmark scheme  
for $u_{tt}$ with $t>t_1$ and  
the approximation $b_{\pw}(u_{\m}^{n,1/4},\varphi,v_{\m}^{n,1/4})$ for the non-linear term imply a 
quadratic convergence in time .
To see this, we look back at the key identity 
\eqref{energy-c0} and Subsection~\ref{motivation} for the nonlinear term that ensures the \textit{energy 
conservation} of \eqref{P1 weak_form}.
For its discrete counterpart, in the proof of stability, the test function 
is chosen as $2k\delta_t u_{\m}^{n}$ in \eqref{P4 fully_discrete1} and the linear terms telescope 
perfectly as $\norm{\bar{\partial}_t u_{\m}^{n+1/2}}^2- \norm{\bar{\partial}_t u_{\m}^{n-1/2}}^2+ \trinl u_{\m}^{n+1/2} \trinl_{\pw}^2- \trinl u_{\m}^{n-1/2} \trinl_{\pw}^2$. From the design of  \eqref{P4 fully_discrete2},  the non-linear term  reduces to
\begin{align}
   b_{\pw}(u_{\m}^{n,1/4}, 2k\delta_tu_{\m}^{n},v_{\m}^{n,1/4})
   &=b_{\pw}(u_{\m}^{n+1/2}, u_{\m}^{n+1/2},v_{\m}^{n,1/4})
   -b_{\pw}(u_{\m}^{n-1/2}, u_{\m}^{n-1/2},v_{\m}^{n,1/4})\nonumber\\
   &=\trinl v_{\m}^{n+1/2}\trinl_{\pw}^2
   -\trinl v_{\m}^{n-1/2}\trinl_{\pw}^2.\label{motiv-final}
\end{align}
{The second equality in \eqref{motiv-final} holds precisely 
because $v_{\m}^{n+1/2}$ satisfies \eqref{P4 fully_discrete2}: 
\eqref{P4 fully_discrete2} is therefore not an independent approximation 
of a separate Partial Differential Equation (PDE), but the exact discrete condition on $v$ forced by the 
requirement that the nonlinear energy balance closes. The sum of 
\eqref{motiv-final} over time steps then leads to a global discrete energy 
inequality with no restriction on $\Delta t$.}
Beyond accuracy, the chosen weighted time-averaging strategy is motivated 
by stability: The discrete energy balance for the nonlinear von 
K\'{a}rm\'{a}n system ensures unconditional stability, unlike classical 
implicit Newmark schemes which are only conditionally stable in the 
nonlinear setting (see, e.g., \cite[eqn. (36)]{akay1980dynamic}). 
Theorem~\ref{TH-WELLPOSED-FINAL}$(a)$ provides the details of 
{\it unconditional stability of this discrete scheme}. {The design is further guided by the requirement that the 
same quarter-average structure must simultaneously deliver a discrete energy 
law (see Subsection~\ref{motivation} for details.)}
\subsection{Initial discretization} \label{P1 fully_discrete_section}
We propose approximations for the solution $(u,v)$ at the initial time discretization points $t_0$ and $t_1$. This is crucial for the design of the fully discrete formulation in Subsection 2.3, since the multi-step scheme \eqref{P4fully} requires not only the solution at these two initial steps, but its stability and error estimates also depend critically on these approximations. The results on well-posedness and error control are stated in Theorem~\ref{TH-WELLPOSED-INI} and Theorem~\ref{LEM-ERROR-INI} respectively, while the proofs follow in Section~\ref{sec-fully-ini}.\\
For $u_0 \in H^2_0(\Omega)$, we propose $(u_{\m}^0,v_{\m}^0)$ as an approximation of $(u(t),v(t))$ at $t=t_0$ defined by
\begin{align}
  & a_{\pw}(u_{\m}^0,\varphi_{\m} )=a( u_0, J\varphi_{\m} ),\;\; 
  a_{\pw}(v_{\m}^0, \psi_{\m}) = \frac{1}{2} b_{\pw}(u_{\m}^0, u_{\m}^0, \psi_{\m}) \; \text{ for all } (\varphi_{{\m}},{\psi}_{{\m}}) \in {{\bf M}({\cal T})}. 
  \label{app_at_0}
\end{align}
\noindent Next, for  
$u_1 \in L^2(\Omega)$, the approximation $(u_{\m}^1,v_{\m}^1)$ of the continuous solution $(u(t),v(t))$ at $t=t_1$ is then derived from
   \begin{subequations}
\label{P1ic}
\begin{align}
{2k^{-1}}(\bar{\partial}_t u_{\m}^{1/2} -u_1, \varphi_{{\m}} )_{L^2(\Omega)}+a_{\pw}(u_{\m}^{1/2},\varphi_{{\m}}) +b_{\pw}(u_{\m}^{1/2},\varphi_{{\m}},v_{\m}^{1/2})&= (f^{1/2},\varphi_{{\m}} )_{L^2(\Omega)}, \label{P1 ic_1}\\
 a_{\pw}(v_{\m}^{1},\psi_{{\m}}) -\frac{1}{2}b_{\pw}(u_{\m}^{1},u_{\m}^{1},\psi_{{\m}})&=0.
 \label{P1 ic_2} 
 \end{align}
\end{subequations}
The key identity \eqref{energy-c0} for the nonlinear term ensures the \textit{energy conservation}  of \eqref{P1 weak_form}. To arrive at discrete counterpart of \eqref{energy-c0} in the proof of stability, the test function will be $\bar{\partial}_t u_{\m}^{1/2}$. For the discrete scheme in \eqref{P1ic}, elementary manipulations lead to 
       \begin{align}
   b_{\pw}(u_{\m}^{1/2}, \bar{\partial}_t u_{\m}^{1/2},v_{\m}^{1/2})&=\frac{1}{2k}\big(b_{\pw}(u_{\m}^{1}, u_{\m}^{1},v_{\m}^{1/2})-b_{\pw}(u_{\m}^{0}, u_{\m}^{0},v_{\m}^{1/2})\big)=\frac{1}{2k}\big(\trinl v_{\m}^{1}\trinl_{\pw}^2-\trinl v_{\m}^{0}\trinl_{\pw}^2\big).\label{motiv-ini}
    \end{align} 
The above identity facilitates the proof of stability of the discrete scheme (see Theorem~\ref{TH-WELLPOSED-INI}$(c)$ below, for details).
\subsection{Key results}\label{key} 
In this subsection, we present the main results on stability and error estimates for the fully discrete scheme \eqref{P4fully}-\eqref{P1ic}. The proofs are given in Sections \ref{sec-fully-ini}-\ref{sect-fully-full}.\\
 The first theorem provides sufficient conditions under which \eqref{app_at_0} and \eqref{P1ic} are well-posed. The proof is provided in Subsection~\ref{subsec} and utilizes the expression
 $$\text{\v{M}}_{2}(u_0,u_1,f,k):= \|J\|^2\trinl u_0\trinl^2\big(1+ \frac{1}{4}C_{\rm dS}^2\|J\|^2\trinl u_0\trinl^2\big)+4\norm{u_1}^2+{k^2}\norm{f}^2_{C([0,t_1];L^2(\Omega))}$$
 with the operator norm  $\|J\|:=\|J\|_{\mathcal{L}(M(\mathcal{T}), H^2_0(\Omega))}$ for a bounded linear companion operator of Subsection~\ref{subsec-comp} below. Recall $C_{\rm dS}$ from \eqref{Sobolev}.
\begin{thm}[well-posedness]\label{TH-WELLPOSED-INI}
    Suppose that the initial data satisfy 
    $u_0 \in H^2_0(\Omega)$,  $u_1 \in L^2(\Omega)$, and the source $f \in C([0,t_1];L^2(\Omega))$. Then $(a)$-$(d)$ hold viz. 
    \begin{itemize}
    \item[($a$)] \eqref{app_at_0}
is well-posed with  $\trinl u_\m^0 \trinr_\pw  \le \|J\| \trinl u_0\trinl$ and $ \trinl v_\m^0 \trinr_\pw  \le  \frac{1}{2} C_{\rm dS}\|J\|^2\trinl u_0\trinl^2 $,
    \item[($b$)] there exists a solution $(u_{\m}^{1},v_{\m}^{1})$  to \eqref{P1ic}, 
    \item[($c$)] $(u_{\m}^1,v_{\m}^1)$ satisfies $\trinl u_{\m}^1\trinl_{\pw}+\trinl v_{\m}^1\trinl_{\pw}  \le \sqrt{2}\big({\rm{\text{\v{\rm M}}}}_{2}(u_0,u_1,f,k)\big)^{1/2}$, 
\item[($d$)] the solution $(u_{\m}^1,v_{\m}^1)$ to \eqref{P1ic} is unique under smallness  assumptions on $u_0$, $u_1$, and the temporal discretization parameter $k$,  precisely ${{\rm M}}_{2}(u_0,u_1,f,k):=C_{\rm dS}^2 \big( \|J\|^2\trinl u_0\trinl^2+{\rm{\text{\v{\rm M}}}}_{2}(u_0,u_1,f,k)\big) \le 1$.
  \end{itemize}
\end{thm}
\begin{rem}[well-posedness of \eqref{app_at_0}]
 Theorem~\ref{TH-WELLPOSED-INI}$(a)$ 
  does not require smallness data assumptions. 
  This is because \eqref{app_at_0} is linear and its equations decouple.  Hence the unique solvability of its first equation follows from the Lax--Milgram lemma, and the solution is used to uniquely solve the second equation.
\end{rem}
\noindent The next result  provides the discretization  error bounds for  \eqref{app_at_0}-\eqref{P1ic}. The proof  is postponed to  Subsection~\ref{sub-inierror}.
\begin{thm}[initial error bounds]\label{LEM-ERROR-INI}
    Suppose that the solution $u$ to \eqref{P1 weak_form} belongs to $ H^2
       (0,t_1;H^2_0(\Omega) \cap H^{2+\sigma}(\Omega))$ 
       and $u_{ttt}  \in L^\infty(0,t_1;L^2(\Omega))$. Let $\sigma \in (1/2,1]$ be an index of elliptic regularity 
for the biharmonic problem \cite{MR595625}. Then
       \begin{align*}
        \norm{\bar{\partial}_t(u^{1/2}-u_{\m}^{1/2})}+ \trinl u^{1/2}-u_{\m}^{1/2}\trinl_{\pw}+ \trinl v^{1/2}-v_{\m}^{1/2} \trinl_{\pw} = \mathcal{O}(h^{\sigma}+k^2).
       \end{align*}      
\end{thm}
\noindent The next theorem provides sufficient conditions under which the scheme \eqref{P4fully} is well-posed. This  result utilizes the  expression
\begin{align*}
&{\rm{\text{\v{\rm M}}}}_{3}(u_0,u_1,f,k)
:=e^2 \big({\|J\|^2}\trinl u_0\trinl^2\big(1+\frac{1}{4}C_{\rm dS}^2\|J\|^2\trinl u_0\trinl^2\big)+2 \norm{u_1}^2+ \frac{k^2}{2}\norm{f}^2_{C([0,t_1];L^2(\Omega))}+ T{k} \sum_{n=1}^{m}\norm{f^{n,1/4}}^2\big)\end{align*}
and the proof is presented in Subsection~\ref{subsec-wellfull}. 
\begin{thm}[well-posedness]\label{TH-WELLPOSED-FINAL}
    Suppose that the initial data satisfy the smoothness assumptions  $u_0 \in H^2_0(\Omega),$ $u_1 \in L^2(\Omega)$, and the source $f \in C([0,T];L^2(\Omega))$. Then      
    \begin{itemize}
      \item[($a$)] the fully discrete scheme  \eqref{P4fully} is unconditionally  stable, in the sense that if $(u_{\m}^{m+1},v_{\m}^{m+1})$ solves \eqref{P4fully} for any $1 \le m \le N-1$, then 
$\trinl u_{\m}^{m+1/2}\trinl_{\pw}+\trinl v_{\m}^{m+1/2}\trinl_{\pw}\le \sqrt{2}\bigl({\rm{\text{\v{\rm M}}}}_{3}(u_0,u_1,f,k)\bigr)^{1/2},$
\item[($b$)] for  any $n=1,2,\cdots$,$N-1$, there exists a solution $(u_{\m}^{n+1},v_{\m}^{n+1})$  to \eqref{P4fully},  

\item[($c$)] If  ${\rm{\text{{\rm M}}}}_{3}(u_0,u_1,f,k):=C_{\rm dS}^2{\rm{\text{\v{\rm M}}}}_{3}(u_0,u_1,f,k) <1$ then the  solution $(u_{\m}^{n+1},v_{\m}^{n+1})$ to  \eqref{P4fully}  is unique.
\end{itemize}
\end{thm}
\noindent Next, we present a theorem that provides the error bounds for the discretization \eqref{P4fully}. It utilizes the expression ${\rm M}_{1}(u_0,u_1,f):=\|J\|C_{\rm dS }\big(\text{\v{M}}_{1}(u_0,u_1,f)\big)^{1/2} $ with  $\text{\v{M}}_{1}(u_0,u_1,f)$ from \eqref{energy-c2}. The proof is provided in Subsection~\ref{apriori-sec}.
\begin{thm}[energy error bounds]\label{TH-ERROR-FINAL}
       Let  $u_0 \in H^2_0(\Omega), u_1\in L^2(\Omega)$, and $f \in H^1(0,T;L^2(\Omega))$ be  such that ${\rm M}_{1}(u_0,u_1,f) $ $\le 1/2$. Suppose that  $ u,u_t,v,v_t  \in H^1(0,T;H^2_0(\Omega)\cap H^{2+\sigma}(\Omega))$ and   $u  \in H^4(0,T;L^2(\Omega))$. Then, for any $ 1 \le m \le n$ with $1 \le n \le N-1$, we have 
       \begin{align*}
        \norm{\bar{\partial}_t (u^{m+1/2}-u_{\m}^{m+1/2})}+    \trinl u^{m+1/2}-u_{\m}^{m+1/2}\trinl_{\pw}+\trinl v^{m+1/2}-v_{\m}^{m+1/2}\trinl_{\pw} = \mathcal{O}(h^\sigma+k^2).
       \end{align*}
\end{thm}
\section{Auxiliary results}\label{Sec-aux}
\noindent The fourth-order terms and the nonlinearity in the coupled system create several technical difficulties, and these are further enhanced by the complexities of second-order time discretization. In this section, we present some technical results that prepare us for the proof of the main results presented later on.  

\medskip \noindent 
\subsection{Companion operator} \label{subsec-comp}
In contrast to other averaging operators \cite{MR2670114}, the conforming companion operator (see, e.g.,  \cite{carsput2020,CC,MR3407244}) preserves the integral mean of the function and its Hessian, and it also satisfies an enhanced best-approximation property. In this article, companion operator serves two main purposes:
(a) it allows a  nonstandard Ritz projection in Subsection~\ref{sub-ritz}, below. This projection possesses optimal approximation features and plays a central role in the error analysis,  
(b) during the derivation of the error bounds, the test functions in the continuous weak formulation are chosen from the range of this operator. This produces cleaner estimates and removes the need for boundary-term computations that arise 
when the PDE is tested with discrete functions. The next lemma gives the construction of this operator and utilizes  the $C^1$-conforming Hsieh--Clough--Tocher $({\rm HCT}(\mathcal{T}))$ element \cite{MR520174}.
\begin{lemma}[companion operator --{\cite[Lemma 3.7]{CC} and \cite[Lemma 5.1]{carsput2020}}]\label{bhcompanion_lem} 
There exists a linear mapping $J: {\rm M}(\mathcal{T})\to ({\rm HCT}(\mathcal{T})+P_8(\mathcal{T})) \cap H^2_0(\Omega)$ such that any $\varphi_{\rm M}\in {\rm M}(\mathcal{T})$ satisfies
\begin{align*}
&\text{($a$) } J\varphi_{\rm M}(z)=\varphi_{\rm M}(z) \quad \text{ for } z\in\mathcal{V}, \; 
\text{($b$) } \nabla ({J}\varphi_{\rm M})(z)=
|\mathcal{T}(z)|^{-1}\sum_{K\in\mathcal{T}(z)}(\nabla \varphi_{\rm M}|_K)(z)
\quad \text{ for }z\in\mathcal{V}(\Omega),  \nonumber\\
&\text{($c$) } \fint_e \frac{\partial J \varphi_{\rm M}}{{\partial {\mathbf{n}}}} \ds=\fint_e \frac{ \partial \varphi_{\rm M}}{{\partial {\mathbf{n}}}} \ds \text{ for any } e\in\mathcal{E},\;\text{($d$) }  \trinl \varphi_{\rm M}- J \varphi_{\rm M} \trinr_{\rm pw} \le \const{jpw}  \min_{\varphi \in  H^2_0(\Omega)}  \trinl \varphi_{\rm M}-\varphi \trinr_{\rm pw}, \nonumber \\
&\text{($e$) } \|{\varphi_{\m} -J\varphi_{\m}}\|_{H^s({\cal T})} \le  \const{ca} h^{2-s}\min_{\varphi\in  H^2_0(\Omega)} \trinl \varphi- \varphi_{\m} \trinl_{\pw}\quad \constref{ca}>0, \text{ and }0 \le s \le 2,\\
&\text{($f$) } \Pi^0\bigl(\varphi_{\m} - J\varphi_{\m}\bigr) = 0
\quad\text{and}\quad
\Pi^0\!\bigr(D_{\mathrm{\pw}}^{2}\bigl(\varphi_{\m} - J\varphi_{\m}\bigr)\bigr) = 0.
\end{align*} 
\end{lemma} 

\subsection{Modified Ritz projection}\label{sub-ritz}
The {\it modified} Ritz projection $\mathcal{R}_{\m}: H^2_0(\Omega)\rightarrow {\rm M}({\cal T}) $   is defined as:
\begin{equation}
    a_{\pw}({\cal R}_{\m}\varphi,\varphi_{\m} )=a(\varphi, J\varphi_{\m} )\qquad  \text{ for all } \varphi_{\m} \in {\rm M}({\cal T}),\, \varphi \in H^2_0(\Omega)\label{P1 ritz_projection}.
\end{equation}
   The choice $\varphi_{\m}={\cal R}_{\m}\varphi$ in \eqref{P1 ritz_projection}, together with the continuity of $a(\bullet,\bullet)$ and $J$, imply the \textit{stability of $\mathcal{R}_{\m}$} 
\begin{align}
    \trinl {\cal R}_{\m}\varphi \trinl_{\pw}\le \|J\|   \trinl  \varphi  \trinl \qquad \text{ for all } \varphi \in H_0^2(\Omega).\label{r-con}
\end{align}
The next lemma states the approximation properties of \eqref{P1 ritz_projection}. For a proof, see   \cite[Appendix A.5–A.6]{MR4848002}.
\begin{lemma}[approximation properties for ${\cal R}_{\m}$ \cite{MR4848002,MR4235819}] %
\label{P1 ritz_lemma}
Let $\varphi \in H^2_0(\Omega) \cap H^{2+\sigma}(\Omega)$, and let $\mathcal{R}_{\m} \varphi$ be its Ritz projection defined in \eqref{P1 ritz_projection}. Then, there exists a constant $\const{car} > 0$ such that
\begin{equation}
    \|\varphi-\mathcal{R}_{\m}\varphi\| 
    +  h^{\sigma}\trinl\varphi-\mathcal{R}_{\m}\varphi\trinl_{\pw} \le \constref{car}h^{2\sigma}\norm{\varphi}_{H^{2+\sigma }(\Omega)}.  \label{P1 norm_ritz}
\end{equation}
\end{lemma}

\subsection{Trilinear forms}
\begin{lemma}[properties of   trilinear forms]\label{IB}
For all  $\psi, \xi \in H^2_0(\Omega)$, $\varphi,\chi \in H^2_0(\Omega) + {\m}({\cal T})$, $\zeta \in H^2_0(\Omega)\cap H^{2+\sigma}(\Omega)$, and $\chi_{\m},\psi_{\m} \in {\m}({\cal T})$, there exist positive constants
$\const{ccj},\;  \const{cdsj}, \;\const{clamd}$, $\const{C3.4a}, $ and $\const{C3.4b}$
such that 
\begin{itemize}
\item [($a$)]  $|b(\xi,\psi,J\chi_{\m})|$, $|b(\xi,J\chi_{\m} ,\psi)|$  $\le \constref{ccj}\trinl \xi\trinl  \trinl\psi\trinl\trinl\chi_{\m}\trinl_{\pw},$
\item[($b$)] $|b_{\pw}(\varphi,J\chi_{\m},\psi_{\m} )|$, $|b_{\pw}(\varphi,\chi_{\m}, J\psi_{\m} )| \le $  $\constref{cdsj}\trinl\varphi \trinl_{\pw}\trinl\chi_{\m}\trinl_{\pw}\trinl\psi_{\m}\trinl_{\pw} $,
\item[($c$)] $  |b_{\pw}(\varphi,\chi,(J-I)\psi_{\m})|$, $ |b_{\pw}(\varphi, (J-I) \psi_{\m},\chi)| \le  $ 
$\constref{clamd} \trinl \varphi\trinr_{\pw} \trinl\chi\trinl_{\pw}\trinl\psi_{\m}\trinl_{\pw},$
\item[($d$)] $|b_{\pw}(\psi,\varphi, (J-I)\chi_{\m}) \le  \constref{C3.4a} h\trinl \psi\trinl\trinl\varphi\trinl _{\pw}\trinl \chi_{\m}\trinl_{\pw}$, \text{ and}
    \item[($e$)]  
 $|b_{\pw}(\zeta,(J-I)\chi_{\m}, \xi)| \le  \constref{C3.4b} h^\sigma\norm{\zeta}_{ H^{2+\sigma}(\Omega)} \trinl \chi_{\m}\trinl _{\pw}\trinl {\xi} \trinr$.
 \end{itemize}
\end{lemma}
\begin{proof}
 ($a$) 
An application of \eqref{b-bdd-c} and the continuity of $J$ reveals that 
\begin{align*}
   &|b(\xi,\psi,J\chi_{\m} )|\le C_{\rm S} \trinl \xi \trinl\trinl \psi\trinl\trinl J\chi_{\m}\trinl\le \constref{ccj}\trinl \xi\trinl\trinl \psi\trinl\trinl\chi_{\m}\trinl_{\pw}
\end{align*}
with $\constref{ccj}:=\|J\|C_{\rm S}$. The bound for $b(\varphi,J\chi_{\m},\psi)$ follows analogously.\\
    ($b$) The boundedness of the discrete trilinear form $b_{\pw}(\bullet,\bullet,\bullet)$ from  \eqref{Bh-bdd} and continuity of $J$ lead to
    \begin{align*}
       & |b_{\pw}(\varphi,J\chi_{\m},\psi_{\m} )|\le C_{\rm dS}\trinl\varphi\trinl_{\pw}\trinl  J\chi_{\m} \trinl\trinl\psi_{\m}\trinl_{\pw} \le \constref{cdsj}\trinl\varphi\trinl_{\pw}\trinl  \chi_{\m} \trinl_{\pw}\trinl\psi_{\m}\trinl_{\pw}
    \end{align*}
    with $\constref{cdsj}:=\|J\|C_{\rm dS}.$
     The bound for $b_{\pw}(\varphi,\chi_{\m}, J\psi_{\m} )$ follows analogously.\\
 ($c$) The boundedness of $b_{\pw}(\bullet,\bullet,\bullet)$ from \eqref{Bh-bdd} and  Lemma~\ref{bhcompanion_lem}$(d)$ allow us to show that 
 \begin{align*}
& |b_{\pw}(\varphi,\chi,(J-I)\psi_{\m})| \le C_{\rm dS} \trinl \varphi\trinl_{\pw}\trinl\chi\trinl_{\pw}\trinl(J-I)\psi_{\m}\trinl_{\pw}\le \constref{clamd} \trinl \varphi\trinl_{\pw}\trinl\chi\trinl_{\pw}\trinl\psi_{\m}\trinl_{\pw}
 \end{align*}
 with $\constref{clamd}:=\constref{jpw}C_{\rm dS}.$ The bound for $b_{\pw}(\varphi, (J-I) \psi_{\m},\chi)$ follows analogously.\\
($d$) The definition of $b_{\pw}(\bullet,\bullet, \bullet)$, a Cauchy-Schwarz inequality, and an inverse estimate $h_{{K}}\norm{ (J-I)\chi_{\m}}_{L^\infty (K)} \le C_{\rm inv} \norm{ (J-I)\chi_{\m}}_{L^2 (K)} $ \cite[Lemma 12.1]{MR4242224}  lead to
 \begin{align*}
   | b_{\pw}(\psi, \varphi, (J-I)\chi_{\m})|  &\le \trinl \psi\trinl\trinl \varphi\trinl _{\pw}\norm{ (J-I)\chi_{\m}}_{L^\infty (\Omega)}\\
   &\le C_{\rm inv} \trinl \psi\trinl \trinl\varphi\trinl _{\pw}\norm{h^{-1}_{\mathcal{T}} (J-I)\chi_{\m}}_{L^2(\Omega)} \le  \constref{C3.4a} h\trinl \psi\trinl\trinl\varphi\trinl _{\pw}\trinl \chi_{\m}\trinl _{\pw},
    \end{align*}
    with  Lemma \ref{bhcompanion_lem}$(e)$ in the last step and $\constref{C3.4a}:=\constref{ca}C_{\rm inv}.$\\
    ($e$)
    The definition of $b_{\pw}(\bullet,\bullet, \bullet)$, ($c$), a Cauchy-Schwarz inequality, and  $b_{\pw}(\mathcal{R}_{\m}\zeta,(J-I)\chi_{\m},\Pi_0 \xi)=0$ (Lemma~\ref{bhcompanion_lem}($f$)) show 
    \begin{align*}
        |b_{\pw}(\zeta,(J-I)\chi_{\m}, \xi)|&=|b_{\pw}(\zeta-\mathcal{R}_{\m}\zeta,(J-I)\chi_{\m}, \xi)| +|b_{\pw}(\mathcal{R}_{\m}\zeta,(J-I)\chi_{\m}, \xi)|\\
        & \le \constref{clamd} \trinl {\zeta}-\mathcal{R}_{\m}\zeta\trinl _{\pw}\trinl \chi_{\m}\trinl_{\pw}\trinl\xi\trinl+\trinl \mathcal{R}_{\m}\zeta\trinl_{\pw}\trinl (J-I)\chi_{\m}\trinl_{\pw}\norm{\xi-\Pi_0\xi}_{L^\infty(\Omega)}.
    \end{align*}
We apply \eqref{P1 norm_ritz} to bound the first term on the right-hand side, \eqref{r-con}, Lemma~\ref{bhcompanion_lem}$(d)$, and $\norm{\xi-\Pi_0 \xi}_{L^\infty(\Omega)} \le C_{\Pi_0}h \norm{\xi}_{ H^{2}(\Omega)} $ \cite[Theorem 18.16]{MR4242224}
to bound the second term on the right-hand side above. This eventually gives 
     \begin{align*}
        |b_{\pw}(\zeta,(J-I)\chi_{\m}, \xi)|&\le \constref{car} \constref{clamd}  h^\sigma \norm{\zeta}_{ H^{2+\sigma}(\Omega)}\trinl \chi_{\m}\trinl_{\pw}\trinl\xi\trinl+\|J\|\constref{jpw}C_{\Pi_0}
        h \trinl \zeta\trinl \trinl \chi_{\m}\trinl_ {\pw}\norm{\xi}_{ H^{2}(\Omega)}\\
        & \le \constref{C3.4b} h^\sigma\norm{\zeta}_{ H^{2+\sigma}(\Omega)} \trinl \chi_{\m}\trinl _{\pw}\norm{\xi}_{H^2(\Omega)}
        \end{align*}
    with $\trinl\xi\trinl\le\norm{\xi}_{H^2(\Omega)}$, 
    $\trinl\zeta\trinl \le \norm{\zeta}_{ H^{2+\sigma}(\Omega)}$ and  $\constref{C3.4b}:=\constref{car} \constref{clamd} +\|J\|\constref{jpw}C_{\Pi_0}|\Omega|^{\frac{1-\sigma}{2}}$ in the  last step.
\end{proof}
\subsection{Temporal discretization results}
\begin{lemma}\label{axul} Let $(X,\|\bullet\|_{X})$ be a Hilbert space  and $\varphi: [0,T] \rightarrow X$.  If $\varphi \in C([t_{m-1},t_{m+1}];X)$, 
for  some  $1\le m \le N-1$,  then  
\begin{itemize}
    \item 
 [($a$)] $\|\varphi^{m}\|_{X}$, $\|\varphi^{m+1/2}\|_{X}\le$ $\|\varphi\|_{C([t_{m},t_{m+1}];X)}$ and 
  ($b$) $\|\varphi^{m,1/4}\|_{X} \le  \|\varphi\|_{C([t_{m-1},t_{m+1}];X)}$.
   \end{itemize}
    In addition, if $\varphi_t \in  L^\infty(t_{m-1},t_{m+2};X)$ for  some  $1\le m \le N-2$, then 
    \begin{itemize}
    \item 
    [($c$)]$ \|\bar{\partial}_t \varphi^{m+1/2}\|_{X} \le \|\varphi_t\|_{L^\infty(t_{m},t_{m+1};X)}$, \; ($d$) $\|\varphi^{m+1/2}-\varphi^{m-1/2}\|_{X} \le \sqrt{k/2}\|\varphi_t\|_{L^2(t_{m-1},t_{m+1};X)},$
     \item 
   [($e$)]$ \| \delta_t\varphi^{m} \|_X\le \norm{\varphi_t} _{L^\infty(t_{m-1},t_{m+1};X)}, $\; and     
      ($f$) $ \norm{{\varphi^{m+1,1/4}-\varphi^{m,1/4}}}_{X} \le  \sqrt{k/2}\|\varphi_t\|_{L^2(t_{m-1},t_{m+2};X)}$. 
        \end{itemize} 
Furthermore, if $\varphi \in  H^2(0,T;X)$,  then 
   ($g$) $ k \sum_{n=1}^{m} \|\bar{\partial}^2_t\varphi^{n}\|^2_{X} \le \frac{4}{3} \norm{\varphi_{tt}}^2_{L^2(0.T;X)}.$
\end{lemma}
\begin{proof}
    ($a$) From \eqref{varphi1}, we have $\varphi^{m}=\varphi(t_{m})$ and hence 
    \begin{equation}
    \|\varphi^{m}\|_{X} =\|\varphi(t_{m})\|_{X} \le \|\varphi\|_{C([t_{m},t_{m+1}];X)}.\label{evaluation}
    \end{equation}
    The definition of $\varphi^{m+1/2}$ from \eqref{varphi1},  triangle inequality, and arguments as above, imply that 
    \[       \|\varphi^{m+1/2}\|_{X} =\frac{1}{2} \|\varphi^{m}+\varphi^{m+1}\|_{X} \le \frac{1}{2} \big(\|\varphi^{m}\|_{X} +\|\varphi^{m+1}\|_{X} \big) \le  \|\varphi\|_{C([t_{m},t_{m+1}];X)}.
    \]
($b$) The proof proceeds analogous to that of ($a$) by using   definition  \eqref{varphi1} and  triangle inequality.\\
($c$) From the definition \eqref{varphi2}, there holds $\bar{\partial}_t \varphi^{m+1/2}=k^{-1}(\varphi^{m+1}-\varphi^{m})=k^{-1}\int_{t_m}^{t_{m+1}}\varphi_t(t)\dt,$ and hence 
\begin{align*}
   \| \bar{\partial}_t \varphi^{m+1/2}\|_X=k^{-1}\|\int_{t_m}^{t_{m+1}}\varphi_t(t)\dt\|_{X} \le k^{-1}\int_{t_m}^{t_{m+1}}\|\varphi_t(t)\|_{X} \dt\le \|\varphi_t\|_{L^\infty(t_{m},t_{m+1};X)}.
\end{align*}
 ($d$)  Since $\varphi^{m+1/2}-\varphi^{m-1/2}=\frac{1}{2}(\varphi^{m+1}-\varphi^{m-1})$ holds from \eqref{varphi1}, a Cauchy-Schwarz inequality shows
             \begin{align*}
                \|\varphi^{m+1/2}-\varphi^{m-1/2}\|_{X} =\frac{1}{2}\|\int_{t_{m-1}}^{t_{m+1}}\varphi_t \dt\|_{X}\le \frac{1}{2}\int_{t_{m-1}}^{t_{m+1}}\|\varphi_t \|_{X}\dt \le \sqrt{\frac{k}{2}}\|\varphi_t \|_{L^2(t_{m-1},t_{m+1};X)}.
             \end{align*}
\noindent          ($e$) Since $ \delta_t\varphi^{m} =\frac{1}{k}\big(\varphi^{m+1/2}-\varphi^{m-1/2}\big)$, analogous arguments to ($d$) provide
          \begin{align*}
               \|\delta_t\varphi^{m}\|_{X} =  \frac{1}{k}\|\varphi^{m+1/2}-\varphi^{m-1/2}\|_{X} \le \frac{1}{\sqrt{2k}}\|\varphi_t \|_{L^2(t_{m-1},t_{m+1};X)} \le  \|\varphi_t \|_{L^\infty(t_{m-1},t_{m+1};X)}  .
             \end{align*}
          ($f$) A triangle inequality $\|{\varphi^{m+1,1/4}-\varphi^{m,1/4}}\|_{X}=\frac{1}{2}\|{\varphi^{m+3/2}-\varphi^{m+1/2}\|_{X}+\|\varphi^{m+1/2}-\varphi^{m-1/2}}\|_{X}$ and ($d$) lead to
    \begin{align*}
        \norm{{\varphi^{m+1,1/4}-\varphi^{m,1/4}}}_{X} 
        &  \le \sqrt{\frac{k}{8}}\big(\|\varphi_t\|_{L^2(t_{m},t_{m+2};X)}+\|\varphi_t\|_{L^2(t_{m-1},t_{m+1};X)}\big)
         \le \sqrt{\frac{k}{2}}\|\varphi_t\|_{L^2(t_{m-1},t_{m+2};X)}.
    \end{align*}
         ($g$) From Taylor’s theorem (with integral remainder), we have 
         \begin{align*}
             \varphi^{n+1}=\varphi^n+k\varphi_t^n+\int_{t_{n}}^{t_{n+1}}(t_{n+1}-t)\varphi_{tt}\dt \text{ and }\varphi^{n-1}=\varphi^n-k\varphi_t^n+\int_{t_{n-1}}^{t_{n}}(t_{n-1}-t)\varphi_{tt}\dt.
         \end{align*}
         These two identities  and  \eqref{varphi2} allow us to assert that 
         \begin{align*}
k^2\|\bar{\partial}_t^2\varphi^n\|_{X}&=\| \varphi^{n+1}-2 \varphi^{n}+ \varphi^{n-1}\|_{X}
            =\|\int_{t_{n}}^{t_{n+1}}(t_{n+1}-t)\varphi_{tt}\dt+\int_{t_{n-1}}^{t_{n}}(t_{n-1}-t)\varphi_{tt}\dt\|_{X}\\
            &\le\int_{t_{n}}^{t_{n+1}}|t_{n+1}-t| \|\varphi_{tt}\|_{X}\dt+\int_{t_{n-1}}^{t_{n}}|t_{n-1}-t|\|\varphi_{tt}\|_{X}\dt.
         \end{align*}
    An appeal to a Cauchy-Schwarz inequality and $\int_{t_{n}}^{t_{n+1}}(t_{n+1}-t)^2\dt=\frac{k^3}{3}=\int_{t_{n-1}}^{t_{n}}(t_{n-1}-t)^2\dt$ and elementary manipulations provide
       \begin{align*}
         k \sum_{n=1}^m \|\bar{\partial}_t^2\varphi^n\|^2_{X} \le \frac{2}{3}\big(\sum_{n=1}^m\int_{t_{n}}^{t_{n+1}}\|\varphi_{tt}\|^2_{X}\dt+\sum_{n=1}^m\int_{t_{n-1}}^{t_{n}}\|\varphi_{tt}\|^2_{X}\dt\big)\le \frac{4}{3}\|\varphi_{tt}\|^2_{L^
         2{}(0,T;X)}. 
       \end{align*}
       This concludes the proof.
\end{proof}
\begin{lemma}\label{axul_ritz}
    Let $\xi=\varphi-\mathcal{R}_{\m}\varphi$. If 
 ($a$) $\varphi \in C([t_{m},t_{m+1}];H_0^{2}(\Omega) \cap H^{2+\sigma}(\Omega))$ for some  $0\le m \le N-1$,
 then  
  $ \trinl\xi^{m}\trinl_{\pw}$, $ \trinl\xi^{m+1/2}\trinl_{\pw}$ $ \le   \constref{car}h^\sigma\|\varphi\|_{C([t_{m},t_{m+1}];H^{2+\sigma}(\Omega))}$,
 ($b$)  $\varphi \in C([t_{m-1},t_{m+1}];H_0^{2}(\Omega) \cap H^{2+\sigma}(\Omega))$ for some  $1\le m \le N-1$, 
 then  $\trinl\xi^{m,1/4}\trinl_{\pw} \le  \constref{car}h^\sigma\|\varphi\|_{C([t_{m-1},t_{m+1}];H^{2+\sigma}(\Omega))}.$

 \medskip \noindent 
   Furthermore,  $\varphi \in  H^1(t_{m-1},t_{m+2};H_0^{2}(\Omega) \cap H^{2+\sigma}(\Omega))$ for some  $1\le m \le N-2$ satisfies \\
   \noindent 
   ($c$) $\trinl k \delta_t\xi^{m}\trinl_{\pw} \le  \constref{car}\sqrt{k/2}h^{\sigma}\|\varphi_t\|_{L^2(t_{m-1},t_{m+1};H^{2+\sigma}(\Omega))},$ 
 ($d$) $\trinl \xi^{m+1,1/4}-\xi^{m,1/4}\trinl_{\pw} \le   \constref{car}\sqrt{k/2}h^{\sigma}\|\varphi_t\|_{L^2(t_{m-1},t_{m+2};H^{2+\sigma}(\Omega))}$.
\end{lemma}
\begin{proof}
   ($a$)  Lemma~\ref{P1 ritz_lemma} and arguments similar to those used in  \eqref{evaluation} reveals
    \begin{equation}\label{eq:aux05}\trinl \xi^m \trinl_{\pw}=\trinl\varphi^{m}-\mathcal{R}_{\m}\varphi^{m}\trinl_{\pw} \le \constref{car}h^{\sigma}\|\varphi^{m}\|_{H^{2+\sigma}(\Omega)} \le \constref{car}h^{\sigma}\|\varphi\|_{C([t_{m},t_{m+1}];H^{2+\sigma}(\Omega))}.
    \end{equation}
    Definition \eqref{varphi1} and  triangle inequality followed by application of \eqref{eq:aux05} twice, shows that
    \begin{align*}
         \trinl\xi^{m+1/2}\trinl_{\pw} =\frac{1}{2}\trinl\xi^{m}+\xi^{m+1}\trinl_{\pw}
         &\le\frac{1}{2}\big(\trinl\xi^{m}\trinl_{\pw}+\trinl\xi^{m+1}\trinl_{\pw}\big)\le \constref{car}h^{\sigma}\|\varphi\|_{C([t_{m},t_{m+1}];H^{2+\sigma}(\Omega))}.
    \end{align*}
($b$) The proof follows analogous to  ($a$).\\
($c$) Owing to similar arguments employed in Lemma~\ref{axul}($d$), we can assert that  
\begin{align*}
  \trinl  k \delta_t\xi^{m}\trinl_{\pw}=  \trinl \xi^{m+1/2}-\xi^{m-1/2}\trinl_{\pw}  \le \sqrt{\frac{{k}}{2}}\big(\int_{t_{m-1}}^{t_{m+1}}\|\xi_t\|^2_{H^2(\Omega)}\big)^{1/2}.
\end{align*} The inequality $\big(\int_{t_{m-1}}^{t_{m+1}}\|\xi_t\|^2_{H^2(\Omega)}\big)^{1/2}\le\constref{car}h^{\sigma}\|\xi_t\|_{L^2(t_{m-1},t_{m+1};H^2(\Omega))}$ follows from the definition of $\xi(t)$ and  \eqref{P1 norm_ritz}.\\
($d$) The definition \eqref{un14} and triangle inequality leads to 
\begin{align*}
    \trinl \xi^{m+1,1/4}-\xi^{m,1/4}\trinl_{\pw}\le \frac{1}{2}\trinl \xi^{m+3/2}-\xi^{m+1/2}\trinl_{\pw}+\frac{1}{2}\trinl \xi^{m+1/2}-\xi^{m-1/2}\trinl_{\pw} .
\end{align*} As a consequence of  ($c$) we conclude  the proof.
\end{proof}
\begin{lemma}{\rm (discrete Gronwall \cite[Lemma~4.1]{MR3003381}\label{P1 d-gronwall})}
 Let $\{a_n\}$, $\{b_n\}$, and $\{c_n\}$ be three non-negative sequences, with $\{c_n\}$ monotone, that satisfy 
  $\displaystyle    a_m+b_m \le c_m + \mu \sum_{n=0}^{m-1} a_n, \quad \mu >0, \ a_0+b_0 \le c_0 $. Then for $m \ge 0,$ it holds that 
  $    a_m+b_m \le c_m e^{m \mu}.$
\end{lemma}
\begin{lemma}{\rm (truncation error  \cite[Theorem~4.4]{MR3003381}\label{trunc-lem})}
Let $\varphi \in C^2([0,T]; L^2(\Omega))$ satisfy
\begin{itemize}
    \item [($a$)]  $\varphi_{ttt} \in L^\infty(0,t_1;L^2(\Omega))$, then  $\norm{2k^{-1}(\bar{\partial}_t \varphi^{1/2}-\varphi_t(0))-\varphi_{tt}^{1/2}} \le  k \norm{\varphi_{ttt}}_{L^\infty(0,t_1;L^2(\Omega))},$ 
    \item [($b$)]  $\varphi \in H^4(0,T;L^2(\Omega))$, then $k\sum_{n=1}^{N-1} 
\norm{\bar{\partial}^2_t \varphi^n -\varphi_{tt}^{n,1/4} }^2 \le {\cal C}(\varphi_{tttt})k^{4}.$
\end{itemize}
The constant ${\cal C}(\varphi_{tttt})$ in ($b$) depends on $\norm{\varphi_{tttt}}_{L^2(0,T;L^2(\Omega))}.$
 \end{lemma}
\section{Proofs of results for the initial discretization}\label{sec-fully-ini}
{This section presents the proofs of the well-posedness result in 
Theorem~\ref{TH-WELLPOSED-INI} and the a priori error estimates in 
Theorem~\ref{LEM-ERROR-INI} for the initialization schemes~\eqref{app_at_0}--\eqref{P1ic}. 
}
\subsection{Well-posedness of initial discretizations}\label{subsec}
{The proof of Theorem~\ref{TH-WELLPOSED-INI}($a$) (resp. Theroem~\ref{TH-WELLPOSED-INI}($b$)) below establishes the existence of the discrete solution $(u_{\m}^0, v_{\m}^0)$ (resp. $(u_{\m}^1, v_{\m}^1)$) to \eqref{app_at_0} (resp. \eqref{P1ic}).
Theorem~\ref{TH-WELLPOSED-INI}($c$)  establishes stability of the solution to the 
initialization schemes~\eqref{app_at_0}--\eqref{P1ic} and 
Theorem~\ref{TH-WELLPOSED-INI}($d$) provides the uniqueness. This result is
subsequently employed in the error analysis of Theorem~\ref{LEM-ERROR-INI} 
and serves as the base case for the inductive well-posedness argument for 
the fully discrete scheme~\eqref{P4fully} in Theorem~\ref{TH-WELLPOSED-FINAL}. }
\begin{proof}[\large{\textbf{{Proof of Theorem~\ref{TH-WELLPOSED-INI}$(a)$}}}]
{The proof 
relies on a Lax-Milgram lemma.  For given $u_0 \in H^2_0(\Omega)$, the map $a(u_0, J(\bullet)): {\m}({\cal T}) \rightarrow \mathbb{R}$ is linear. This follows directly from the fact that $J$ is linear  and $a(\bullet,\bullet)$ is bilinear. Also, from continuity of $a(\bullet,\bullet)$ and $J$  we have $$|a(u_0, J\varphi_{\m})| \le \trinl u_0\trinl |J\varphi_{\m}|_{H^2(\Omega)} \le \|J\| \trinl u_0\trinl\trinl\varphi_{\m}\trinl _{\pw}.$$
 This implies that $a(u_0, J(\bullet))$ is a bounded linear functional on ${\m}({\cal T}).$ This, the continuity and ellipticity of the bilinear form  $a_{\pw}(\bullet,\bullet)$ on  ${\m}({\cal T})$ lead to the existence of a unique solution $u_{\m}^0$ to the first equation in  \eqref{app_at_0}. Similarly, the boundedness of  $b_{\pw}(u_{\m}^0, u_{\m}^0, \bullet): {\m}({\cal T}) \rightarrow \mathbb{R}$ from \eqref{Bh-bdd}, and the ellipticity of 
$a_{\pw}(\bullet,\bullet)$ ensures the existence of a unique solution $v_{\m}^0$ to the second equation in \eqref{app_at_0}. 

\medskip \noindent  The bound $\trinl u_{\m}^0\trinl_{\pw} \le \|J\|\trinl u_0\trinl$  follows from $\varphi_{\m}=u_{\m}^0$ in \eqref{app_at_0}, the continuity of $a(\bullet,\bullet)$, and $J$. The choice $\psi_{\m}=v_{\m}^0$ in \eqref{app_at_0},  \eqref{Bh-bdd}, and the a priori bound for $u_{\m}^0$ lead to  $\trinl v_{\m}^0 \trinl_{\pw} \le\frac{1}{2} C_{\rm dS}\|J\|^2\trinl u_0\trinl^2$.}
 \end{proof}
\noindent For any \( (\varphi_\m, \psi_\m) \in {\bf M}(\cal T) \),  recall from \eqref{P1ic} that the solution \( (u_{\m}^1, v_{\m}^1) \) of \eqref{P1ic} (near \( (u(t_1),v(t_1)) \)) satisfies  
\begin{align*}
{2k^{-1}}(\bar{\partial}_t u_{\m}^{1/2} -u_1, \varphi_{{\m}} )_{L^2(\Omega)}+a_{\pw}(u_{\m}^{1/2},\varphi_{{\m}}) +b_{\pw}(u_{\m}^{1/2},\varphi_{{\m}},v_{\m}^{1/2})&= (f^{1/2},\varphi_{{\m}} )_{L^2(\Omega)},\\
 a_{\pw}(v_{\m}^{1},\psi_{{\m}}) -\frac{1}{2}b_{\pw}(u_{\m}^{1},u_{\m}^{1},\psi_{{\m}})&=0  .
 \end{align*}
Unlike the coupled system in \eqref{app_at_0}, where the first equation does not involve $v_\m^0$
(allowing the equations to be solved sequentially and the system to remain linear), the proof of the existence result for this nonlinear coupled system requires say, fixed-point theory; we employ a corollary of Brouwer fixed-point theorem.
\begin{thm}[Brouwer fixed-point theorem {\cite[Theorem 5.2.5]{kesavan}\cite[Page no. 529]{MR2597943}}] \label{Brouwer}
    Let ${\bf H}$ be a finite dimensional
Hilbert space with the norm $\norm{\bullet}_{\bf H} $ induced from the inner product $(\bullet,\bullet)_{\bf H}$.
 Let ${\boldsymbol{\mathcal{L}}}:{\bf H}\rightarrow{\bf H}$ be a  continuous mapping such  $(\boldsymbol{\mathcal{L}}(\boldsymbol{{\varphi}}),\boldsymbol{{\varphi}})_{\bf H} \ge 0$ for all $\boldsymbol{{\varphi}} \in {\bf H}$ with $\norm{\boldsymbol{{\varphi}}}_{\bf H}={\rm R}>0$. Then there exists $\boldsymbol{{\varphi^*}} \in{\bf H}$ with $\norm{\boldsymbol{{\varphi}}^*}_{\bf H} \le {\rm R}$ such that  $\boldsymbol{\mathcal{L}}(\boldsymbol{{\varphi}^*})=0$.
\end{thm}
\begin{proof}[\large{\textbf{{Proof of Theorem~\ref{TH-WELLPOSED-INI}($b$)}}}] We follow  {\it four} steps. The first step describes the settings in the context of Brouwer's fixed-point theorem. This is followed by verification of the assumptions in Steps 2 and 3. Step 4 concludes the proof. 

\medskip \noindent {\it Step 1. Setting.}
Choose ${\bf H}={{\bf M}({\cal T})}$ and $\displaystyle \text{for all } \boldsymbol{\Psi_{\m}}=({\psi_{{\m}1}},{\psi_{{\m}2}}) $ , $ \boldsymbol{\Phi_{\m}}=({\varphi_{{\m1}}},{\varphi_{{\m2}}}) \in {{\bf M}({\cal T})}$, define the inner product in ${\bf H}$ by $(\boldsymbol{\Psi_{\m}},\boldsymbol{\Phi_{\m}})_{\pw}:=2(\psi_{{\m}1},{\varphi_{{\m1}}})_{L^2(\Omega)}+\frac{k^2}{2}\left( a_{\pw}(\psi_{{\m}1},{\varphi_{{\m1}}})+a_{\pw}(\psi_{{\m 2}},{\varphi_{{\m2}}})\right)$.
The norm induced by this inner product is $\norm{\boldsymbol{\Phi_{\m}}}_{\pw}=(\boldsymbol{\Phi_{\m}},\boldsymbol{\Phi_{\m}})^{1/2}_{\pw}.$ The definition of $\| \bullet \|_{\pw}$ leads to
\begin{equation}
  \sqrt{2}\|\varphi_{{\m}1}\| \le \| \boldsymbol{\Phi_{\m}}\|_{\pw},\;{k}\trinl\varphi_{{\m}j}\trinl_{\pw}\le{\sqrt{2}} \| \boldsymbol{\Phi_{\m}} \|_{\pw} \; \text{ for } j=1,2
  .
  \label{norm-com}
\end{equation}
Next, we introduce the map $\boldsymbol{\mathcal{G}}_{\boldsymbol{\Psi_{\m}}}: {{\bf M}({\cal T})} \rightarrow \mathbb{R}$ that helps to define $\boldsymbol{\mathcal{L}}$ via a Riesz representation. For a fixed $\boldsymbol{\Psi_{\m}} =(\psi_{{\m}1},\psi_{{\m}2}) \in {{\bf M}({\cal T})}$ define 
\begin{align}
    \boldsymbol{\mathcal{G}}_{\boldsymbol{\Psi_{\m}}}(\boldsymbol{\Phi_{\m}}):=
       (\boldsymbol{\Psi_{\m}},\boldsymbol{\Phi_{\m}})_{\pw}+\frac{{k^2}}{4}\big[b_{\pw}( \psi_{{\m}1}+2u_{\m}^{0}, \varphi_{{\m}1},\psi_{{\m}2})-b_{\pw}( \psi_{{\m}1}+u_{\m}^{0}, \psi_{{\m}1}+u_{\m}^{0},\varphi_{{\m}2})\big]\nonumber\\
       -{k^2}\big[{(f^{1/2}+2k^{-1}u_1,\varphi_{{\m}1})_{L^2(\Omega)}}-a_{\pw}(u_{\m}^{0},\varphi_{{\m}1})+\frac{1}{2}a_{\pw}(v_{\m}^{0},\varphi_{{\m}2})\big] .\label{gmap}
\end{align}
(The definition of $\boldsymbol{\mathcal{G}}_{\boldsymbol{\Psi_{\m}}}$ is in fact motivated by  \eqref{P1ic} as seen in Step 4). 
It is easy to verify that
$\boldsymbol{\mathcal{G}}_{\boldsymbol{\Psi_{\m}}}(\bullet)$ is linear. We establish  its continuity with respect to $\|\bullet\|_\pw$. Since \eqref{Bh-bdd} and \eqref{norm-com} hold and $k \le T$, we conclude that
\begin{align*}
   &\frac{{k^2}}{4}\big|b_{\pw}( \psi_{{\m}1}+2u_{\m}^{0}, \varphi_{{\m}1},\psi_{{\m}2})-b_{\pw}( \psi_{{\m}1}+u_{\m}^{0}, \psi_{{\m}1}+u_{\m}^{0},\varphi_{{\m}2})\big| \\
   &\quad \le C_{\rm dS}\frac{{k^2}}{4}\big[\trinl\psi_{{\m}1}+2u_{\m}^{0}\trinl_{\pw}\trinl\varphi_{{\m}1}\trinl_{\pw}\trinl\psi_{{\m}2}\trinl_{\pw}+\trinl\psi_{{\m}1}+u_{\m}^{0}\trinl_{\pw}^2\trinl\varphi_{{\m}2}\trinl_{\pw}\big]\\
   &\quad \le C_{\rm dS}\frac{T}{\sqrt{8}}\big[\trinl\psi_{{\m}1}+2u_{\m}^{0}\trinl_{\pw}\trinl\psi_{{\m}2}\trinl_{\pw}+\trinl\psi_{{\m}1}+u_{\m}^{0}\trinl_{\pw}^2\big]\norm{\boldsymbol{\Phi_{{\m}}}}_{\pw}.
\end{align*}
Elementary manipulations with  \eqref{norm-com} and $k \le T$ provide
\begin{align*}
   {k^2} |(f^{1/2},\varphi_{{\m}1})_{L^2(\Omega)}+2k^{-1}(u_1,\varphi_{{\m}1})_{L^2(\Omega)}|&\le \frac{1}{\sqrt{2}}\big(T^2\|f^{1/2}\|+2T\|u_1\|\big)\norm{\boldsymbol{\Phi_{{\m}}}}_{\pw},\\
  k^2 |a_{\pw}(u_{\m}^{0},\varphi_{{\m}1})-\frac{1}{2}a_{\pw}(v_{\m}^{0},\varphi_{{\m}2})| &\le \sqrt{2} T (\trinl u_{\m}^{0} \trinr_\pw +\frac{1}{2}\trinl v_{\m}^{0} \trinr_\pw) \norm{\boldsymbol{\Phi_{{\m}}}}_{\pw} .
\end{align*}
A combination of the last three displayed inequalities and $(\boldsymbol{\Psi_{\m}},\boldsymbol{\Phi_{\m}})_{\pw} \le \|\boldsymbol{\Psi_{\m}}\|_{\pw}\|\boldsymbol{\Phi_{\m}}\|_{\pw }$ in \eqref{gmap} 
reveals that $\boldsymbol{\mathcal{G}}_{\boldsymbol{\Psi_{\m}}}(\bullet)$ is a bounded linear functional in the Hilbert space $\big({{\bf M}({\cal T})},\|\bullet \|_\pw\big)$.
Its Riesz representation ${\boldsymbol{\mathcal{L}}}(\boldsymbol{\Psi_{\m}})$ defines the  map ${\boldsymbol{\mathcal{L}}}: {{\bf M}({\cal T})} \rightarrow{{\bf M}({\cal T})}$ given by 
   \begin{align}
       (\boldsymbol{\mathcal{L}}(\boldsymbol{\Psi_{\m}}),\boldsymbol{\Phi_{\m}})_{\pw} =\boldsymbol{\mathcal{G}}_{\boldsymbol{\Psi_{\m}}}(\boldsymbol{\Phi_{\m}}) \;\text{for all } \boldsymbol{\Phi_{\m}}\in {{\bf M}({\cal T})}. \label{Fmap-ini}
       \end{align}
\noindent{\it Step 2. {$\boldsymbol{\mathcal{L}}$ is continuous.}} For $\boldsymbol{\Psi_{\m}}= (\psi_{{\m}1}, \psi_{{\m}2}), \boldsymbol{\Psi^*_{\m}}=(\psi^*_{{\m}1},\psi^*_{{\m}2}) \in 
  {{\bf M}({\cal T})}$ and for all $\boldsymbol{\Phi_{\m}}=(\varphi_{{\m}1},\varphi_{{\m}2}) \in {{\bf M}({\cal T})}$, the relations \eqref{gmap} and \eqref{Fmap-ini} show 
\begin{align*}
  &(\boldsymbol{\mathcal{L}}(\boldsymbol{\Psi_{\m}}) -\boldsymbol{\mathcal{L}}(\boldsymbol{\Psi^*_{\m}}),\boldsymbol{\Phi_{\m}})_{\pw}=(\boldsymbol{\Psi_{\m}}-\boldsymbol{\Psi^*_{\m}},\boldsymbol{\Phi_{{\m}}})_{\pw}
     +\frac{k^2}{4}\big[b_{\pw}( \psi_{{\m}1}+2u_{\m}^{0}, \varphi_{{\m}1},\psi_{{\m}2} )\\
     &-b_{\pw}( \psi_{{\m}1}+u_{\m}^{0}, \psi_{{\m}1}+u_{\m}^{0},\varphi_{{\m}2})-b_{\pw}( \psi_{{\m}1}^*+2u_{\m}^{0}, \varphi_{{\m}1},\psi_{{\m}2}^* )+b_{\pw}( \psi_{{\m}1}^*+u_{\m}^{0}, \psi_{{\m}1}^*+u_{\m}^{0},\varphi_{{\m}2})\big].
\end{align*}
Abbreviate $\boldsymbol{\mathcal{L}}(\boldsymbol{\Psi_{\m}})-\boldsymbol{\mathcal{L}}(\boldsymbol{\Psi^*_{\m}})=\boldsymbol{\Theta_{{\m}}}=(\theta_{{\m}1},\theta_{{\m}2})$ and  test  $\boldsymbol{\Phi_{\m}}=\boldsymbol{\Theta_{{\m}}}$ in the above expression to infer
\begin{align}
    \norm{\boldsymbol{\Theta_{{\m}}}}^2_{\pw}=(\boldsymbol{\Psi_{\m}}-\boldsymbol{\Psi^*_{\m}},\boldsymbol{\Theta_{{\m}}})
    _{{ {\pw}}}+\frac{k^2}{4}\big[b_{\pw}( \psi_{{\m}1}+2u_{\m}^{0}, \theta_{{\m}1},\psi_{{\m}2} )-b_{\pw}( \psi_{{\m}1}+u_{\m}^{0}, \psi_{{\m}1}+u_{\m}^{0},\theta_{{\m}2})\nonumber\\
    -b_{\pw}( \psi_{{\m}1}^*+2u_{\m}^{0}, \theta_{{\m}1},\psi_{{\m}2}^* )+b_{\pw}( \psi_{{\m}1}^*+u_{\m}^{0}, \psi_{{\m}1}^*+u_{\m}^{0},\theta_{{\m}2})\big].\label{b1+b2-ini}
\end{align}
Applying Cauchy-Schwarz inequality we readily obtain $(\boldsymbol{\Psi_{\m}}-\boldsymbol{\Psi^*_{\m}},\boldsymbol{\Theta_{\m}})_{\pw} \le \|\boldsymbol{\Psi_{\m}}-\boldsymbol{\Psi^*_{\m}}\|_{\pw}\|\boldsymbol{\Theta_{\m}}\|_{\pw}$. 
 The linearity of $b_{\pw}(\bullet,\bullet,\bullet)$ in each component and elementary manipulations then show
\begin{align*}
  &b_{\pw}( \psi_{{\m}1}+2u_{\m}^{0}, \theta_{{\m}1},\psi_{{\m}2} )-b_{\pw}( \psi_{{\m}1}^*+2u_{\m}^{0}, \theta_{{\m}1},\psi_{{\m}2}^*)\nonumber\\
   &\qquad =b_{\pw}( \psi_{{\m}1}- \psi_{{\m}1}^*, \theta_{{\m}1},\psi_{{\m}2})+b_{\pw}( \psi_{{\m}1}^*+2u_{\m}^{0}, \theta_{{\m}1},\psi_{{\m}2}- \psi_{{\m}2}^*),\\
  & b_{\pw}( \psi_{{\m}1}^*+u_{\m}^{0}, \psi_{{\m}1}^*+u_{\m}^{0},\theta_{{\m}2})-b_{\pw}( \psi_{{\m}1}+u_{\m}^{0}, \psi_{{\m}1}+u_{\m}^{0},\theta_{{\m}2})\nonumber\\
  &\qquad =b_{\pw}( \psi_{{\m}1}^*- \psi_{{\m}1},\psi_{{\m}1}^*+u_{\m}^{0},\theta_{{\m}2})+b_{\pw}( \psi_{{\m}1}+u_{\m}^{0}, \psi_{{\m}1}^*-\psi_{{\m}1},\theta_{{\m}2}). 
\end{align*}
The boundedness of $b_{\pw}(\bullet,\bullet,\bullet)$ from \eqref{Bh-bdd} and 
\eqref{norm-com} with the abbreviation   $(\psi_{{\m}1}- \psi_{{\m}1}^*,\psi_{{\m}2}- \psi_{{\m}2}^*)=\boldsymbol{\Psi_{\m}}- \boldsymbol{\Psi^*_{\m}}$, $(\theta_{{\m}1},\theta_{{\m}2})=\boldsymbol{\Theta_{\m}}$ lead to
\begin{align*}
    b_{\pw}( \psi_{{\m}1}- \psi_{{\m}1}^*, \theta_{{\m}1},\psi_{{\m}2})
   & \le C_{\rm dS}\trinl  \psi_{{\m}1}- \psi_{{\m}1}^*\trinl _{\pw}\trinl\theta_{{\m}1}\trinl_{\pw}\trinl \psi_{{\m}2} \trinl_{\pw} \nonumber\\
    &\le 2 C_{\rm dS}k^{-2}  \|
    \boldsymbol{\Psi_{\m}}- \boldsymbol{\Psi^*_{\m}}\|_{\pw}\|\boldsymbol{\Theta_{\m}}\|_{\pw} \trinl \psi_{{\m}2} \trinl_{\pw}. 
\end{align*}
Analogous arguments reveal 
\begin{align*}
      &b_{\pw}( \psi_{{\m}1}^*+2u_{\m}^{0}, \theta_{{\m}1},\psi_{{\m}2}- \psi_{{\m}2}^*)+b_{\pw}( \psi_{{\m}1}^*- \psi_{{\m}1},\psi_{{\m}1}^*+u_{\m}^{0},\theta_{{\m}2})+b_{\pw}( \psi_{{\m}1}+u_{\m}^{0}, \psi_{{\m}1}-\psi_{{\m}1}^*,\theta_{{\m}2})\nonumber\\
      &\qquad\quad\le   2C_{\rm dS}k^{-2}\big(\trinl \psi_{{\m}1}^*+2u_{\m}^{0}\trinl_{\pw} +\trinl\psi_{{\m}1}^*+u_{\m}^{0}\trinl_{\pw} 
    +\trinl\psi_{{\m}1}+u_{\m}^{0}\trinl_{\pw}\big)\|\boldsymbol{\Psi_{\m}}- \boldsymbol{\Psi^*_{\m}}\|_{\pw}\|\boldsymbol{\Theta_{\m}}\|_{\pw} . 
\end{align*}
A  combination of all this in  \eqref{b1+b2-ini} shows  
that the map $\boldsymbol{\mathcal{L}}$ is (locally Lipschitz) continuous.

\medskip \noindent 
\textit{Step 3.  {$(\boldsymbol{\mathcal{L}}(\boldsymbol{\Psi_{\m}}),\boldsymbol{\Psi_{\m}})_{\pw} \ge 0$.}} Select $\boldsymbol{\Psi_{\m}} \in {{\bf M}({\cal T})}$ such that $\norm{\boldsymbol{\Psi_{\m}}}_{\pw}={\rm R}$, 
where
\begin{align}
   {\rm R}&=k \sqrt{2{k^2}\norm{f}^2_{C([0,t_1];L^2(\Omega))}+8\norm{u_1}^2+\|J\|^2\trinl u_0\trinl^2\big(4+ C^2_{\rm dS}\|J\|^2\trinl u_0\trinl^2\big)}.\label{raduis-ini}
\end{align} 
 First observe that the linearity, symmetry of $b_{\pw}(\bullet,\bullet,\bullet)$ in its first two components, and \eqref{app_at_0} shows $$b_{\pw}( \psi_{{\m}1}+2u_{\m}^{0}, \psi_{{\m}1},\psi_{{\m}2} )-b_{\pw}( \psi_{{\m}1}+u_{\m}^{0}, \psi_{{\m}1}+u_{\m}^{0},\psi_{{\m}2})=-b_{\pw}( u_{\m}^{0}, u_{\m}^{0},\psi_{{\m}2})=-2a_{\pw}(v_{\m}^{0},\psi_{{\m}2}).$$ The  choice 
$\boldsymbol{\Phi_{\m}}=\boldsymbol{\Psi_{\m}}$ in \eqref{Fmap-ini},  \eqref{gmap}, and the last displayed identity lead to
\begin{align*}
(\boldsymbol{\mathcal{L}}(\boldsymbol{\Psi_{\m}}),\boldsymbol{\Psi_{\m}})_{\pw}
     &=\norm{\boldsymbol{\Psi_{\m}}}^2_{\pw}-{k^2}(f^{1/2},\psi_{{\m}1})_{L^2(\Omega)}-2k(u_1,\psi_{{\m}1})_{L^2(\Omega)}+{{k^2}}a_{\pw}(u_{\m}^{0},\psi_{{\m}1})-{k^2}a_{\pw}(v_{\m}^{0},\psi_{{\m}2}).
\end{align*}
A Cauchy-Schwarz inequality and Young's inequality with {$a=k^2 \norm{f^{1/2}+2k^{-1}u_1},b=\norm{\psi_{{\m}1}}$, $\epsilon=1/2$} followed by a triangle inequality and  the bound $\norm{f^{1/2}} \le \norm{f}_{C([0,t_1];L^2(\Omega))}$ from Lemma~\ref{axul}$(a)$ reveal 
{
\begin{align*}
    k^2(f^{1/2}+2k^{-1}u_1,\psi_{{\m}1})_{L^2(\Omega)}
   \le k^2\|f^{1/2}+2k^{-1}u_1\|\|\psi_{{\m}1}\|
    & \le\frac{k^4}{2}\norm{f}^2_{C([0,t_1];L^2(\Omega))}+2k^2\norm{u_1}^2 +\|\psi_{{\m}1}\|^2.
\end{align*}}
The continuity of bilinear form $a_{\pw}(\bullet,\bullet)$, Young's inequality with $a=\trinl u_{\m}^{0}\trinl_{\pw},b=\trinl\psi_{{\m}1}\trinl_{\pw} $, and $\epsilon=2$ (resp. $a=\trinl v_{\m}^{0}\trinl_{\pw},b=\trinl\psi_{{\m}2}\trinl_{\pw} $, and $\epsilon=2$), and the bounds from ($a$) reveal
\begin{align*}
    |a_{\pw}(u_{\m}^0,\psi_{{\m}1})|& \le\trinl u_{\m}^0\trinl_{\pw}\trinl\psi_{{\m}1}\trinl_{\pw} \le \trinl u_{\m}^0\trinl_{\pw}^2+\frac{1}{4}\trinl\psi_{{\m}1}\trinl_{\pw} ^2\le \|J\|^2\trinl u_0\trinl^2+\frac{1}{4}\trinl\psi_{{\m}1}\trinl_{\pw} ^2,\\
\big(\text{resp. }a_{\pw}(v_{\m}^{0},\psi_{{\m}2}) & \le \trinl v_{\m}^{0}\trinl_{\pw}\trinl\psi_{{\m}2}\trinl_{\pw} \le \trinl v_{\m}^{0}\trinl_{\pw}^2+\frac{1}{4}\trinl \psi_{{\m}2}\trinl_{\pw}^2\le \frac{1}{4} C^2_{\rm dS}\|J\|^4\trinl u_0\trinl^4_{H^2(\Omega)}+\frac{1}{4}\trinl \psi_{{\m}2}\trinl_{\pw}^2\big).
\end{align*}
A combination of the four last displayed results and the definition of ${\rm R}$ from \eqref{raduis-ini}  show
\begin{align*}
(\boldsymbol{\mathcal{L}}(\boldsymbol{\Psi_{\m}}),\boldsymbol{\Psi_{\m}})_{\pw}
     &\ge \norm{\boldsymbol{\Psi_{\m}}}^2_{\pw}-\big(\|\psi_{{\m}1} \|^2+\frac{k^2}{4}\trinl \psi_{{\m}1} \trinl_{\pw}^2+\frac{k^2}{4}\trinl \psi_{{\m}2} \trinl_{\pw}^2\big)-\frac{{\rm R}^2}{4}.
\end{align*}
Recall that $\boldsymbol{\Psi_{\m}}\in {{\bf M}({\cal T})}$ is chosen such that $\norm{\boldsymbol{\Psi_{\m}}}^2_{\pw}=2\| \psi_{{\m}1} \|^2+\frac{k^2}{2}\trinl \psi_{{\m}1} \trinl^2_{\pw}+\frac{k^2}{2}\trinl \psi_{{\m}2} \trinl_{\pw}^2={\rm R}^2$. Utilize this in the last displayed inequality to obtain 
$$(\boldsymbol{\mathcal{L}}(\boldsymbol{\Psi_{\m}}),\boldsymbol{\Psi_{\m}})_{\pw} \ge {\rm R}^2-\frac{{\rm R}^2}{2}-\frac{{\rm R}^2}{4}\ge 0.$$

\noindent {\it Step 4. Conclusion.} Steps 2 and 3 show that \( \boldsymbol{\mathcal{L}} \) is continuous and there exists ${\rm R} >0$ such that {$ (\boldsymbol{\mathcal{L}} (\boldsymbol{ \Psi_{\m}}),\boldsymbol{ \Psi_{\m}})_{\pw} \ge 0$} for all $\boldsymbol{ \Psi_{\m}}$ with $\norm{\boldsymbol{ \Psi_{\m}}}_{\pw}={\rm R}$.  Theorem \ref{Brouwer} establishes the existence of $\boldsymbol{ \Psi^*_{\m}} \in{{\bf M}({\cal T})}$ such that $\norm{\boldsymbol{ \Psi^*_{\m}}}_{\pw}\le {\rm R}$ and $ \boldsymbol{\mathcal{L}} (\boldsymbol{ \Psi^*_{\m}})=0$. {Define $(u^{1}_{\m},v^{1}_{\m}):=( \psi^*_{{\m}1}+u_{\m}^{0},\psi^*_{{\m}2}-v_{\m}^{0})$, choose $\boldsymbol{\Phi_{\m}}=k^{-2}(\varphi_{\m},0)$ (resp.  $\boldsymbol{\Phi_{\m}}=k^{-2}(0,2\psi_{\m})$) in  \eqref{Fmap-ini}, and utilize $ \boldsymbol{\mathcal{L}} (\boldsymbol{ \Psi^*_{\m}})=0$ to obtain \eqref{P1 ic_1} (resp. \eqref{P1 ic_2}).}
\end{proof}
 \begin{proof}[\large {\textbf{Proof of Theorem~\ref{TH-WELLPOSED-INI}$(c)$}}]
This  part establishes the stability for $(u_{\m}^1, v_{\m}^1)$ 
by deriving a discrete energy estimate for the scheme~\eqref{P1ic}, 
with the key ingredient being the nonlinear identity~\eqref{motiv-ini}.\\
     Choose $\varphi_{{\m}}= u_{\m}^{1}-u_{\m}^0=k\bar{\partial}_t u_{\m}^{1/2}$ in \eqref{P1 ic_1}, utilize $a_{\pw}(u_{\m}^{1/2},u_{\m}^{1}-u_{\m}^0)=\frac{1}{2}\big(\trinl u_{\m}^1\trinl_{\pw}^2-\trinl u_{\m}^0\trinl_{\pw}^2\big)$ and \eqref{motiv-ini} to obtain
     \begin{align}
        2 \norm{\bar{\partial}_t u_{\m}^{1/2}}^2+\frac{1}{2}\big(\trinl u_{\m}^1\trinl_{\pw}^2-\trinl u_{\m}^0\trinl_{\pw}^2+\trinl v_{\m}^{1}\trinl_{\pw}^2-\trinl v_{\m}^{0}\trinl_{\pw}^2\big)=(f^{1/2}+ {2k^{-1}u_1},u_{\m}^{1}- u_{\m}^{0})_{L^2(\Omega)}.\label{stabic1}
     \end{align}
 A   Cauchy-Schwarz 
    and a Young's inequality with $a=\|f^{1/2}+ {2k^{-1}u_1}\|, b=\|u_{\m}^{1}- u_{\m}^{0} \|,$ and $\epsilon=k^2/2$ reveal 
    \begin{align}
        (f^{1/2}+ {2k^{-1}u_1},u_{\m}^{1}- u_{\m}^{0} )_{L^2(\Omega)}  
        \le  \|f^{1/2}+2k^{-1}u_1\|\norm{u_{\m}^{1}- u_{\m}^{0} }
        &
        \le \frac{k^2}{2}\norm{f}^2_{C([0,t_1];L^2(\Omega))}+2\norm{u_1}^2+\norm{\bar{\partial}_t u_{\m}^{1/2}}^2 \label{f11}\end{align}
       with $\norm{f^{1/2}} \le \norm{f}_{C([0,t_1];L^2(\Omega))} $ from Lemma~\ref{axul}$(a)$ in the last inequality.
    A combination of \eqref{stabic1}- \eqref{f11}  and  the bounds $\trinl u_{\m}^0\trinl_{\pw}\le  \|J\|\trinl u_0\trinl$ and $\trinl v_{\m}^0\trinl_{\pw} \le \frac{1}{2}C_{\rm dS}\|J\|^2\trinl u_0\trinl^2$ from ($a$)  shows
     \begin{align}
          2\norm{\bar{\partial}_t u_{\m}^{1/2}}^2+\trinl u_{\m}^1\trinl_{\pw}^2+\trinl v_{\m}^1\trinl_{\pw}^2 
          & \le \|J\|^2\trinl u_0\trinl^2\big(1+\frac{1}{4}C_{\rm dS}^2\|J\|^2\trinl u_0\trinl^2\big)+4 \norm{u_1}^2+ {k^2}\norm{f}^2_{C([0,t_1];L^2(\Omega))}.
          \label{stab-halfway}
     \end{align}
Elementary algebra
concludes  the proof.
 \end{proof}
\begin{proof}[\large{\textbf{Proof of Theorem~\ref{TH-WELLPOSED-INI}$(d)$}}]
If possible, let $(u_{\m}^{1},v_{\m}^{1})$ and $(\hat{u}_{\m}^{1},\hat{v}_{\m}^{1})$ be two solutions to \eqref{P1ic} and ${e}_{\m}=u_{\m}^{1}-\hat{u}_{\m}^{1}, \bar{e}_{\m}=v_{\m}^{1}-\hat{v}_{\m}^{1}$. For all $(\varphi_{\m}, \psi_{\m}) \in {{\bf M}({\cal T})},$ \eqref{P1ic} provides
\begin{align*}
   \frac{2}{k^2}( {e}_{\m},\varphi_{\m})_{L^2(\Omega)} & + \frac{1}{2}a_{\pw} ({e}_{\m},\varphi_{\m} )+\frac{1}{4}b_{\pw}(u_{\m}^{1}+u_{\m}^{0},\varphi_{\m},v_{\m}^{1}+v_{\m}^0)-\frac{1}{4}b_{\pw}(\hat{u}_{\m}^{1}+u_{\m}^{0},\varphi_{\m},\hat{v}_{\m}^{1}+v_{\m}^0)= 0, \\
  &\qquad a_{\pw}(\bar{e}_{\m},\psi_{\m}) =\frac{1}{2}b_{\pw}(u_{\m}^1,u_{\m}^1,\psi_{\m})-\frac{1}{2}b_{\pw}(\hat{u}_{\m}^1,\hat{u}_{\m}^1,\psi_{\m}).
    \end{align*}
   The choices of  test functions $\varphi_{\m}=4e_{\m}$ and $\psi_{\m}=2\bar{e}_{\m}$ in the two equations above, respectively, lead to
  \begin{subequations}  \begin{align}
    \frac{8}{k^2}\norm{{e}_{\m}}^2 + 2\trinl{e}_{\m}\trinl_{\pw}^2 &=-b_{\pw}(u_{\m}^{1}+u_{\m}^{0},e_{\m},v_{\m}^{1}+v_{\m}^0)+b_{\pw}(\hat{u}_{\m}^{1}+u_{\m}^{0},e_{\m},\hat{v}_{\m}^{1}+v_{\m}^0),\label{u-1-ini} \\
  2\trinl\bar{e}_{\m}\trinl^2_{\pw} &=b_{\pw}(u_{\m}^1,u_{\m}^1,\psi_{\m})-b_{\pw}(\hat{u}_{\m}^1,\hat{u}_{\m}^1,\psi_{\m}). \label{v-ini}
    \end{align}\end{subequations}
Add  $\pm b_{\pw}(\hat{u}_{\m}^{1}+u_{\m}^{0},e_{\m},{v}_{\m}^{1}+v_{\m}^0)$ and  utilize trilinearity of $b_{\pw}(\bullet,\bullet,\bullet)$   to simplify the right-hand side of  \eqref{u-1-ini} as
\begin{align*}
  &-b_{\pw}(u_{\m}^{1}+u_{\m}^{0},e_{\m},v_{\m}^{1}+v_{\m}^0)+b_{\pw}(\hat{u}_{\m}^{1}+u_{\m}^{0},e_{\m},\hat{v}_{\m}^{1}+v_{\m}^0) 
 =-b_{\pw}(e_{\m},e_{\m},v_{\m}^{1}+v_{\m}^{0})-b_{\pw}(\hat{u}_{\m}^{1}+u_{\m}^{0},e_{\m},\bar{e}_{\m}).
\end{align*}
Add $\pm b_{\pw}(\hat{u}_{\m}^1,{u}_{\m}^1,\psi_{\m})$ and utilize  symmetry of $b_{\pw}(\bullet,\bullet,\bullet)$ to rewrite the right-hand side of \eqref{v-ini} as
\begin{align*}
    b_{\pw}(u_{\m}^{1},u_{\m}^1,\bar{e}_{\m})-b_{\pw}(\hat{u}_{\m}^{1},\hat{u}_{\m}^{1},\bar{e}_{\m})=b_{\pw}(e_{\m},u_{\m}^{1},\bar{e}_{\m})+b_{\pw}(\hat{u}_{\m}^{1},e_{\m},\bar{e}_{\m})=b_{\pw}(u_{\m}^{1}+\hat{u}_{\m}^{1},e_{\m},\bar{e}_{\m}).
\end{align*}
{
A combination of last four displayed equations and  boundedness of $b_{\pw}(\bullet,\bullet,\bullet)$ from \eqref{Bh-bdd} imply that 
\begin{align*}
    \frac{8}{k^2}\norm{{e}_{\m}}^2 + 2\trinl{e}_{\m}\trinl_{\pw}^2+2\trinl\bar{e}_{\m}\trinl_{\pw}^2
& =-b_{\pw}(e_{\m},e_{\m},v_{\m}^{1}+v_{\m}^{0})+b_{\pw}(u_{\m}^1-u_{\m}^{0},e_{\m},\bar{e}_{\m})\nonumber\\
    &\le C_{\rm dS} \trinl v_{\m}^{1}+v_{\m}^{0}\trinl_{\pw}\trinl e_{\m}\trinl_{\pw}^2+C_{\rm dS}\trinl u_{\m}^1-u_{\m}^{0}\trinl_{\pw}\trinl e_{\m}\trinl_{\pw}\trinl \bar{e}_{\m}\trinl_{\pw}. 
\end{align*} 
 Ignore the first non-negative term on the left-hand side and apply Young's inequality with $a=C_{\rm dS}\trinl u_{\m}^1-u_{\m}^{0}\trinl_{\pw}\trinl e_{\m}\trinl_{\pw}$, $b=\trinl\bar{e}_{\m}\trinl_{\pw}$, and $\epsilon=1$ for the last term on the right-hand side to derive
\begin{align}
 & 4\trinl {e}_{\m}\trinl_{\pw}^2+4\trinl\bar{e}_{\m}\trinl_{\pw}^2 
 \le  C_{\rm dS}\ \big(4\trinl v_{\m}^{1/2}\trinl_{\pw} + C_{\rm dS}\trinl u_{\m}^1-u_{\m}^{0}\trinl_{\pw}^2\big)\trinl e_{\m}\trinl_{\pw}^2  +\trinl \bar{e}_{\m}\trinl_{\pw}^2.\label{use-small}
\end{align}
Choose $\psi_{\m}=v_{\m}^{0}$ (resp. $\psi_{\m}=v_{\m}^{1}$) in \eqref{P1 ic_2} and utilize \eqref{P1 mesh_norm} followed by \eqref{Bh-bdd} to obtain
\begin{align*}
 \trinl v_{\m}^{0}\trinl_{\pw} \le\frac{1}{2}C_{\rm dS}\trinl u_{\m}^0\trinl_{\pw}^2\;\big(\text{resp. }  \trinl v_{\m}^{1}\trinl_{\pw} \le\frac{1}{2}C_{\rm dS}\trinl u_{\m}^1\trinl_{\pw}^2\big).
\end{align*}
This, a triangle inequality, the  bounds  for $\trinl u_{\m}^{0}\trinl_{\pw}$ from  ($a$), and that of $\trinl u_{\m}^1\trinl_{\pw}$ from \eqref{stab-halfway} yield
\begin{align}
 & 4 \trinl v_{\m}^{1/2}\trinl_{\pw} + C_{\rm dS}\trinl u_{\m}^1-u_{\m}^{0}\trinl_{\pw}^2 \le 2 \big(\trinl v_{\m}^{0}\trinl_{\pw}+\trinl v_{\m}^{1}\trinl_{\pw}\big) + 2C_{\rm dS}\big(\trinl u_{\m}^{0}\trinl_{\pw}^2+\trinl u_{\m}^1\trinl_{\pw}^2\big)\nonumber\\
  & \le 3C_{\rm dS}\big(\trinl u_{\m}^0\trinl_{\pw}^2+\trinl u_{\m}^1\trinl_{\pw}^2\big) \le 3C_{\rm dS}\big(\|J\|^2\trinl u_0\trinl^2+ \text{\v{M}}_{2}(u_0,u_1,f,k)\big)= 3C_{\rm dS}^{-1} {\rm M}_{2}(u_0,u_1,f,k).\label{sm}
\end{align}
Utilize this in the right-hand side of \eqref{use-small} to obtain
\begin{align}
   4\trinl {e}_{\m}\trinl_{\pw}^2+3\trinl\bar{e}_{\m}\trinl_{\pw}^2 
   \le 3{\rm M}_{2}(u_0,u_1,f,k)\trinl {e}_{\m}\trinl_{\pw}^2.\label{sml}
\end{align}
   For sufficiently small $u_0$, $u_1$, and $k$, we assume   ${\rm M}_{2}(u_0,u_1,f,k)\le 1 $ and then  \eqref{sml} reveals 
    $ \trinl {e}_{\m}\trinl_{\pw}^2+3\trinl \bar{e}_{\m}\trinl_{\pw}^2\le 0 $}.
\end{proof}
 \subsection{A priori error analysis for initial discretization}\label{sub-inierror}

 For any $n=0,1,2,\cdots,N$, let us  split  the errors as follows:
\begin{subequations} \begin{align}
&u^n-u_{\m}^n= \big(u^n-\mathcal{R}_{\m}u^n\big)+\big(\mathcal{R}_{\m}u^n-u_{\m}^n\big):=\rho^n+\zeta^n,\label{p4-split1}\\
& v^n-v_{\m}^n= \big(v^n-\mathcal{R}_{\m}v^n\big)+\big(\mathcal{R}_{\m}v^n-v_{\m}^n\big):=\chi^n+\theta^n.\label{p4-split2}
\end{align} \end{subequations}
 This, the linearity of $\mathcal{R}_{\m}$, and the definitions in \eqref{varphi1}-\eqref{varphi2} lead to 
 \begin{align*}
     &\rho^{n+1/2}=\frac{1}{2}\big(\rho^{n+1}+\rho^{n}\big),\;\rho^{n,1/4}=\frac{1}{4}\big(\rho^{n+1}+2\rho^{n}+\rho^{n-1}\big),\\
     & \bar{\partial}_t\rho^{n+1/2}=\frac{1}{k}\big(\rho^{n+1}-\rho^{n}\big),\;\delta_t \rho^n=\frac{1}{2k}\big(\rho^{n+1}-\rho^{n-1}\big),\; \bar{\partial}_t^2\rho^{n}=\frac{1}{k^2}\big(\rho^{n+1}-2\rho^{n}+\rho^{n-1}\big).
 \end{align*}
Analog notation replace $\rho^n$ by $\zeta^n,\chi^n,$ and $\theta^n$.
\begin{proof}[\large \textbf{Proof of Theorem~\ref{LEM-ERROR-INI}}] We follow \textit{five} steps. Firstly, we derive the error equation. Step 2 bounds the linear terms appearing in the error equation. In Step 3, the nonlinear terms are rearranged.   Bounds for these terms are derived in Step 4. In Step 5, the truncation term is estimated. Step 6  concludes the proof. 

\medskip \noindent{\it Step 1. Error equation.}
For any $\varphi_{{\m}} \in  \m(\cal{T})$, choose the test function in \eqref{P1 weak_form1} as $J\varphi_{{\m}}$ and consider the average of \eqref{P1 weak_form1} at $t=t_0$ and $t_1$. Subtract \eqref{P1 ic_1} from the resultant expression to arrive at 
\begin{align}
    &(u_{tt}^{1/2}, J\varphi_{{\m}})_{L^2(\Omega)}+a(u^{1/2},J\varphi_{{\m}})-{2k^{-1}}(\bar{\partial}_t u_{\m}^{1/2} -u_1, \varphi_{{\m}} )_{L^2(\Omega)}-a_{\pw}(u_{\m}^{1/2},\varphi_{{\m}}) \nonumber\\
   & \qquad  =(f^{1/2},(J-I)\varphi_{{\m}} )_{L^2(\Omega)}
    -  \frac{1}{2}\big(b(u^{1},J\varphi_{{\m}},v^{1})+b(u_0,J\varphi_{{\m}},v^{0})\big)+b_{\pw}(u_{\m}^{1/2},\varphi_{{\m}},v_{\m}^{1/2}).
    \label{ee1}
\end{align}
The linearity of \(b(\bullet,\bullet,\bullet)\) in each component and definitions \eqref{un14}-\eqref{dtu} with elementary algebra lead to 
$$ b(u^{1},J\varphi_{{\m}},v^{1})+b(u_0,J\varphi_{{\m}},v^{0})
=2b(u^{1/2},J\varphi_{{\m}},v^{1/2})+\frac{k^2}{2}b(\bar{\partial}_t u^{1/2},J\varphi_{{\m}},\bar{\partial}_t v^{1/2}).$$ 
Then, $a(u^{1/2},J\varphi_{{\m}})=a_{\pw}(\mathcal{R}_{\m}u^{1/2},\varphi_{{\m}})$ from  \eqref{P1 ritz_projection}, \eqref{p4-split1}, and elementary manipulations in \eqref{ee1} establish 
\begin{align*}
  &   2k^{-1}(\bar{\partial}_t \zeta^{1/2}, \varphi_{{\m}} )_{L^2(\Omega)}+a_{\pw}(\zeta^{1/2},\varphi_{{\m}})
    =(f^{1/2}-u_{tt}^{1/2},(J-I)\varphi_{{\m}} )_{L^2(\Omega)}-2k^{-1}(\bar{\partial}_t \rho^{1/2}, \varphi_{{\m}} )_{L^2(\Omega)}\\
    &+(2k^{-1}(\bar{\partial}_t u^{1/2}-u_1)-u_{tt}^{1/2},\varphi_{{\m}})_{L^2(\Omega)} +b_{\pw}(u_{\m}^{1/2},\varphi_{{\m}},v_{\m}^{1/2})-b(u^{1/2},J\varphi_{{\m}},v^{1/2})-\frac{k^2}{4}b(\bar{\partial}_t u^{1/2},J\varphi_{{\m}},\bar{\partial}_t v^{1/2}). 
\end{align*}
From \eqref{app_at_0} and \eqref{P1 ritz_projection} , $u_{\m}^0=\mathcal{R}_{\m}u_0$. This and   \eqref{p4-split1} imply $\zeta^0=0$. The choice  $\varphi_{{\m}}=\zeta^{1/2}$ and $2\zeta^{1/2}=\zeta^{1}+\zeta^{0}=\zeta^{1}-\zeta^{0}={k}\bar{\partial}_t\zeta^{1/2}$ and \eqref{P1 mesh_norm} in the above displayed equation reveals
\begin{align}
 \|\bar{\partial}_t\zeta^{1/2}\|^2+\trinl \zeta^{1/2}\trinl_{\pw}^2 &= (f^{1/2}-u_{tt}^{1/2},(J-I)\zeta^{1/2} )_{L^2(\Omega)}+(\bar{\partial}_t u^{1/2}-u_1-\frac{k}{2}u_{tt}^{1/2}-\bar{\partial}_t \rho^{1/2},\bar{\partial}_t\zeta^{1/2})_{L^2(\Omega)}\nonumber\\
    & \quad+b_{\pw}(u_{\m}^{1/2},\zeta^{1/2},v_{\m}^{1/2})-b(u^{1/2},J\zeta^{1/2},v^{1/2})-\frac{k^2}{4}b(\bar{\partial}_t u^{1/2},J\zeta^{1/2},\bar{\partial}_t v^{1/2}). \label{ib-zeta}
\end{align}
Choose the test function in  \eqref{P1 weak_form2} as $J \psi_{{\m}}$ for any $\psi_{{\m}} \in \m(\cal{T})$ and consider the average of \eqref{P1 weak_form2} at $t=t_0$ and $t_1$ to obtain  $ a(v^{1/2},J\psi_{{\m}})=\frac{1}{4} \big(b(u^{1},u^1,J\psi_{{\m}})+b(u_0,u_0,J\psi_{{\m}})\big) $. Subtract this from  $a_{\pw}(v_{\m}^{1/2},\psi_{{\m}})=\frac{1}{4} \big(b_{\pw}(u_{\m}^{1},u_{\m}^1,\psi_{{\m}})+b_{\pw}(u_{\m}^{0},u_{\m}^0,\psi_{{\m}})\big)$ for all $\psi_{{\m}} \in \m(\cal{T})$  (obtained from a combination of  \eqref{app_at_0} and \eqref{P1 ic_2})  to show
\begin{align}
   & a(v^{1/2},J\psi_{{\m}})-a_{\pw}(v_{\m}^{1/2},\psi_{{\m}})\nonumber
   \\
    &=\frac{1}{4} \big(b(u^{1},u^1,J\psi_{{\m}})+b(u_0,u_0,J\psi_{{\m}}) -b_{\pw}(u_{\m}^{1},u_{\m}^1,\psi_{{\m}})-b_{\pw}(u_{\m}^{0},u_{\m}^0,\psi_{{\m}})\big).\label{ee2}
\end{align}
 The definition of ${\cal{R}}_{\m}$ once again shows $a(v^{1/2},J\psi_{{\m}})=a_{\pw}({\cal{R}}_{\m}v^{1/2},\psi_{{\m}})$. The identity $b_{\pw}(\varphi^{1},\varphi^1,\chi)+b_{\pw}(\varphi^0,\varphi^0,\chi)=4 b_{\pw}(\varphi^{1/2},\varphi^{1/2},\chi)-2b_{\pw}(\varphi^0,\varphi^1,\chi)$ follows from the trilinearity of \(b_{\pw}(\bullet,\bullet,\bullet)\)   and its symmetry in first two components. A combination of these steps for   $\psi_{{\m}}=\theta^{1/2}$ in \eqref{ee2} with the definitions $\theta^{1/2}$ from \eqref{p4-split2} and $\trinl \bullet \trinr_\pw$ from \eqref{P1 mesh_norm}, eventually shows that 
\begin{align}
  \trinl \theta^{1/2}\trinl_{\pw}^2&= b(u^{1/2},u^{1/2},J\theta^{1/2})-\frac{1}{2} b(u_0,u^1,J\theta^{1/2}) -b_{\pw}(u_{\m}^{1/2},u_{\m}^{1/2},\theta^{1/2})+ \frac{1}{2}b_{\pw}(u_{\m}^{0},u_{\m}^1,\theta^{1/2}).
 \label{ib-theta}
\end{align}
A combination of \eqref{ib-zeta} and \eqref{ib-theta} with a re-grouping of the nonlinear and truncation term, lead to
\begin{align}
& \|\bar{\partial}_t\zeta^{1/2}\|^2+\trinl\zeta^{1/2}\trinl_{\pw}^2+\trinl\theta^{1/2}\trinl_{\pw}^2\nonumber\\
&= \left[(f^{1/2}-u_{tt}^{1/2},(J-I)\zeta^{1/2} )_{L^2(\Omega)}+(\bar{\partial}_t u^{1/2}-u_1-\frac{k}{2}u_{tt}^{1/2}-\bar{\partial}_t \rho^{1/2},\bar{\partial}_t\zeta^{1/2})_{L^2(\Omega)}\right] \nonumber\\
    &\quad+ \left[ b_{\pw}(u_{\m}^{1/2},\zeta^{1/2},v_{\m}^{1/2})-b(u^{1/2},J\zeta^{1/2},v^{1/2}) \right]+ \left[ b(u^{1/2},u^{1/2},J\theta^{1/2})-b_{\pw}(u_{\m}^{1/2},u_{\m}^{1/2},\theta^{1/2}) \right]\nonumber\\
    &\quad + \frac{1}{2} \left[ b_{\pw}(u_{\m}^{0},u_{\m}^1,\theta^{1/2})-b(u_0,u^1,J\theta^{1/2}) \right] -\frac{k^2}{4}b(\bar{\partial}_t u^{1/2},J\zeta^{1/2},\bar{\partial}_t v^{1/2}).\label{all_ini}
\end{align}
The first  term on the right-hand side is linear contributions, the last term is the nonlinear truncation term, and the remaining terms involve the trilinear terms. 

 \medskip \noindent{\it Step 2. Bound for linear terms in \eqref{all_ini}.}  The three inequalities below come from Cauchy-Schwarz inequality (applied in  all the three), the definition of $\bar{\partial}_t \rho^{1/2}$ followed by Lemma~\ref{bhcompanion_lem}$(e)$, Lemma~\ref{trunc-lem}$(a)$, and  Lemma~\ref{P1 ritz_lemma},  respectively:
    \begin{align}
    & (f^{1/2}-u_{tt}^{1/2},(J-I)\zeta^{1/2} )_{L^2(\Omega)} \le \|f^{1/2}-u_{tt}^{1/2}\|\|(J-I)\zeta^{1/2} \|\le \constref{ca}h^2\|f^{1/2}-u_{tt}^{1/2}\|\trinl\zeta^{1/2} \trinl_{\pw},\nonumber \\ 
    &k(2k^{-1}(\bar{\partial}_t u^{1/2}-u_1)-u_{tt}^{1/2},\bar{\partial}_t\zeta^{1/2})_{L^2(\Omega)} \le k\|2k^{-1}(\bar{\partial}_t u^{1/2}-u_1)-u_{tt}^{1/2}\|\|\bar{\partial}_t\zeta^{1/2}\| \le k^2\|u_{ttt}\|_{L^\infty(0,t_1;L^2(\Omega))}\|\bar{\partial}_t\zeta^{1/2}\|, \nonumber \\
    &|(\bar{\partial}_t \rho^{1/2}, \bar{\partial}_t\zeta^{1/2})_{L^2(\Omega)} |\le \|\bar{\partial}_t \rho^{1/2}\|\|\bar{\partial}_t\zeta^{1/2}\| =\frac{1}{k}\|\int_{0}^{t_1} \rho_t \dt\|\|\bar{\partial}_t\zeta^{1/2}\|\le \constref{car}h^{2\sigma}\|u_t\|_{L^\infty(0,t_1;H^{2+\sigma}(\Omega))}\|\bar{\partial}_t\zeta^{1/2}\| .\nonumber
\end{align}
We apply Young's inequality to each of the three inequalities above with the following choices: $a= \constref{ca}h^2\|f^{1/2}-u_{tt}^{1/2}\|$, $b=\trinl\zeta^{1/2} \trinl_{\pw}$, $\epsilon=8$ (for the first term); $a=k\|u_{ttt}\|_{L^\infty(0,t_1;L^2(\Omega))}$, $b=\|\bar{\partial}_t\zeta^{1/2}\|$, $\epsilon=1$ (for the second term); $a=\constref{car}h^{2\sigma}\|u_t\|_{L^\infty(0,t_1;H^{2+\sigma}(\Omega))}$, $b=\|\bar{\partial}_t\zeta^{1/2}\|$, $\epsilon=2$ (for the third term). This leads to
\begin{align}
   &(f^{1/2}-u_{tt}^{1/2},(J-I)\zeta^{1/2} )_{L^2(\Omega)}+(\bar{\partial}_t u^{1/2}-u_1-\frac{k}{2}u_{tt}^{1/2}-\bar{\partial}_t \rho^{1/2},\bar{\partial}_t\zeta^{1/2})_{L^2(\Omega)}\nonumber\\
    & \qquad \le  {\mathcal C}(f,u_t,u_{tt}, u_{ttt})  (h^{4\sigma} + k^4) +\frac{1}{16}\trinl \zeta^{1/2} \trinl_{\pw}^2+\frac{1}{2}\|\bar{\partial}_t\zeta^{1/2}\|^2.  \label{ini1} 
   \end{align}
   In the sequel, the notation ${\mathcal C} $ denotes a generic term that depends on $f, u, v, u_t, v_t, u_{tt}, v_{tt}, u_{ttt}$ etc. with the dependence pointed out explicitly, but is independent of the spatial and temporal discretization parameters.
   
\medskip \noindent{\it Step 3. Rearrangement of nonlinear terms in \eqref{all_ini}.} The nonlinear terms on the right-hand side of \eqref{all_ini} require a clever manipulation to introduce  errors that can be controlled. The trilinearity of  $b(\bullet,\bullet,\bullet)$ and  $b_{\pw}(\bullet,\bullet,\bullet)$, their symmetry with respect to the first and second variables,  \eqref{p4-split1}-\eqref{p4-split2}, and elementary algebra leads to  
\begin{align*}
    &b_{\pw}(u_{\m}^{1/2},\zeta^{1/2},v_{\m}^{1/2})-b(u^{1/2},J\zeta^{1/2},v^{1/2}) =b_{\pw}(u^{1/2},(I-J)\zeta^{1/2},v^{1/2})\\
    &-b_{\pw}(\zeta^{1/2},\zeta^{1/2},v_{\m}^{1/2})-b_{\pw}(u^{1/2},\zeta^{1/2},\theta^{1/2})-b_{\pw}(\rho^{1/2},\zeta^{1/2},v_{\m}^{1/2})-b_{\pw}(u^{1/2},\zeta^{1/2},\chi^{1/2}).
\end{align*}
Analogous arguments also imply that 
\begin{align*}
&b(u^{1/2},u^{1/2},J\theta^{1/2})-b_{\pw}(u_{\m}^{1/2},u_{\m}^{1/2},\theta^{1/2})=b_{\pw}(u^{1/2},u^{1/2},(J-I)\theta^{1/2})+b_{\pw}(u^{1/2},\zeta^{1/2},\theta^{1/2})\\
     &\qquad \qquad \quad\quad\quad+b_{\pw}(\zeta^{1/2},u_{\m}^{1/2},\theta^{1/2})+b_{\pw}(u^{1/2}
,\rho^{1/2},\theta^{1/2})+b_{\pw}(\rho^{1/2},u_{\m}^{1/2},\theta^{1/2}),\\
 & \frac{1}{2} \left[b_{\pw}(u_{\m}^{0},u_{\m}^1,\theta^{1/2})-b(u_0,u^1,J\theta^{1/2})\right]=-\frac{1}{2}b_{\pw}(u_0,u^1,(J-I)\theta^{1/2})
     -b_{\pw}(u^0_{\m},\zeta^{1/2},\theta^{1/2})\\
     & \quad\quad\quad -\frac{1}{2}b_{\pw}(u^0_{\m},\rho^{1},\theta^{1/2})-\frac{1}{2}b_{\pw}(\rho^{0},u^1,\theta^{1/2})
\end{align*}
with $u_{\m}^{0} ={\cal R}_{\m}u_{0}$ from \eqref{app_at_0}  (which shows $\zeta^0=0$ and $\zeta^1=2\zeta^{1/2}$ from \eqref{p4-split1}) in the last identity. Since $b_{\pw}(\zeta^{1/2},u_{\m}^{1/2},\theta^{1/2}) =b_{\pw}(u_{\m}^{1/2}, \zeta^{1/2},\theta^{1/2}) $, two terms appearing on the right-hand side of the above two identities can be combined  as  $b_{\pw}(\zeta^{1/2},u_{\m}^{1/2},\theta^{1/2})-b_{\pw}(u^0_{\m},\zeta^{1/2},\theta^{1/2}) 
=\frac{1}{2}b_{\pw}(u_{\m}^1-u^0_{\m},\zeta^{1/2},\theta^{1/2}).$ 
Since $b_{\pw}(u^{1/2},\zeta^{1/2},\theta^{1/2})$ cancels out, the 
nonlinear terms on the right-hand side of \eqref{all_ini} rearrange as 
\begin{align}
   & -b_{\pw}(\zeta^{1/2},\zeta^{1/2},v_{\m}^{1/2})+\frac{1}{2}b_{\pw}(u_{\m}^1-u^0_{\m},\zeta^{1/2},\theta^{1/2}) +b_{\pw}(u^{1/2},(I-J)\zeta^{1/2},v^{1/2})+b_{\pw}(u^{1/2},u^{1/2},(J-I)\theta^{1/2})\nonumber\\
   &-\frac{1}{2}b_{\pw}(u_0,u^1,(J-I)\theta^{1/2})-b_{\pw}(\rho^{1/2},\zeta^{1/2},v_{\m}^{1/2})-b_{\pw}(u^{1/2},\zeta^{1/2},\chi^{1/2})+ b_{\pw}(u^{1/2}
,\rho^{1/2},\theta^{1/2})\nonumber\\
  & +b(\rho^{1/2},u_{\m}^{1/2},\theta^{1/2})-\frac{1}{2}b_{\pw}(u^0_{\m},\rho^{1},\theta^{1/2})-\frac{1}{2}b(\rho^{0},u^1,\theta^{1/2})
  :=T_1+T_2+\cdots+T_{11}.\label{non-all}
\end{align}
\medskip \noindent{\it Step 4. Bound for  nonlinear terms in \eqref{non-all}.}
Applying \eqref{Bh-bdd} (twice) shows $T_1+T_2 := -b_{\pw}(\zeta^{1/2},\zeta^{1/2},v_{\m}^{1/2})+\frac{1}{2}b_{\pw}(u_{\m}^{1}-u_{\m}^{0},\zeta^{1/2},\theta^{1/2})$ can be bounded viz.
\begin{align*}
   T_1+T_2
   & \le |b_{\pw}(\zeta^{1/2},\zeta^{1/2},v_{\m}^{1/2})|+\frac{1}{2}|b_{\pw}(u_{\m}^{1}-u_{\m}^{0},\zeta^{1/2},\theta^{1/2})| \nonumber\\
    &  \le C_{\rm dS}\trinl \zeta^{1/2}\trinl_{\pw}^2 \trinl v_{\m}^{1/2}\trinl_{\pw}+ \frac{1}{2}C_{\rm dS}\trinl u_{\m}^{1}-u_{\m}^{0}\trinl_{\pw}\trinl \zeta^{1/2} \trinl_{\pw}\trinl \theta^{1/2} \trinl_{\pw}\nonumber\\
    &  \le C_{\rm dS}\big( \trinl v_{\m}^{1/2}\trinl_{\pw}+\frac{1}{4}C_{\rm dS}\trinl u_{\m}^{1}-u_{\m}^{0} \trinl_{\pw}^2\big)\trinl \zeta^{1/2} \trinl_{\pw}^2+\frac{1}{4}\trinl \theta^{1/2} \trinl_{\pw}^2 
\end{align*}
with the Young's inequality (taking $a={\frac{1}{2}C_{\rm dS}}\trinl u_{\m}^{1}-u_{\m}^{0}  \trinl_{\pw}\trinl \zeta^{1/2} \trinl_{\pw},b=\trinl \theta^{1/2} \trinl_{\pw}, \epsilon=2$)  for the second term in the penultimate step. 
This, \eqref{sm} and the smallness assumption $ {\rm M}_{2}(u_0,u_1,f,k) \le 1$ from Theorem \ref{TH-WELLPOSED-INI}$(d)$ lead to 
\begin{align*}
    T_1+T_2
    &\le \frac{3}{4} {\rm M}_{2}(u_0,u_1,f,k)\trinl \zeta^{1/2} \trinl_{\pw}^2 +\frac{1}{4}\trinl \theta^{1/2} \trinl_{\pw}^2 \le   \frac{3}{4}\trinl \zeta^{1/2} \trinl_{\pw}^2 +\frac{1}{4}\trinl \theta^{1/2} \trinl_{\pw}^2 .
\end{align*}
Lemma~\ref{IB}$(e)$,  Lemma~\ref{axul}$(a)$, and  Young's inequality (skipping details for brevity) lead to
\begin{align*}
   T_3:=b_{\pw}(u^{1/2},(I-J)\zeta^{1/2},v^{1/2})&
   \le   \constref{C3.4b} h^\sigma\norm{u}_{C([0,t_1];H^{2+\sigma}(\Omega))}\norm{v}_{C([0,t_1];H^2(\Omega))}\trinl \zeta^{1/2} \trinl _{\pw} 
  \le { {\mathcal C}(u,v) h^{2\sigma}+\frac{1}{64}\trinl \zeta^{1/2} \trinl_{\pw}^2}.
  \end{align*}
  Analogous arguments with Lemma~\ref{IB}$(d)$   reveal
\begin{align*}
    T_4+T_5 &:=b_{\pw}(u^{1/2},u^{1/2},(J-I)\theta^{1/2})  
   -\frac{1}{2}b_{\pw}(u_0,u^1,(I-J)\theta^{1/2}) \\
   &\le |b_{\pw}(u^{1/2},u^{1/2},(J-I)\theta^{1/2})| 
   +\frac{1}{2}|b_{\pw}(u_0,u^1,(I-J)\theta^{1/2})|\\
    &\le \constref{C3.4a} h\big(\norm{u}_{C([0,t_1];H^2(\Omega))}^2+\frac{1}{2}\norm{u_0}_{H^2(\Omega)}\norm{u}_{C([0,t_1];H^2(\Omega))}\big)\trinl \theta^{1/2}\trinl_{\pw} \le {\mathcal C}(u) h^{2}+\frac{1}{8}\trinl \theta^{1/2}\trinl_{\pw}^2. 
   \end{align*}
Next, we utilize \eqref{Bh-bdd}, Lemma~\ref{axul_ritz}$(a)$ followed by Lemma~\ref{axul}$(a)$ to obtain
\begin{align*}
&T_6+T_7:=-b_{\pw}(\rho^{1/2},\zeta^{1/2},v_{\m}^{1/2})-b_{\pw}(u^{1/2},\zeta^{1/2},\chi^{1/2})\le |b_{\pw}(\rho^{1/2},\zeta^{1/2},v_{\m}^{1/2})|+|b_{\pw}(u^{1/2},\zeta^{1/2},\chi^{1/2})|\\
&\le C_{\rm dS}\constref{car}h^\sigma \big(\| u\|_{C({[0,t_1];H^{2+\sigma}(\Omega))}}\trinl v_{\m}^{1/2} \trinl_{\pw}+\|u\|_{C({[0,t_1];H^{2}(\Omega))}}\| v\|_{C({[0,t_1];H^{2+\sigma}(\Omega))}}\big)\trinl \zeta^{1/2} \trinl_{\pw}\le {\mathcal C}(u) h^{2\sigma}+\frac{1}{32}\trinl \zeta^{1/2}\trinl_{\pw}^2
\end{align*}
with Young's inequality, the bound  $\trinl v_{\m}^{1/2} \trinl_{\pw} \le  \frac{3}{4}{\rm M}_{2}(u_0,u_1,f,k) \le \frac{3}{4}$  from \eqref{sm}, and the smallness assumption from Theorem \ref{TH-WELLPOSED-INI}$(d)$ in the last step.
Similarly, we infer
\begin{align*}
    T_8&+T_9+T_{10}+T_{11}\le |b_{\pw}(u^{1/2},\rho^{1/2},\theta^{1/2})|+|b(\rho^{1/2},u_{\m}^{1/2},\theta^{1/2})|+|\frac{1}{2}b_{\pw}(u^0_{\m},\rho^{1},\theta^{1/2})|+\frac{1}{2}|b(\rho^{0},u^1,\theta^{1/2})|\\
    &\le C_{\rm dS}\constref{car}\| u\|_{C({[0,t_1];H^{2+\sigma}(\Omega))}}h^\sigma\big(\|u\|_{C({[0,t_1];H^{2}(\Omega))}} +\trinl u_{\m}^{1/2} \trinl_{\pw}+\frac{1}{2}\trinl u^0_{\m} \trinl_{\pw}+\frac{1}{2}\|u\|_{C({[0,t_1];H^{2}(\Omega))}}  \big)\trinl  \theta^{1/2}\trinl_{\pw}\\
 & \le {\mathcal C}(u) h^{2\sigma}+\frac{1}{2}\trinl  \theta^{1/2}\trinl_{\pw}^2,   
\end{align*}
where we have utilized  $\trinl u_{\m}^{1/2} \trinl_{\pw} \le \frac{1}{2}\big(\trinl u_{\m}^{0} \trinl_{\pw}+\trinl u_{\m}^{1} \trinl_{\pw}\big)$ and the bound
from \eqref{sm} in the last step. 

\medskip \noindent{\it Step 5. Control of the truncation term.}    Apply Lemma~\ref{IB}$(a)$ and  Lemma~\ref{axul}$(c)$ to observe 
\begin{align*}
  \frac{k^2}{4}|b(\bar{\partial}_t u^{1/2},J\zeta^{1/2},\bar{\partial}_t v^{1/2})|
 &\le\constref{ccj}\frac{k^2}{4}\|u_t\|_{L^\infty(0,t_1;H^2(\Omega))}\|v_t\|_{L^\infty(0,t_1;H^2(\Omega))}\trinl \zeta^{1/2} \trinl_{\pw}
 \le  { {\mathcal C}(u_t,v_{t})  k^4+\frac{1}{64}\trinl \zeta^{1/2} \trinl_{\pw}^2}
\end{align*}
with Young's inequality ($a= \frac{1}{8}\constref{ccj}k^2\|u_t\|_{L^\infty(0,t_1;H^2(\Omega))}\|v_t\|_{L^\infty(0,t_1;H^2(\Omega))}$, $b=\trinl \zeta^{1/2} \trinl_{\pw}$, and $\epsilon=32$) in the last step. 

    \medskip \noindent{\it Step 6. Consolidation.}
A Combination of  \eqref{ini1}, the bounds of the terms in  \eqref{non-all} from Step 4, Step 5, and \eqref{all_ini} lead to
\begin{align}
   \frac{1}{2}\|\bar{\partial}_t\zeta^{1/2}\|^2+\frac{1}{8}\trinl \zeta^{1/2} \trinl_{\pw}^2+ \frac{1}{8}\trinl \theta^{1/2} \trinl_{\pw}^2& 
   \le { {\mathcal C}(u_0,f,u,v,u_t,v_t,u_{tt}, u_{ttt})  (h^{2\sigma} + k^4) }.
   \label{eqn-ni}
\end{align}
This and triangle inequalities show $\norm{\bar{\partial}_t(u^{1/2}-u_{\m}^{1/2})}+ \trinl u^{1/2}-u_{\m}^{1/2}\trinl_{\pw}+ \trinl v^{1/2}-v_{\m}^{1/2} \trinl_{\pw}\le  \norm{\bar{\partial}_t\rho^{1/2}}+ \trinl \rho^{1/2}\trinl_{\pw}+ \trinl \chi^{1/2} \trinl_{\pw} + \norm{\bar{\partial}_t\zeta^{1/2}}+ \trinl \zeta^{1/2}\trinl_{\pw}+ \trinl\theta^{1/2} \trinl_{\pw}$. Using Lemma~\ref{P1 ritz_lemma} for the first three terms and \eqref{eqn-ni} to the remaining terms, we conclude the proof.
\end{proof}

\section{Proofs of results for fully discrete scheme}\label{sect-fully-full}
{This section presents the proofs of the well-posedness result in 
Theorem~\ref{TH-WELLPOSED-FINAL} and the a priori error estimates in 
Theorem~\ref{TH-ERROR-FINAL} for the fully discrete scheme~\eqref{P4fully}. 
Subsection~\ref{subsec-wellfull} establishes the existence, uniqueness, and 
stability of the discrete solution $(u_{\m}^{n+1}, v_{\m}^{n+1})$ via an 
inductive argument, where the base case is established in 
Theorem~\ref{TH-WELLPOSED-INI}. Subsection~\ref{apriori-sec} derives the 
optimal error estimates in the energy norm, achieving optimal convergence 
in space and quadratic convergence in time, building on the error equation, 
approximation properties of the Ritz projection, and the truncation error 
bounds from Lemma~\ref{trunc-lem}.} 
\subsection{Well-posedness of fully discrete scheme}\label{subsec-wellfull}
This subsection presents the proof of Theorem~\ref{TH-WELLPOSED-FINAL}. 
We begin by proving Theorem~\ref{TH-WELLPOSED-FINAL}$(a)$, which establishes that any solution 
$(u_{\m}^{n+1}, v_{\m}^{n+1})$ to~\eqref{P4fully}, if it exists, 
satisfies an a priori bound via a discrete energy estimate, with the 
nonlinear identity~\eqref{motiv-final} as the key ingredient. This result reveals that the scheme \eqref{P4fully} preserves the energy balance property at a discrete level and ensures its unconditional stability.
\begin{proof}[{\textbf{{Proof of Theorem~\ref{TH-WELLPOSED-FINAL}$(a)$}}}]
 We establish the a priori stability bound~\eqref{stabf1} for 
the fully discrete scheme~\eqref{P4fully} via a discrete energy argument, 
with the nonlinear identity~\eqref{motiv-final} and a discrete 
Gr{o}nwall lemma as the key ingredients.\\
Choose $\varphi_{\m}=u_{\m}^{n+1}-u_{\m}^{n-1}=k(\bar{\partial}_t u_{\m}^{n+1/2}+\bar{\partial}_t u_{\m}^{n-1/2})=2(u_{\m}^{n+1/2}-u_{\m}^{n+1/2})=2k\delta_t u_{\m}^{n}$ \big(from \eqref{dtu}\big) in \eqref{P4 fully_discrete1} to obtain
 \begin{align}
 &k(\bar{\partial}^2_t u_{\m}^{n},\bar{\partial}_t u_{\m}^{n+1/2}+\bar{\partial}_t u_{\m}^{n-1/2})_{L^2(\Omega)} + 2a_{\pw} (u_{\m}^{n,1/4},u_{\m}^{n+1/2}-u_{\m}^{n-1/2})\nonumber\\
 &+2kb_{\pw}(u_{\m}^{n,1/4},\delta_t u_{\m}^{n},v_{\m}^{n,1/4})=k(f^{n,1/4},\bar{\partial}_t u_{\m}^{n+1/2} +\bar{\partial}_t u_{\m}^{n-1/2} )_{L^2(\Omega)}.    \label{s11}
\end{align}
Utilize $\bar{\partial}^2_t u_{\m}^{n}=k^{-1}(\bar{\partial}_t u_{\m}^{n+1/2}-\bar{\partial}_t u_{\m}^{n-1/2})$ from \eqref{varphi2} and  $u_{\m}^{n,1/4}=\frac{1}{2}(u_{\m}^{n+1/2}+u_{\m}^{n+1/2})$ from \eqref{un14}, linearity, and symmetry of $(\bullet,\bullet)_{L^2(\Omega)}$ and $a_{\pw}(\bullet,\bullet)$  to show
\begin{align*}
    &k(\bar{\partial}^2_t u_{\m}^{n},\bar{\partial}_t u_{\m}^{n+1/2}+\bar{\partial}_t u_{\m}^{n-1/2})_{L^2(\Omega)}=(\bar{\partial}_t u_{\m}^{n+1/2}-\bar{\partial}_t u_{\m}^{n-1/2},\bar{\partial}_t u_{\m}^{n+1/2}+\bar{\partial}_t u_{\m}^{n-1/2})_{L^2(\Omega)} = \norm{\bar{\partial}_t u_{\m}^{n+1/2}}^2- \norm{\bar{\partial}_t u_{\m}^{n-1/2}}^2,\\
   & 2a_{\pw} (u_{\m}^{n,1/4},u_{\m}^{n+1/2}-u_{\m}^{n-1/2})=a_{\pw} (u_{\m}^{n+1/2}+u_{\m}^{n-1/2},u_{\m}^{n+1/2}-u_{\m}^{n-1/2})=\trinl u_{\m}^{n+1/2} \trinl_{\pw}^2- \trinl u_{\m}^{n-1/2} \trinl_{\pw}^2.
\end{align*}
A combination of these identities and \eqref{motiv-final} in  \eqref{s11} leads to
 \begin{align*}
     &\norm{\bar{\partial}_t u_{\m}^{n+1/2}}^2- \norm{\bar{\partial}_t u_{\m}^{n-1/2}}^2+ \trinl u_{\m}^{n+1/2} \trinl_{\pw}^2- \trinl u_{\m}^{n-1/2} \trinl_{\pw}^2\\
    &+\trinl v_{\m}^{n+1/2} \trinl_{\pw}^2- \trinl v_{\m}^{n-1/2} \trinl_{\pw}^2=k (f^{n,1/4},\bar{\partial}_tu_{\m}^{n+1/2}+\bar{\partial}_tu_{\m}^{n-1/2})_{L^2(\Omega)}.
 \end{align*} 
We sum over $n = 1, 2, \ldots, m$ and utilize the telescopic behavior of the left-hand side terms to derive
  \begin{align}  
 \norm{\bar{\partial}_t u_{\m}^{m+1/2}}^2 &+ \trinl u_{\m}^{m+1/2} \trinl_{\pw}^2+\trinl v_{\m}^{m+1/2} \trinl_{\pw}^2 =\norm{\bar{\partial}_t u_{\m}^{1/2}}^2 + \trinl u_{\m}^{1/2} \trinl_{\pw}^2+  \trinl v_{\m}^{1/2} \trinl_{\pw}^2 \nonumber\\
 &
 +k\sum_{n=1}^{m}(f^{n,1/4},\bar{\partial}_tu_{\m}^{n+1/2}+\bar{\partial}_tu_{\m}^{n-1/2})_{L^2(\Omega)}  \text{ for all }1 \le m \le N-1. \label{f1}
 \end{align} 
An application of Cauchy-Schwarz and Young's inequalities with $a=\norm{f^{n,1/4}}, \;b=\norm{\bar{\partial}_t u_{\m}^{n+1/2}+\bar{\partial}_t u_{\m}^{n-1/2}},\;\epsilon=2T$ followed by a
summation from  $n=1$ to $m$, and  $\norm{\bar{\partial}_tu_{\m}^{n+1/2}+\bar{\partial}_t u_{\m}^{n-1/2}} ^2 \le 2\norm{\bar{\partial}_tu_{\m}^{n+1/2}}^2+2\norm{\bar{\partial}_t u_{\m}^{n-1/2}}  $ from triangle inequality lead to
\begin{align} 
 k\sum_{n=1}^{m}(f^{n,1/4},\bar{\partial}_tu_{\m}^{n+1/2}+\bar{\partial}_t u_{\m}^{n-1/2})_{L^2(\Omega)}
  &\le T{k} \sum_{n=1}^{m}\norm{f^{n,1/4}}^2+\frac{1}{2} \norm{\bar{\partial}_t u_{\m}^{m+1/2}} ^2+\frac{k}{T} \sum_{n=0}^{m}\norm{\bar{\partial}_t u_{\m}^{n+1/2}} ^2\label{f2}
\end{align} 
with \(\frac{k}{T} \le 1\) in  the second term on the right-hand side of \eqref{f2}. 
A combination of \eqref{f1}-\eqref{f2} reads 
\begin{align*}
   & \frac{1}{2}\norm{\bar{\partial}_t u_{\m}^{m+1/2}}^2+ \trinl u_{\m}^{m+1/2}\trinl_{\pw}^2+ \trinl v_{\m}^{m+1/2}\trinl_{\pw}^2 \nonumber\\
    &
    \le \norm{\bar{\partial}_t u_{\m}^{1/2}}^2+\trinl u_{\m}^{1/2}\trinl_{\pw}^2+\trinl v_{\m}^{1/2}\trinl_{\pw}^2+T{k} \sum_{n=1}^{m}\norm{f^{n,1/4}}^2+\mu \sum_{n=0}^{m-1} \frac{1}{2}\norm{\bar{\partial}_t u_{\m}^{n+1/2}} ^2
 \end{align*} 
 with $\mu= 2\frac{k}{T}$. Lemma~\ref{P1 d-gronwall} with $a_m=\frac{1}{2}\norm{\bar{\partial}_t u_{\m}^{m+1/2}}^2 $, $b_m=\trinl u_{\m}^{m+1/2}\trinl_{\pw}^2+ \trinl v_{\m}^{m+1/2}\trinl_{\pw}^2 $, and  $c_m= \norm{\bar{\partial}_t u_{\m}^{1/2}}^2+\trinl u_{\m}^{1/2}\trinl_{\pw}^2+\trinl v_{\m}^{1/2}\trinl_{\pw}^2$ (resp. $c_m= \norm{\bar{\partial}_t u_{\m}^{1/2}}^2+\trinl u_{\m}^{1/2}\trinl_{\pw}^2+\trinl v_{\m}^{1/2}\trinl_{\pw}^2+T{k} \sum_{n=1}^{m}\norm{f^{n,1/4}}^2$) if $m=0$ (resp. if $m \ge 1$),   leads to
\begin{align}
    &\frac{1}{2}\norm{\bar{\partial}_t u_{\m}^{m+1/2}}^2+ \trinl u_{\m}^{m+1/2}\trinl_{\pw}^2+ \trinl v_{\m}^{m+1/2}\trinl_{\pw}^2 \le  
    e^2 \big(\norm{\bar{\partial}_t u_{\m}^{1/2}}^2+\trinl u_{\m}^{1/2}\trinl_{\pw}^2+\trinl v_{\m}^{1/2}\trinl_{\pw}^2+ T{k} \sum_{n=1}^{m}\norm{f^{n,1/4}}^2\big)\label{stabf}
 \end{align} 
 with  $e^{m\mu} = e^{2m\frac{k}{T}}  \le  e^2$ from $mk \le T$ in the last step.
Triangle inequalities and  
\eqref{stab-halfway}  result in
     \begin{align}
      &  \norm{\bar{\partial}_t u_{\m}^{1/2}}^2+\trinl u_{\m}^{1/2}\trinl_{\pw}^2+\trinl v_{\m}^{1/2}\trinl_{\pw}^2 \le \norm{\bar{\partial}_t u_{\m}^{1/2}}^2+ \frac{1}{2}\big(\trinl u_{\m}^1\trinl_{\pw}^2 + \trinl v_{\m}^1\trinl_{\pw}^2 + \trinl u_{\m}^0\trinl_{\pw}^2 +\trinl v_{\m}^0\trinl_{\pw}^2 \big)\nonumber\\
        &\le \|J\|^2\trinl u_0\trinl^2\big(1+\frac{1}{4}C_{\rm dS}^2\|J\|^2\trinl u_0\trinl^2\big)+2 \norm{u_1}^2+ \frac{k^2}{2}\norm{f}^2_{C([0,t_1];L^2(\Omega))} \label{smlll}
     \end{align}
  with the bounds for $\trinl u_{\m}^0\trinl_{\pw}$ and $\trinl v_{\m}^0\trinl_{\pw}$ from   Theorem~\ref{TH-WELLPOSED-INI}$(a)$ in the last step. 
Then \eqref{smlll}, \eqref{stabf}, and  elementary algebra lead to 
  \begin{align}
    &\frac{1}{2}\norm{\bar{\partial}_t u_{\m}^{m+1/2}}^2+ \trinl u_{\m}^{m+1/2}\trinl_{\pw}^2+ \trinl v_{\m}^{m+1/2}\trinl_{\pw}^2 \le \text{\v{M}}_{3}(u_0,u_1,f,k).\label{stabf1}
 \end{align} 
Some elementary manipulations conclude the proof.
\end{proof} 
{We next prove  Theorem~\ref{TH-WELLPOSED-FINAL}$(b)$, which establishes the existence of a solution 
$(u_{\m}^{n+1}, v_{\m}^{n+1})$ to the nonlinear coupled system~\eqref{P4fully} 
at step $t_{n+1}$. The proof proceeds by induction, assuming the existence 
of solutions at levels $t_{n}$ and $t_{n-1}$, where the base case is 
furnished by Theorem~\ref{TH-WELLPOSED-INI}. The existence is established 
via a corollary of Brouwer's fixed-point theorem (Theorem~\ref{Brouwer}), 
utilizing the a priori bounds from part~$(a)$.}
\begin{proof}[{\textbf{{Proof of Theorem~\ref{TH-WELLPOSED-FINAL}$(b)$}}}]
 The  existence follows as in Theorem~\ref{TH-WELLPOSED-INI}$(b)$ and hence we highlight the choices of the inner product and the mapping, which differ  because the initial discretization and fully discrete schemes are different.
  
 Choose ${\bf H}={{\bf M}({\cal T})}$ and $\displaystyle \text{for all } \boldsymbol{\Psi_{\m}}=({\psi_{{\m}1}},{\psi_{{\m}2}}) $ , $ \boldsymbol{\Phi_{\m}}=({\varphi_{{\m1}}},{\varphi_{{\m2}}}) \in {{\bf M}({\cal T})}$, consider the inner product in ${\bf H}$ defined as  $(\boldsymbol{\Psi_{\m}},\boldsymbol{\Phi_{\m}})_{\pw}:=(\psi_{{\m}1},{\varphi_{{\m1}}})_{L^2(\Omega)}+\frac{k^2}{4}\left( a_{\pw}(\psi_{{\m}1},{\varphi_{{\m1}}})+a_{\pw}(\psi_{{\m 2}},{\varphi_{{\m2}}})\right)$. Given  $\boldsymbol{\Psi_{\m}} =(\psi_{{\m}1},\psi_{{\m}2}) \in {{\bf M}({\cal T})}$ and motivated by  \eqref{P4fully}, we define  $\boldsymbol{\mathcal{G}}_{\boldsymbol{\Psi_{\m}}}: {{\bf M}({\cal T})} \rightarrow \mathbb{R}$ by 
\begin{align}
   \boldsymbol{\mathcal{G}}_{\boldsymbol{\Psi_{\m}}}(\boldsymbol{\Phi_{\m}})=
    (\boldsymbol{\Psi_{\m}},\boldsymbol{\Phi_{\m}})_{\pw}+\frac{k^2}{16}\big[b_{\pw}(\psi_{{\m}1}+4u_{\m}^{n-1/2},\varphi_{{\m1}},\psi_{{\m}2})&-b_{\pw}(\psi_{{\m}1}+{2}u_{\m}^{n-1/2}, \psi_{{\m}1}+2u_{\m}^{n-1/2},\varphi_{{\m2}})\big]\nonumber\\
       -k^2\big[(f^{n,1/4}+{2}{k^{-1}}\bar{\partial}_t u_{\m}^{n-1/2},\varphi_{{\m1}})_{L^2(\Omega)}-a_{\pw}(&u_{\m}^{n-1/2},\varphi_{{\m1}})+\frac{1}{2}a_{\pw}(v_{\m}^{n-1/2},\varphi_{{\m2}})\big].\label{gmap-final}
\end{align}
Note that 
$\boldsymbol{\mathcal{G}}_{\boldsymbol{\Psi_{\m}}}(\bullet)$ is linear and   bounded with respect to $\|\bullet \|_\pw$ and hence by the Riesz representation theorem,  we have that $\boldsymbol{\boldsymbol{\mathcal{L}}}: {{\bf M}({\cal T})} \rightarrow{{\bf M}({\cal T})}$ defined by 
  $
       (\boldsymbol{\mathcal{L}}(\boldsymbol{\Psi_{\m}}),\boldsymbol{\Phi_{\m}})_{\pw} =\boldsymbol{\mathcal{G}}_{\boldsymbol{\Psi_{\m}}}(\boldsymbol{\Phi_{\m}}) \;\text{is well defined and}
     $  continuous.
     To apply Theorem \ref{Brouwer} it remains to prove that there exists some ${\rm R}>0$ such that  $(\boldsymbol{\mathcal{L}}(\boldsymbol{\Psi_{\m}}),\boldsymbol{\Psi_{\m}})_{\pw} \ge 0$, for all $\boldsymbol{\Psi_{\m}} \in {{\bf M}({\cal T})}$ with  $\norm{\boldsymbol{\Psi_{\m}}}_{\pw}={\rm R}$. For this, we set 
\begin{equation}
   {\rm R}^2:=2k \bigl[{k^2}\norm{f}^2_{C([0,T];L^2(\Omega))}+4\norm{\bar{\partial}_t u_{\m}^{n-1/2}}^2+2\trinl u_{\m}^{n-1/2}\trinl_{\pw}^2 +2\trinl v_{\m}^{n-1/2}\trinl_{\pw}^2\bigr], \label{raduis-final}
\end{equation} 
and choose 
$\boldsymbol{\Psi_{\m}} \in {{\bf M}({\cal T})}$ such that  $\norm{\boldsymbol{\Psi_{\m}}}_{\pw}={\rm R}$. 
 Selecting  
$\boldsymbol{\Phi_{\m}}=\boldsymbol{\Psi_{\m}}$ in \eqref{gmap-final} we observe that the linearity, symmetry of $b_{\pw}(\bullet,\bullet\,\bullet)$ in first two components, and \eqref{P4fully}$(b)$ implies $b_{\pw}( \psi_{{\m}1}+4u_{\m}^{n-1/2}, \psi_{{\m}1},\psi_{{\m}2} )-b_{\pw}( \psi_{{\m}1}+2u_{\m}^{n-1/2}, \psi_{{\m}1}+2u_{\m}^{n-1/2},\psi_{{\m}2})=-4b_{\pw}( u_{\m}^{n-1/2}, u_{\m}^{n-1/2},\psi_{{\m}2})=-8a_{\pw}(v_{\m}^{n-1/2},\psi_{{\m}2})$. Hence  
\begin{align}
&(\boldsymbol{\mathcal{L}}(\boldsymbol{\Psi_{\m}})),\boldsymbol{\Psi_{\m}})_{\pw}=\boldsymbol{\mathcal{G}}_{\boldsymbol{\Psi_{\m}}}(\boldsymbol{\Psi_{\m}})\nonumber\\
     &=\norm{\boldsymbol{\Psi_{\m}}}^2_{\pw}-{k^2}\big[(f^{n,1/4}+{2}{k^{-1}}\bar{\partial}_t u_{\m}^{n-1/2},\varphi_{{\m1}})_{L^2(\Omega)}-a_{\pw}(u_{\m}^{n-1/2},\psi_{{\m}1})+a_{\pw}(v_{\m}^{n-1/2},\psi_{{\m}2})\big].\label{put}
\end{align}
Cauchy-Schwarz and Young's inequalities with {$a=k^2 \norm{f^{n,1/4}+{2}{k^{-1}}\bar{\partial}_t u_{\m}^{n-1/2}},b=\norm{\psi_{{\m}1}}$, $\epsilon=k^{2}$}, and  $\norm{f^{n,1/4}} \le \norm{f}_{C([0,T];L^2(\Omega))}$ from Lemma~\ref{axul}$(b)$ reveal 
{
\begin{align*}
    k^2(f^{n,1/4}&+{2}{k^{-1}}\bar{\partial}_t u_{\m}^{n-1/2},\psi_{{\m}1})_{L^2(\Omega)}
    \le k^2\|f^{n,1/4}+{2}{k^{-1}}\bar{\partial}_t u_{\m}^{n-1/2}\|\|\psi_{{\m}1}\|\\
    &\le \frac{k^4}{2}\|f^{n,1/4}+{2}{k^{-1}}\bar{\partial}_t u_{\m}^{n-1/2}\|^2+\frac{1}{2}\|\psi_{{\m}1}\|^2\le{k^4}\norm{f}^2_{C([0,T];L^2(\Omega))}+4k^2\norm{\bar{\partial}_t u_{\m}^{n-1/2}}^2 +\frac{1}{2}\|\psi_{{\m}1}\|^2.
\end{align*}}
The continuity of the bilinear form $a_{\pw}(\bullet,\bullet)$  and Young's inequality (applied twice) gives 
\begin{align*}
 | a_{\pw}(u_{\m}^{n-1/2},\psi_{{\m}1})|+|a_{\pw}(v_{\m}^{n-1/2},\psi_{{\m}2}) | 
  &\le 2\big(\trinl v_{\m}^{n-1/2}\trinl_{\pw}^2+\trinl u_{\m}^{n-1/2}\trinl_{\pw}^2\big)+\frac{1}{8}\big(\trinl \psi_{{\m}1}\trinl_{\pw}^2+\trinl \psi_{{\m}2}\trinl_{\pw}^2\big).
\end{align*}
Combine the last two displayed inequalities with 
 $\norm{\boldsymbol{\Psi_{\m}}}^2_{\pw}=\| \psi_{{\m}1} \|^2+\frac{k^2}{4}\big(\trinl \psi_{{\m}1} \trinl^2_{\pw}+\trinl \psi_{{\m}2} \trinl_{\pw}^2\big)={\rm R}^2$ and utilize   \eqref{raduis-final} to obtain
\begin{align*}
    &{k^2}\big[(f^{n,1/4}+{2}{k^{-1}}\bar{\partial}_t u_{\m}^{n-1/2},\varphi_{{\m1}})_{L^2(\Omega)}-a_{\pw}(u_{\m}^{n-1/2},\psi_{{\m}1})-a_{\pw}(v_{\m}^{n-1/2},\psi_{{\m}2})\big]\\
    & \le \frac{1}{2}{\rm R}^2+k^2\big({k^2}\norm{f}^2_{C([0,T];L^2(\Omega))}+4\norm{\bar{\partial}_t u_{\m}^{n-1/2}}^2 +2\trinl u_{\m}^{n-1/2}\trinl_{\pw}^2+2\trinl v_{\m}^{n-1/2}\trinl_{\pw}^2\big)= \frac{3}{4}{\rm R}^2.
\end{align*}
This and  \eqref{put}  along with the fact that $\norm{\boldsymbol{\Psi_{\m}}}_{\pw}={\rm R}$ leads to
$
(\boldsymbol{\mathcal{L}}(\boldsymbol{\Psi_{\m}})),\boldsymbol{\Psi_{\m}})_{\pw}
     \ge\frac{{\rm R}^2}{4}>0.$
Theorem \ref{Brouwer} shows that there exists $\boldsymbol{ \Psi^*_{\m}} =(\psi^*_{{\m}1}, \psi^*_{{\m}2})  \in {\bf M}({\cal T}) $ such that $\norm{\boldsymbol{ \Psi^*_{\m}}}_{\pw} \le {\rm R}$ and $ \boldsymbol{\mathcal{L}} (\boldsymbol{ \Psi^*_{\m}})=0$. Define $(u^{n+1}_{\m},v^{n+1}_{\m})=( \psi^*_{{\m}1}+u_{\m}^{n-1},\psi^*_{{\m}2}-2v_{\m}^{n}-v_{\m}^{n-1})$. This,
\eqref{gmap-final}, and the definition of $\boldsymbol{\mathcal{L}}$ show that 
$(u^{n+1}_{\m},v^{n+1}_{\m})$ 
solves the fully-discrete scheme \eqref{P4fully}.
\end{proof}
\noindent {We conclude this subsection with the proof of Theorem~\ref{TH-WELLPOSED-FINAL}$(c)$, which establishes the uniqueness of the 
solution $(u_{\m}^{n+1}, v_{\m}^{n+1})$ to~\eqref{P4fully} at step 
$t_{n+1}$. The proof again proceeds by induction, assuming uniqueness 
at levels $t_{n}$ and $t_{n-1}$, and employs the a priori bounds 
from part~$(a)$ and the smallness assumption of the given data.}
\begin{proof}[{\textbf{{Proof of Theorem~\ref{TH-WELLPOSED-FINAL}$(c)$}}}]
If possible, let $(u_{\m}^{n+1},v_{\m}^{n+1})$ and $(\hat{u}_{\m}^{n+1},\hat{v}_{\m}^{n+1})$ be two solutions of  \eqref{P4fully} at $t_{n+1}$ given that  $(u_{\m}^{n},v_{\m}^{n})$ (resp. $({u}_{\m}^{n-1},{v}_{\m}^{n-1})$) are unique solutions at $t_n$ (resp. $t_{n-1}$). Define ${e}_{\m}:=u_{\m}^{n+1}-\hat{u}_{\m}^{n+1} $ and $\bar{e}_{\m}: =v_{\m}^{n+1}-\hat{v}_{\m}^{n+1}$. With the definitions from \eqref{varphi1} and \eqref{varphi2}, \eqref{P4fully} implies 
\begin{align*}
    &\frac{1}{k^2}( {e}_{\m},\varphi_{\m})_{L^2(\Omega)} + \frac{1}{4}a_{\pw} ({e}_{\m},\varphi_{\m} )+\frac{1}{16}b_{\pw}(u_{\m}^{n+1}+2u_{\m}^{n}+u_{\m}^{n-1},\varphi_{\m},v_{\m}^{n+1}+2v_{\m}^{n}+v_{\m}^{n-1})\\ \nonumber
    &\qquad\qquad\quad\quad-\frac{1}{16}b_{\pw}(\hat{u}_{\m}^{n+1}+2u_{\m}^{n}+u_{\m}^{n-1},\varphi_{\m},\hat{v}_{\m}^{n+1}+2v_{\m}^{n}+v_{\m}^{n-1})= 0\;\text{ for all } \varphi_{\m} \in {\m}({\cal T}), \\
  &a_{\pw}(\bar{e}_{\m},\psi_{\m}) =\frac{1}{4}b_{\pw}(u_{\m}^{n+1}+u_{\m}^{n},u_{\m}^{n+1}+u_{\m}^{n},\psi_{\m})-\frac{1}{4}b_{\pw}(\hat{u}_{\m}^{n+1}+u_{\m}^{n},\hat{u}_{\m}^{n+1}+u_{\m}^{n},\psi_{\m}) \;\text{ for all }\psi_{\m} \in {\m}({\cal T}) .
    \end{align*}
   The test functions $\varphi_{\m}=16e_{\m}$ and $\psi_{\m}=4\bar{e}_{\m}$ in the above system of equation lead to
   \begin{subequations}\label{com}
    \begin{align}
    {16}{k^{-2}}\norm{{e}_{\m}}^2 + 4\trinl {e}_{\m}\trinl_{\pw}^2=-b_{\pw}(u_{\m}^{n+1}+2u_{\m}^{n}+u_{\m}^{n-1},{e}_{\m},v_{\m}^{n+1}+2v_{\m}^{n}+v_{\m}^{n-1})\nonumber\\ 
    +b_{\pw}(\hat{u}_{\m}^{n+1}+2u_{\m}^{n}+u_{\m}^{n-1},{e}_{\m},\hat{v}_{\m}^{n+1}+2v_{\m}^{n}+v_{\m}^{n-1}),\label{u-1} \\
  4\trinl \bar{e}_{\m}\trinl_{\pw}^2 =b_{\pw}(u_{\m}^{n+1}+u_{\m}^{n},u_{\m}^{n+1}+u_{\m}^{n},\bar{e}_{\m})-b_{\pw}(\hat{u}_{\m}^{n+1}+u_{\m}^{n},\hat{u}_{\m}^{n+1}+u_{\m}^{n},\bar{e}_{\m}) . \label{u-2}
    \end{align}\end{subequations}
The  linearity of $b_{\pw}(\bullet,\bullet,\bullet)$ in each component and its symmetry in the first two components reveal that the terms on the right-hand sides of \eqref{u-1} and \eqref{u-2} can be combined as 
 $2b_{\pw}({u}_{\m}^{n+1/2}-{u}_{\m}^{n-1/2},e_{\m},\bar{e}_{\m})-4b_{\pw}(e_{\m},e_{\m},v_{\m}^{n,1/4})$.
This, an addition of \eqref{u-1} and \eqref{u-2}, and the boundedness of $b_{\pw}(\bullet,\bullet,\bullet)$ from \eqref{Bh-bdd} leads to 
\begin{align*}
   & {8}{k^{-2}}\norm{{e}_{\m}}^2 + 2\trinl {e}_{\m}\trinl_{\pw}^2+2\trinl \bar{e}_{\m}\trinl_{\pw}^2
=b_{\pw}({u}_{\m}^{n+1/2}-{u}_{\m}^{n-1/2},e_{\m},\bar{e}_{\m})-2b_{\pw}(e_{\m},e_{\m},v_{\m}^{n,1/4})\nonumber\\
    &\qquad\le C_{\rm dS}\big(\trinl {u}_{\m}^{n+1/2}-{u}_{\m}^{n-1/2}\trinl_{\pw}\trinl e_{\m}\trinl_{\pw}\trinl \bar{e}_\m\trinl_{\pw}+\trinl {v}_{\m}^{n+1/2}+ {v}_{\m}^{n-1/2}\trinl_{\pw}\trinl e_{\m}\trinl_{\pw}^2\big). 
\end{align*}
A Young's inequality with $a=\trinl {u}_{\m}^{n+1/2}-{u}_{\m}^{n-1/2}\trinl_{\pw}\trinl e_{\m}\trinl_{\pw},b=\trinl \bar{e}_\m\trinl_{\pw}$, and $\epsilon=C_{\rm dS}/2$, and a  triangle inequality  imply 
\begin{align*}
\trinl {u}_{\m}^{n+1/2}-{u}_{\m}^{n-1/2}\trinl_{\pw}\trinl e_{\m}\trinl_{\pw}\trinl \bar{e}_\m\trinl_{\pw} 
& \le  \frac{C_{\rm dS}}{2}\big(\trinl {u}_{\m}^{n+1/2}\trinl_{\pw}^2+\trinl {u}_{\m}^{n-1/2}\trinl_{\pw}^2\big)\trinl e_{\m}\trinl_{\pw}^2+{C_{\rm dS}^{-1}}\trinl \bar{e}_\m\trinl_{\pw}^2.
\end{align*}
Choose $\psi=v_{\m}^{n+1/2}$ in  \eqref{P4 fully_discrete2} and utilize   \eqref{Bh-bdd} to show   $\trinl v_{\m}^{n+1/2}\trinl_{\pw}\le \frac{1}{2}C_{\rm dS}\trinl u_{\m}^{n+1/2}\trinl_{\pw}^2$.
From the bound $\trinl v_{\m}^{n-1/2}\trinl_{\pw}\le \frac{1}{2}C_{\rm dS}\trinl u_{\m}^{n-1/2}\trinl_{\pw}^2$ (obtained by replacing $n$ by $n-1$ in the last inequality)  and   triangle inequality, we have 
\begin{align*}
   \trinl {v}_{\m}^{n+1/2}+ {v}_{\m}^{n-1/2}\trinl_{\pw} \le \trinl {v}_{\m}^{n+1/2}\trinl_{\pw}+\trinl  {v}_{\m}^{n-1/2}\trinl_{\pw} \le \frac{1}{2} C_{\rm dS}\big( \trinl {u}_{\m}^{n+1/2}\trinl_{\pw}^2+\trinl  {u}_{\m}^{n-1/2}\trinl_{\pw}^2 \big).
\end{align*}
In turn, a combination of the last three displayed  inequalities followed by   \eqref{stabf1} (applied twice) implies 
\begin{align*}
&  {8}{k^{-2}}\norm{{e}_{\m}}^2 + 2\trinl {e}_{\m}\trinl_{\pw}^2+\trinl \bar{e}_{\m}\trinl_{\pw}^2  \le  C_{\rm dS}^2\big( \trinl {u}_{\m}^{n+1/2}\trinl_{\pw}^2+\trinl  {u}_{\m}^{n-1/2}\trinl_{\pw}^2 \big)\trinl e_{\m}\trinl_{\pw}^2 \le 2{\rm M}_{3}(u_0,u_1,f,k)\trinl e_{\m}\trinl_{\pw}^2 .
\end{align*}
Under the smallness assumption on  $u_0,u_1,k$  
from Theorem \eqref{TH-WELLPOSED-FINAL}$(c)$, we have ${\rm M}_{3}(u_0,u_1,f,k) =\delta$ for some $\delta <1$. This  shows that $ {8}{k^{-2}}\norm{{e}_{\m}}^2+2(1-\delta)\trinl {e}_{\m}\trinl_{\pw}^2+\trinl \bar{e}_{\m}\trinl_{\pw}^2 \le 0$, which is only possible if $\trinl {e}_{\m}\trinl_{\pw}=\trinl \bar{e}_{\m}\trinl_{\pw}=0$. 
\end{proof}
\subsection{A priori error analysis}\label{apriori-sec}
In this subsection, we first present the error equation. This is followed by Lemmas \ref{f,u,r,rho-lem}-\ref{lem-non-trunc} with  estimates for the intermediate terms appearing in the error equation. The proof of   
the error bound in energy-norm for the fully-discrete scheme concludes this subsection with optimal convergence in space and 
quadratic convergence in time.

\smallskip \noindent
\textbf{Error equation.} By the definition of the companion operator, we have \(J\varphi_{\m} \in H^2_0(\Omega)\) for all \(\varphi_{{\m}} \in {\m}({\cal T})\).  
With this choice of the test function, consider the convex combination  of  \eqref{P1 weak_form1}  with coefficients 
\( \frac{1}{4} \), \(\frac{1}{2}\), and \( \frac{1}{4} \) at the three time levels \(t_{n-1}\), \(t_n\), and \(t_{n+1}\), respectively, and subtract \eqref{P4 fully_discrete1} from resultant equation to obtain
\begin{align}
        &(u_{tt}^{n,1/4},J\varphi_{\m})_{L^2(\Omega)}+a(
u^{n,1/4},J\varphi_{\m} )- (\bar{\partial}_t ^2 u_{\m}^n,\varphi_{\m})_{L^2(\Omega)} - a_{\pw} (u_{\m}^{n,1/4},\varphi_{\m} )= (f^{n,1/4},(J-I)\varphi_{\m} )_{L^2(\Omega)}\nonumber\\
&\quad-\frac{1}{4}\big[b(u^{n+1},J\varphi_{\m},v^{n+1})+2b(u^{n},J\varphi_{\m},v^{n}) +b(u^{n-1},J\varphi_{\m},v^{n-1})\big]+b_{\pw}(u_{\m}^{n,1/4},\varphi_{\m},v_{\m}^{n,1/4}).\label{er1}
\end{align}
The linearity of \(b(\bullet,\bullet,\bullet)\) in each component and  \eqref{un14}-\eqref{dtu} with elementary algebra leads to 
the identity
\begin{align*}   
 &b(u^{n+1},J\varphi_{\m},v^{n+1})+2b(u^{n},J\varphi_{\m},v^{n}) +b(u^{n-1},J\varphi_{\m},v^{n-1})\\
 &=4b(u^{n,1/4},J\varphi_{\m},v^{n,1/4})+2k^2\big[b(\delta_tu^{n},J\varphi_{\m},\delta_tv^{n})+\frac{k^2}{8}b(\bar{\partial}^2_tu^{n},J\varphi_{\m},\bar{\partial}^2_t v^n)\big].
\end{align*}
This, $a(u^{n,1/4},J\varphi_{{\m}})=a_{\pw}(\mathcal{R}_{\m}u^{n,1/4},\varphi_{{\m}})$ from  \eqref{P1 ritz_projection}, \eqref{p4-split1}, and some elementary manipulations in \eqref{er1} with the truncation  error $r^n:=\bar{\partial}_t ^2 u^n- u_{tt}^{n,1/4} $ lead to
\begin{align*}
     (\bar{\partial}_t ^2 \zeta^n,\varphi_{\m})_{L^2(\Omega)} + a_{\pw} (\zeta^{n,1/4},\varphi_{\m})&= ( f^{n,1/4}- u_{tt}^{n,1/4}, (J-I)\varphi_{\m})_{L^2(\Omega)}+({r}^n-{\bar{\partial}_t}^2 \rho^n,\varphi_{\m} )_{L^2(\Omega)} \nonumber\\  
    &\quad +b_{\pw}(u_{\m}^{n,1/4},\varphi_{\m},v_{\m}^{n,1/4})-b(u^{n,1/4},J\varphi_{\m},v^{n,1/4}) \\
    &\quad+\frac{k^2}{2}\big[b(\delta_tu^{n},J\varphi_{\m},\delta_tv^{n})+\frac{k^2}{8}b(\bar{\partial}^2_tu^{n},J\varphi_{\m},\bar{\partial}^2_t v^n)\big].
    \end{align*}
     The choice $\varphi_{\m} =  2k \delta_t \zeta^n $ in the above equation, \eqref{un14}-\eqref{dtu}, and \eqref{P1 mesh_norm} reveal
\begin{align}
&\norm{\bar{\partial}_t \zeta^{n+1/2}}^2-\norm{\bar{\partial}_t \zeta^{n-1/2}}^2+ \trinl\zeta^{n+1/2}\trinl_{\pw}^2-\trinl \zeta^{n-1/2}\trinl_{\pw}^2=2k( f^{n,1/4}- u_{tt}^{n,1/4}, (J-I)\delta _t\zeta^{n}) +2k(r^n-{\bar{\partial}_t}^2 \rho^n,\delta _t\zeta^{n})\nonumber\\
 &
+2k\big[b_{\pw}(u_{\m}^{n,1/4},\delta_t \zeta^n,v_{\m}^{n,1/4})-b(u^{n,1/4},J\delta_t \zeta^n,v^{n,1/4})\big]+{k^3}\big[b(\delta_tu^{n},J\delta_t \zeta^n,\delta_tv^{n})+\frac{k^2}{8}b(\bar{\partial}^2_tu^{n},J\delta_t \zeta^n,\bar{\partial}^2_t v^n)\big].\label{eru1}
 \end{align}
 Choose the test function in \eqref{P1 weak_form2} as $J\psi_{\m}$ and consider the difference at $t=t_{n+1}$ and $t=t_{n-1}$. Utilize the linearity of $a(\bullet,\bullet)$ to show $ a(v^{n+1}-v^{n-1},J\psi_{\m})=\frac{1}{2}b(u^{n+1},u^{n+1},J\psi_{\m})-\frac{1}{2}b(u^{n-1},u^{n-1},J\psi_{\m})$ for all $ \psi_{\m} \in {\m}({\cal T})$.  Also multiply $\eqref{P4 fully_discrete2}$ by 2 and  consider its difference  at  $n$ and $n+1$ steps  
 to obtain $2a_{\pw}(v_{\m}^{n+1/2}-v_{\m}^{n-1/2},\psi_{\m})=b_{\pw}(u_{\m}^{n+1/2},u_{\m}^{n+1/2},\psi_{\m})-b_{\pw}(u_{\m}^{n-1/2},u_{\m}^{n-1/2},\psi_{\m})$. Subtract  the two equations and utilize \eqref{un14} to arrive at
\begin{align}
   & a(v^{n+1}-v^{n-1},J\psi_{\m})- a_{\pw}(v_{\m}^{n+1}-v_{\m}^{n-1},\psi_{\m})=\frac{1}{2}\left(b(u^{n+1},u^{n+1},J\psi_{\m})-b(u^{n-1},u^{n-1},J\psi_{\m})\right)\nonumber
     \\&\qquad \qquad-b_{\pw}(u_{\m}^{n+1/2},u_{\m}^{n+1/2},\psi_{\m})+b_{\pw}(u_{\m}^{n-1/2},u_{\m}^{n-1/2},\psi_{\m})
     .\label{errv1}
\end{align}
The identity $a(v^{n+1}-v^{n-1},J\psi_{\m})=a_{\pw}({\cal R}_{\m}(v^{n+1}-v^{n-1}),\psi_{\m})$ from \eqref{P1 ritz_projection} and an application of \eqref{p4-split2} rewrite the left-hand side of \eqref{errv1} as
 \begin{align*}
     a(v^{n+1}-v^{n-1},J\psi_{\m})- a_{\pw}(v_{\m}^{n+1}-v_{\m}^{n-1},\psi_{\m})=a_{\pw}(\theta^{n+1}-\theta^{n-1},\psi_{\m}).
 \end{align*}
 The linearity of \(b(\bullet,\bullet,\bullet)\)  (and \(b_{\pw}(\bullet,\bullet,\bullet)\)) and \eqref{un14}-\eqref{dtu} modify the right-hand side of \eqref{errv1} as
 \begin{align*}
     &\frac{1}{2}\left(b(u^{n+1},u^{n+1},J\psi_{\m})-b(u^{n-1},u^{n-1},J\psi_{\m})\right)=2kb(u^{n,1/4},\delta_tu^{n},J\psi_{\m})+\frac{k^3}{2}b(\bar{\partial}^2_tu^{n},\delta_tu^{n},J\psi_{\m}),\\
     &b_{\pw}(u_{\m}^{n+1/2},u_{\m}^{n+1/2},\psi_{\m})-b_{\pw}(u_{\m}^{n-1/2},u_{\m}^{n-1/2},\psi_{\m}) =2kb_{\pw}(u_{\m}^{n,1/4},\delta_tu_{\m}^{n},\psi_{\m}).
     \end{align*}
 A combination of the last three displayed identities in
\eqref{errv1} shows
\begin{align*}
    &a_{\pw}(\theta^{n+1}-\theta^{n-1},\psi_{\m})
    =2k\big(b(u^{n,1/4},\delta_tu^{n},J\psi_{\m})-b_{\pw}(u_{\m}^{n,1/4},\delta_tu_{\m}^{n},\psi_{\m})\big)+\frac{k^3}{2}b(\bar{\partial}^2_tu^{n},\delta_tu^{n},J\psi_{\m}).
\end{align*}
 The choice $\psi_{\m}=\theta^{n,1/4}=\frac{1}{2}\big(\theta^{n+1/2}+\theta^{n-1/2}\big)$ and the identity $\theta^{n+1}-\theta^{n-1}=2\big(\theta^{n+1/2}-\theta^{n-1/2}\big)$ from \eqref{varphi1} then gives 
 \begin{align*}
\trinl{\theta^{n+1/2}}\trinl_{\pw}^2-\trinl{\theta^{n-1/2}}\trinl_{\pw}^2&=2k\big[b(u^{n,1/4},\delta_tu^{n},J\theta^{n,1/4})-b_{\pw}(u_{\m}^{n,1/4},\delta_tu_{\m}^{n},\theta^{n,1/4})\big] +\frac{k^3}{2}b(\bar{\partial}^2_tu^{n},\delta_tu^{n},J\theta^{n,1/4}).
 \end{align*}
 We then add this and \eqref{eru1}, sum from $n=1,2,\cdots,m$, and utilize the telescopic property 
 to arrive at 
 \begin{align}
  &  \|\bar{\partial}_t \zeta^{m+1/2}\|^2+\trinl{\zeta^{m+1/2}}\trinl_{\pw}^2+\trinl{\theta^{m+1/2}}\trinl_{\pw}^2 = \big[\norm{\bar{\partial}_t \zeta^{1/2}}^2+ \trinl {\zeta^{1/2}}\trinl_{\pw}^2  +  \trinl{\theta^{1/2}}\trinl_{\pw}^2 \big]\nonumber\\
 & \quad 
 +2k\sum_{n=1}^m\big[
 ( f^{n,1/4}- u_{tt}^{n,1/4}, (J-I)\delta _t\zeta^{n})+(r^n-{\bar{\partial}_t}^2 \rho^n,\delta _t\zeta^{n})\big]
 \nonumber\\
 & \quad 
 +2k\sum_{n=1}^m\big[b_{\pw}(u_{\m}^{n,1/4},\delta_t \zeta^n,v_{\m}^{n,1/4})-b(u^{n,1/4},J\delta_t \zeta^n,v^{n,1/4})+ b(u^{n,1/4},\delta_tu^{n},J\theta^{n,1/4})-b_{\pw}(u_{\m}^{n,1/4},\delta_tu_{\m}^{n},\theta^{n,1/4})\big]\nonumber\\
 & \quad + k^3\sum_{n=1}^m  \big[b(\delta_tu^{n},J\delta_t \zeta^n,\delta_tv^{n})+\frac{k^2}{8}b(\bar{\partial}^2_tu^{n},J\delta_t \zeta^n,\bar{\partial}^2_t v^n)  +\frac{1}{2}b(\bar{\partial}^2_tu^{n},\delta_tu^{n},J\theta^{n,1/4}) \big] \nonumber\\
 &\quad =:T_1+T_2+T_3+T_4 \; \text{ for } 1 \le m \le N-1. \label{TJ}
\end{align}
This is analogous to \eqref{all_ini}; with an (additional) initial error contribution $T_1$, linear contributions $T_2$, nonlinear contributions $T_3$, and truncation terms $T_4$. The term $T_1$ is controlled in \eqref{eqn-ni}. The estimates for $T_2$-$T_4$  are technical and we present them in the Lemmas \ref{f,u,r,rho-lem}-\ref{lem-non-trunc} below. Combining these arguments gives the proof of Theorem~\ref{TH-ERROR-FINAL}.
\begin{lemma}[bound of linear term-$T_2$]\label{f,u,r,rho-lem}
Let  $f  \in H^1(0,T;L^2(\Omega))$, $u_{tt}  \in L^2(0,T;H^{2+\sigma}(\Omega))$, and   $u  \in H^4(0,T;L^2(\Omega))$. Then the following bound holds:
\begin{align*}
   T_2:=2 k \sum_{n=1}^m\big[
( f^{n,1/4}- u_{tt}^{n,1/4}, (J-I)\delta _t\zeta^{n})_{L^2(\Omega)}+(r^n-{\bar{\partial}_t}^2 \rho^n,\delta _t\zeta^{n})_{L^2(\Omega)}\big] \le {\cal C}(u_{tt},u_{ttt},u_{tttt},f,f_t)\big(h^{2\sigma}+k^4\big)&\\
+\frac{1}{32}\big[\trinl  \zeta^{1/2}\trinl_{\pw}^2+\trinl  \zeta^{m+1/2}\trinl_{\pw}^2\big]+\frac{7}{8}\big[\|\bar{\partial}_t \zeta^{1/2}\|^2+\|\bar{\partial}_t \zeta^{m+1/2}\|^2\big]+ k \sum_{n=0}^{m-1}\trinl \zeta^{n+1/2}\trinl_{\pw}^2+\frac{k}{T}\sum_{n=0}^{m-1}\|\bar{\partial}_t \zeta^{n+1/2}\|^2.
\end{align*}
\end{lemma}
\begin{proof}
    The definition \eqref{varphi2} and the identity   \eqref{dtu} show  $k\delta _t\zeta^{n}=\zeta^{n+1/2} -\zeta^{n-1/2}$. Then,  we obtain 
\begin{align}
k\sum_{n=1}^{m}
 ( f^{n,1/4}- u_{tt}^{n,1/4}, (J-I)\delta _t\zeta^{n})_{L^2(\Omega)}
 &
 = ( f^{m,1/4}- u_{tt}^{m,1/4}, (J-I) \zeta^ {m+1/2})_{L^2(\Omega)} 
 +(  f^{1,1/4}-u_{tt}^{1,1/4}, (I-J) \zeta^ {1/2} )_{L^2(\Omega)}  \nonumber \\
&  + \sum_{n=1}^{m-1} \big( (f^{n+1,1/4}-f^{n,1/4})-(u_{ttt}^{n+1,1/4}-u_{ttt}^{n,1/4}),
(I-J)\zeta^{n+1/2}\big)_{L^2(\Omega)}.\label{fini}
\end{align}
A Cauchy-Schwarz inequality, Lemma \ref{bhcompanion_lem}$(e)$, and   Young’s inequality with $a=\constref{ca}h^2  \norm{ f^{m,1/4}- u_{tt}^{m,1/4}}$, $b=\trinl \zeta^ {m+1/2}\trinl_{\pw}$, and $\epsilon=32$ lead to
\begin{align}
   ( f^{m,1/4}&- u_{tt}^{m,1/4}, (J-I) \zeta^ {m+1/2})_{L^2(\Omega)}  \le \norm{ f^{m,1/4}- u_{tt}^{m,1/4}}\norm{(J-I) \zeta^ {m+1/2}}\nonumber\\
    & \le    {32}\constref{ca}^2h^4 \big(\norm{ f}_{C(0,T;L^2(\Omega))}^2+\norm{ u_{tt}}^2_{C(0,T;L^2(\Omega))}\big)+\frac{1}{64}\trinl \zeta^ {m+1/2}\trinl_{\pw}^2\label{f0}
\end{align}
with $\norm{ f^{m,1/4}- u_{tt}^{m,1/4}} \le 2 \norm{ f}_{C(0,T;L^2(\Omega))}^2+2\norm{ u_{tt}}^2_{C(0,T;L^2(\Omega))}$ in the last step.
Analogous arguments show
\begin{align}
  &( f^{1,1/4}- u_{tt}^{1,1/4}, (I-J) \zeta^ {1/2})_{L^2(\Omega)} 
  \le    {32}\constref{ca}^2h^4 \big(\norm{ f}_{C(0,T;L^2(\Omega))}^2+\norm{ u_{tt}}^2_{C(0,T;L^2(\Omega))}\big)+\frac{1}{64}\trinl \zeta^ {1/2}\trinl_{\pw}^2.\label{f00}
\end{align}
The definition  \eqref{un14} reveals 
\begin{align}
    & (f^{n+1,1/4}-f^{n,1/4})-(u_{tt}^{n+1,1/4}-u_{tt}^{n,1/4})= \frac{1}{4}\Big( \int_{t_{n}}^{t_{n+2}}(f_t(t)-u_{ttt}(t))\dt+\int_{t_{n-1}}^{t_{n+1}}(f_t(t)-u_{ttt}(t))\dt\Big).\label{17}
\end{align}
A Cauchy–Schwarz inequality, Lemma \ref{bhcompanion_lem}$(e)$, and elementary manipulations lead to
\begin{align}
    \big( \int_{t_{n}}^{t_{n+2}}( f_t(t)-u_{ttt}(t)))\dt ,
(I-J)\zeta^{n+1/2}\big)_{L^2(\Omega)} \le \|\int_{t_{n}}^{t_{n+2}}( f_t(t)-u_{ttt}(t))\dt\|\|(I-J)\zeta^{n+1/2}\|\qquad\quad\nonumber\\
 \le \constref{ca}h^2\|\int_{t_{n}}^{t_{n+2}}( f_t(t)-u_{ttt}(t))\dt\|\trinl \zeta^ {n+1/2}\trinl_{\pw}\le \constref{ca}\sqrt{2{k}}h^2\|f_t-u_{ttt}\|_{L^2({t_{n}},{t_{n+2}},L^2(\Omega))}\trinl \zeta^ {n+1/2}\trinl_{\pw} \label{ff1}
\end{align}
with $\|\int_{t_{n}}^{t_{n+2}}( f_t(t)-u_{ttt}(t))\dt\|  \le \int_{t_{n}}^{t_{n+2}}\| f_t(t)-u_{ttt}(t)\|\dt \le \sqrt{2k}\|f_t-u_{ttt}\|_{L^2({t_{n}},{t_{n+2}},L^2(\Omega))}$ in the last step. A Young's inequality with $a=\constref{ca}h^2\|f_t-u_{ttt}\|_{L^2({t_{n}},{t_{n+2}},L^2(\Omega))}$, $b=\sqrt{2{k}}\trinl \zeta^ {n+1/2}\trinl_{\pw}$, and $\epsilon=1$ on the right-hand side of  \eqref{ff1}, and a triangle inequality  result in
\begin{align*}
\big( \int_{t_{n}}^{t_{n+2}}( f_t(t)-u_{ttt}(t))\dt ,
(I-J)\zeta^{n+1/2}\big)_{L^2(\Omega)} &
\le \constref{ca}^2h^4\big(\|f_t\|^2_{L^2({t_{n}},{t_{n+2}},L^2(\Omega))}+\|u_{ttt}\|^2_{L^2({t_{n}},{t_{n+2}},L^2(\Omega))}\big)+k\trinl \zeta^ {n+1/2}\trinl_{\pw}^2.
\end{align*}
Analogous arguments show
\begin{align*}
\big( \int_{t_{n-1}}^{t_{n+1}}( f_t(t)-u_{ttt}(t))\dt ,
(I-J)\zeta^{n+1/2}\big)_{L^2(\Omega)}  & \le \constref{ca}^2h^4\big(\|f_t\|^2_{L^2({t_{n-1}},{t_{n+1}},L^2(\Omega))}+\|u_{ttt}\|^2_{L^2({t_{n-1}},{t_{n+1}},L^2(\Omega))}\big)+k\trinl \zeta^ {n+1/2}\trinl_{\pw}^2.
\end{align*}
A combination of last two inequalities in \eqref{17} and $\sum_{i=1}^{m-1}\norm{\varphi}^2_{L^2({t_{i}},{t_{i+2}},L^2(\Omega))} \le 2 \norm{\varphi}^2_{L^2(0,T,L^2(\Omega))}$ for $\varphi \in \{f,u_{ttt}\}$, reveals
\begin{align*}
  &  \sum_{n=1}^{m-1} \big( (f^{n+1,1/4}-f^{n,1/4})-(u_{ttt}^{n+1,1/4}-u_{ttt}^{n,1/4}),
(I-J)\zeta^{n+1/2}\big)_{L^2(\Omega)}\nonumber\\
&\le \constref{ca}^2h^4\big(\|f_t\|^2_{L^2(0,T;L^2(\Omega))}+\|u_{ttt}\|^2_{L^2(0,T;L^2(\Omega))}\big)+ \frac{k}{2} \sum_{n=1}^{m-1}\trinl \zeta^ {n+1/2}\trinl_{\pw}^2.
\end{align*}
This and a combination of \eqref{fini}- \eqref{f00} with $h^4 \lesssim h^{2\sigma}$ establish
\begin{align}
   & 2k\sum_{n=1}^{m}
 ( f^{n,1/4}- u_{tt}^{n,1/4}, (J-I)\delta _t\zeta^{n})_{L^2(\Omega)}\nonumber\\
 & \le  {\cal C}(u_{tt},u_{ttt},f,f_t)h^{2\sigma}+\frac{1}{32}\big[\trinl \zeta^ {1/2}\trinl_{\pw}^2
  +\trinl \zeta^ {m+1/2}\trinl_{\pw}^2\big]+ k \sum_{n=1}^{m-1}\trinl \zeta^ {n+1/2}\trinl_{\pw}^2 \label{f,utt}.
\end{align}
Cauchy-Schwarz and Young's inequalities (with $a=\|r^n-{\bar{\partial}_t}^2 \rho^n\|,$ $b=\|\delta _t \zeta^{n}\|$, and $\epsilon=4T/7$), a summation  over $n$, 
a triangle inequality $\|r^n-{\bar{\partial}_t}^2 \rho^n\|^2\le 2  \|r^n\|^2+2\|{\bar{\partial}_t}^2 \rho^n\|^2$, the bounds from Lemma~\ref{trunc-lem}$(b)$ for the truncation error, and Lemma~\ref{P1 ritz_lemma} result in
\begin{align*}
    2k\sum_{n=1}^{m}(r^n-{\bar{\partial}_t}^2 \rho^n,\delta _t \zeta^{n} )_{L^2(\Omega)}\nonumber 
    &\le \frac{8T}{7}\big[{\cal C}(u_{tttt})k^4+\frac{4}{3} h^{4\sigma}\|u_{tt}\|^2_{L^2(0,T;H^{2+\sigma}(\Omega))}\big]+\frac{7k}{4T}\sum_{n=1}^{m}\|\delta _t \zeta^{n}\|^2.
\end{align*}
From \eqref{dtu}, we have $2\delta_t\zeta^{n}=\bar{\partial}_t \zeta^{n-1/2}+\bar{\partial}_t \zeta^{n+1/2}$. This and a  triangle inequality infer
\begin{align*}
  \frac{k}{T}\sum_{n=1}^m \trinl\delta_t\zeta^{n}\trinl_{\pw}^2 &\le \frac{k}{2T}\sum_{n=1}^m\big[\trinl\bar{\partial}_t \zeta^{n-1/2}\trinl_{\pw}^2+\trinl\bar{\partial}_t \zeta^{n+1/2}\trinl_{\pw}^2 \big]  \le \frac{k}{2T}\big[\trinl \bar{\partial}_t\zeta^{1/2}\trinl_{\pw}^2+\trinl \bar{\partial}_t \zeta^{m+1/2}\trinl_{\pw}^2\big] \nonumber\\
  &\qquad+\frac{k}{T}\sum_{n=1}^{m-1} \trinl\bar{\partial}_t \zeta^{n+1/2}\trinl_{\pw}^2\le \frac{1}{2}\big[\trinl \bar{\partial}_t \zeta^{1/2}\trinl_{\pw}^2+\trinl \bar{\partial}_t \zeta^{m+1/2}\trinl_{\pw}^2\big]+\frac{k}{T}\sum_{n=1}^{m-1} \trinl \bar{\partial}_t\zeta^{n+1/2}\trinl_{\pw}^2
\end{align*}
with $k \le T$ for the first two terms on the right-hand side. The last two displayed inequalities and $h^{4\sigma} \le h^{2\sigma}|\Omega|^{\sigma} $ reveal 
\begin{align*}
    2k\sum_{n=1}^{m}(r^n-{\bar{\partial}_t}^2 \rho^n,\delta _t \zeta^{n} )_{L^2(\Omega)}
    & \le {\cal C}(u_{tt},u_{tttt})(k^4+h^{2\sigma})
   +\frac{7}{8}\big[\|\bar{\partial}_t \zeta^{1/2}\|^2+\|\bar{\partial}_t \zeta^{m+1/2}\|^2\big]+\frac{7k}{4}\sum_{n=1}^{m-1}\|\bar{\partial}_t \zeta^{n+1/2}\|^2.
\end{align*}
This in combination with  \eqref{f,utt} and the addition of non-negative terms $k \trinl \zeta^ {1/2}\trinl_{\pw}^2+\frac{7k}{4}\|\bar{\partial}_t \zeta^{1/2}\|^2$ to the right-hand side conclude the proof.
\end{proof}
The next lemma provides bounds for the nonlinear contributions $T_3$, representing the main 
difficulty. The estimates rely on the stability  of continuous and discrete trilinear forms, approximation properties of the  Ritz 
projection (cf. Lemma~\ref{P1 ritz_lemma}) and of the companion operator (cf.  Lemma~\ref{bhcompanion_lem}($d,e$)). 
\begin{lemma}[Bound for nonlinear term $T_3$]Let  $u_0 \in H^2_0(\Omega), u_1\in L^2(\Omega)$, and $f \in L^2(0,T,L^2(\Omega))$ be  such that ${\rm M}_{1}(u_0,u_1,f) \le 1/2$. Assume that   $ u,v,  \in H^1(0,T;H^2_0(\Omega)\cap H^{2+\sigma}(\Omega)),$ $u_{t},v_t\in L^\infty(0,T;H^{2}(\Omega))$. Then
    \begin{align*}
     &T_3:=2k\sum_{n=1}^m \big[b_{\pw}(u_{\m}^{n,1/4},\delta_t \zeta^n,v_{\m}^{n,1/4})-b(u^{n,1/4},J\delta_t \zeta^n,v^{n,1/4})+  b(u^{n,1/4},\delta_t u^{n},J\theta^{n,1/4})-b_{\pw}(u_{\m}^{n,1/4},\delta_t u_{\m}^{n},\theta^{n,1/4})\big]\\
     &\qquad\le{\cal C}(u,u_t,v,v_t)h^{2\sigma}+\frac{17}{32}\trinl\zeta^{1/2}\trinl_{\pw}^2 +\frac{9}{16}\trinl\theta^{1/2}\trinl_{\pw}^2+ \frac{25}{32}\trinl\zeta^{m+1/2}\trinl_{\pw}^2+\frac{13}{16}\trinl\theta^{m+1/2}\trinl_{\pw}^2\\
      &\qquad \quad +k\big(\constref{cdsj}\big[\|u_t\|_{L^\infty(0,T;H^2(\Omega))}+\|v_t\|_{L^\infty(0,T;H^2(\Omega))}\big]+3\big)\sum_{n=0
      }^{m-1} \trinl\zeta^{n+1/2}\trinl_{\pw}^2\\
      &\qquad \quad+k\big(\constref{cdsj}\|u_t\|_{L^\infty(0,T;H^2(\Omega))}+\frac{9}{8T}\big)\sum_{n=0}^{m-1} \trinl\theta^{n+1/2}\trinl_{\pw}^2 \;\text{  holds {for any $m$ with $1 \le m \le N-1$}.}
    \end{align*}
\end{lemma}
\begin{proof} Definitions  \eqref{varphi1}-\eqref{varphi2} and linearity of $b_{\pw}(\bullet,\bullet,\bullet)$ in each component imply the relation 
    \begin{align*}
       & b_{\pw}(u_{\m}^{n,1/4},\delta_t \zeta^n,v_{\m}^{n,1/4})-b(u^{n,1/4},J\delta_t \zeta^n,v^{n,1/4}) =b_{\pw}(\zeta^{n,1/4},\delta_t \zeta^n,\theta^{n,1/4})\\
       &-b_{\pw}(\zeta^{n,1/4},\delta_t \zeta^n,\mathcal{R}_{\m}v^{n,1/4})+b_{\pw}(\rho^{n,1/4},\delta_t \zeta^n,\theta^{n,1/4})-b_{\pw}(\rho^{n,1/4},\delta_t \zeta^n,\mathcal{R}_{\m}v^{n,1/4})\\
       &-b_{\pw}(u^{n,1/4},\delta_t \zeta^n,\theta^{n,1/4})-b_{\pw}(u^{n,1/4},\delta_t \zeta^n,\chi^{n,1/4})+b_{\pw}(u^{n,1/4},(I-J)\delta_t \zeta^n,v^{n,1/4}).\nonumber
    \end{align*}
Analogous arguments with $\delta_t (\mathcal{R}_{\m} u^{n})=\mathcal{R}_{\m}\delta_t u^n$ from linearity of $\mathcal{R}_{\m}$ lead to
    \begin{align*}
      &  b(u^{n,1/4},\delta_t u^{n},J\theta^{n,1/4})-b_{\pw}(u_{\m}^{n,1/4},\delta_t u_{\m}^{n},\theta^{n,1/4})=-b_{\pw}(\zeta^{n,1/4},\delta_t \zeta^{n},\theta^{n,1/4})\\
       & +b_{\pw}(\zeta^{n,1/4},\mathcal{R}_{\m}\delta_t u^n,\theta^{n,1/4})-b_{\pw}(\rho^{n,1/4},\delta_t\zeta^{n},\theta^{n,1/4})+b_{\pw}(\rho^{n,1/4},\mathcal{R}_{\m}\delta_t u^n,\theta^{n,1/4})\\
       &+b_{\pw}(u^{n,1/4},\delta_t \zeta^{n},\theta^{n,1/4})+  b_{\pw}(u^{n,1/4},\delta_t \rho^{n},\theta^{n,1/4})+b_{\pw}(u^{n,1/4},\delta_t u^{n},(J-I)\theta^{n,1/4}).
    \end{align*}
Then, noticing  that $b_{\pw}(\zeta^{n,1/4},\delta_t \zeta^n,\theta^{n,1/4}),$ $b_{\pw}(\rho^{n,1/4},\delta_t \zeta^n,\theta^{n,1/4})$, and  $ b_{\pw}(u^{n,1/4},\delta_t \zeta^n,\theta^{n,1/4})$ cancel out, we arrive at  
    \begin{align}
     T_3
        = 2k\sum_{n=1}^m \big[ &-b_{\pw}(\zeta^{n,1/4},\delta_t \zeta^n,\mathcal{R}_{\m}v^{n,1/4})+b_{\pw}(\zeta^{n,1/4},\mathcal{R}_{\m}\delta_t u^n,\theta^{n,1/4})-b_{\pw}(\rho^{n,1/4},\delta_t \zeta^n,\mathcal{R}_{\m}v^{n,1/4})\nonumber\\
       &+b_{\pw}(\rho^{n,1/4},\mathcal{R}_{\m}\delta_t u^n,\theta^{n,1/4})-b_{\pw}(u^{n,1/4},\delta_t \zeta^n,\chi^{n,1/4})+  b_{\pw}(u^{n,1/4},\delta_t \rho^{n},\theta^{n,1/4})\nonumber\\
       &+b_{\pw}(u^{n,1/4},(I-J)\delta_t \zeta^n,v^{n,1/4})+b_{\pw}(u^{n,1/4},\delta_t u^{n},(J-I)\theta^{n,1/4})\big]=:\sum_{i=1}^8S_i.\label{sn}
    \end{align}
    The rest of the proof focuses on bounding the terms $S_1,S_2, \cdots,S_8$. 

\medskip \noindent {\it Bound for $S_1$.}
 Using $2\zeta^{n,1/4}=\zeta^{n+1/2}+\zeta^{n-1/2}$, $k\delta_t \zeta^n=\zeta^{n+1/2}-\zeta^{n-1/2}$ (obtained from  \eqref{varphi1}-\eqref{varphi2})  followed by linearity of $b_{\pw}(\bullet,\bullet,\bullet)$ in each component  and symmetry  in first two components show \begin{align*}
 &-2kb_{\pw}(\zeta^{n,1/4},\delta_t \zeta^n,\mathcal{R}_{\m}v^{n,1/4})=-b_{\pw}(\zeta^{n+1/2},\zeta^{n+1/2},\mathcal{R}_{\m}v^{n,1/4})+b_{\pw}(\zeta^{n-1/2},\zeta^{n-1/2},\mathcal{R}_{\m}v^{n,1/4}).
\end{align*}
 Utilize  
${\mathcal{R}_{\m} v^{n,1/4}=\mathcal{R}_{\m} v^{n+1/2}-\frac{k}{2}\mathcal{R}_{\m} \delta_t v^{n}}$ and  ${ \mathcal{R}_{\m} v^{n,1/4}=\mathcal{R}_{\m} v^{n-1/2}+\frac{k}{2}\mathcal{R}_{\m} \delta_t v^{n}}$
in the third components of {first and second term, respectively, } to infer
    \begin{align*}
    -2kb_{\pw}(\zeta^{n,1/4},\delta_t \zeta^n,\mathcal{R}_{\m}v^{n,1/4})
    & =-b_{\pw}(\zeta^{n+1/2},\zeta^{n+1/2},\mathcal{R}_{\m}v^{n+1/2})+b_{\pw}(\zeta^{n-1/2},\zeta^{n-1/2},\mathcal{R}_{\m}v^{n-1/2}) \\
    &\quad +\frac{k}{2}\big(b_{\pw}(\zeta^{n+1/2},\zeta^{n+1/2},\mathcal{R}_{\m}\delta_tv^{n})+b_{\pw}(\zeta^{n-1/2},\zeta^{n-1/2},\mathcal{R}_{\m}\delta_tv^{n})\big).
    \end{align*}
A summation over $n=1,2,\cdots,m$ (for $1 \le m \le N-1$), the definition $S_1 := -2k \sum_{n=1}^m b_{\pw}(\zeta^{n,1/4},\delta_t \zeta^n,\mathcal{R}_{\m}v^{n,1/4})$, and a telescopic sum property in the first two terms on the  right-hand side show that 
  \begin{align*}
      S_1 & =-b_{\pw}(\zeta^{m+1/2},\zeta^{m+1/2},\mathcal{R}_{\m}v^{m+1/2})+b_{\pw}(\zeta^{1/2},\zeta^{1/2},\mathcal{R}_{\m}v^{1/2}) \\
    &\quad +\frac{k}{2} \sum_{n=1}^m\big(b_{\pw}(\zeta^{n+1/2},\zeta^{n+1/2},\mathcal{R}_{\m}\delta_tv^{n})+b_{\pw}(\zeta^{n-1/2},\zeta^{n-1/2},\mathcal{R}_{\m}\delta_tv^{n})\big).
  \end{align*}
  Utilize ${\cal R}_{\m}v^{m+1/2}={\cal R}_{\m}v^{m,1/4}+\frac{k}{2}{\cal R}_{\m}\delta_tv^{m}$  and linearity of $b_{\pw}(\bullet,\bullet,\bullet)$ in third component to obtain
\begin{align*}
  &-b_{\pw}(\zeta^{m+1/2},\zeta^{m+1/2},{\cal R}_{\m}v^{m+1/2})+\frac{k}{2}\sum_{n=1}^m b_{\pw}(\zeta^{n+1/2},\zeta^{n+1/2},{\cal R}_{\m}\delta_tv^{n}) \\
  &=-b_{\pw}(\zeta^{m+1/2},\zeta^{m+1/2},{\cal R}_{\m}v^{m,1/4})+\frac{k}{2}\sum_{n=1}^{m-1}b_{\pw}(\zeta^{n+1/2},\zeta^{n+1/2},{\cal R}_{\m}\delta_tv^{n}).
  \end{align*}
   A combination of last two displayed identities reveal
    \begin{align}
       S_1&=-b_{\pw}(\zeta^{m+1/2},\zeta^{m+1/2},\mathcal{R}_{\m}v^{m,1/4})+b_{\pw}(\zeta^{1/2},\zeta^{1/2},\mathcal{R}_{\m}v^{1/2})\nonumber\\
      & \qquad  +\frac{k}{2}\Big[ \sum_{n=1}^{m-1}b_{\pw}(\zeta^{n+1/2},\zeta^{n+1/2},\mathcal{R}_{\m}\delta_tv^{n})+\sum_{n=1}^{m}b_{\pw}(\zeta^{n-1/2},\zeta^{n-1/2},\mathcal{R}_{\m}\delta_tv^{n})\Big].\label{s1}
    \end{align}
    Then,    from \eqref{Bh-bdd}, \eqref{r-con},  $\constref{cdsj}=\|J\|C_{\rm dS }$, 
     $\trinl v^{m,1/4} \trinl_{\pw} \le \frac{1}{4}\big(\trinl v^{m-1} \trinl_{\pw}+2\trinl v^{m} \trinl_{\pw}+\trinl v^{m+1} \trinl_{\pw}\big)$  from  \eqref{varphi1},  a triangle  inequality, \eqref{energy-c2} with $t=t_{m-1},t_m,$ and $t_{m+1}$, we obtain  
    \begin{align*}
        &|b_{\pw}(\zeta^{m+1/2},\zeta^{m+1/2},\mathcal{R}_{\m}v^{m,1/4})| \le C_{\rm dS }\trinl \zeta^{m+1/2} \trinl_{\pw}^2\trinl\mathcal{R}_{\m}v^{m,1/4} \trinl_{\pw}\\
        &\qquad\le\constref{cdsj}\trinl \zeta^{m+1/2} \trinl_{\pw}^2\trinl v^{m,1/4} \trinl_{\pw}\le {\rm M}_{1}(u_0,u_1,f)\trinl \zeta^{m+1/2} \trinl_{\pw}^2 \le \frac{1}{2}\trinl \zeta^{m+1/2} \trinl_{\pw}^2.
    \end{align*}
  The last inequality follows from the  smallness data assumption  ${\rm M}_{1}(u_0,u_1,f) \le 1/2$. Similar arguments  lead to
   \begin{align*}
       b_{\pw}(\zeta^{1/2},\zeta^{1/2},\mathcal{R}_{\m}v^{1/2}) \le \frac{1}{2}\trinl \zeta^{1/2} \trinl_{\pw}^2. 
   \end{align*}
  Recall $\constref{cdsj}:=C_{\rm dS}\|J\|$. The inequalities \eqref{Bh-bdd},  \eqref{r-con}, and the bound 
   \begin{align}
   \trinl \delta_tv^{n} \trinl_{\pw}\le 
       \| \delta_tv^{n} \|_{H^2(\Omega)} \le \norm{v_t}_{L^\infty(0,T;H^2(\Omega))}
        \label{23}   
   \end{align}
 from Lemma~\ref{axul}$(e)$  reveal 
   \begin{align*}
         &b_{\pw}(\zeta^{n+1/2},\zeta^{n+1/2},\mathcal{R}_{\m}\delta_tv^{n}) 
         \le\constref{cdsj}\trinl \zeta^{n+1/2} \trinl_{\pw}^2\trinl \delta_tv^{n} \trinl_{\pw} \le \constref{cdsj}\norm{v_t}_{L^\infty(0,T;H^2(\Omega))}\trinl \zeta^{n+1/2} \trinl_{\pw}^2,
        \\
        &b_{\pw}(\zeta^{n-1/2},\zeta^{n-1/2},\mathcal{R}_{\m}\delta_tv^{n} )\le\constref{cdsj}\trinl \zeta^{n-1/2} \trinl_{\pw}^2\trinl \delta_tv^{n} \trinl_{\pw} \le \constref{cdsj}\norm{v_t}_{L^\infty(0,T;H^2(\Omega))}\trinl \zeta^{n-1/2} \trinl_{\pw}^2.
   \end{align*}
   A combination of all this with an addition of the term $\constref{cdsj}k\norm{v_t}_{L^\infty(0,T;H^2(\Omega))}\trinl \zeta^{1/2} \trinl_{\pw}^2$ on the  right-hand side  in \eqref{s1} leads to
   \begin{align*}
       S_1 \le   \frac{1}{2}\big(\trinl \zeta^{1/2} \trinl_{\pw}^2+ \trinl \zeta^{m+1/2} \trinl_{\pw}^2\big)+  \constref{cdsj}k\norm{v_t}_{L^\infty(0,T;H^2(\Omega))} \sum_{n=0}^{m-1}\trinl \zeta^{n+1/2} \trinl_{\pw}^2. \qed
   \end{align*}
   \medskip \noindent {\it Bound for $S_2$.}
Recall $2\zeta^{n,1/4}=\zeta^{n+1/2}+\zeta^{n-1/2}$ and $2\theta^{n,1/4}=\theta^{n+1/2}+\theta^{n-1/2}$ from  \eqref{varphi1}. The  linearity of $b_{\pw}(\bullet,\bullet,\bullet)$ in first and third components,
  \eqref{Bh-bdd}, and  \eqref{r-con} lead to
    \begin{align}
        S_2&:=2k \sum_{n=1}^m b_{\pw}(\zeta^{n,1/4},\mathcal{R}_{\m}\delta_t u^n,\theta^{n,1/4})  \le \frac{1}{2}\constref{cdsj}k \sum_{n=1}^m  \trinl  \delta_t u^n\trinl_{\pw} \big[\trinl\zeta^{n+1/2} \trinl_{\pw} \trinl\theta^{n+1/2} \trinl_{\pw} \nonumber \\
        &\qquad\qquad +\trinl \zeta^{n+1/2}\trinl_{\pw}  \trinl\theta^{n-1/2} \trinl_{\pw}+\trinl \zeta^{n-1/2}\trinl_{\pw}  \trinl\theta^{n+1/2} \trinl_{\pw} +\trinl \zeta^{n-1/2}\trinl_{\pw}\trinl\theta^{n-1/2} \trinl_{\pw}\big] .\label{s2}
    \end{align}
    We now control the terms on the right-hand side of \eqref{s2}. A  
  Young's inequality with  $a=\sqrt{\constref{cdsj} \trinl  \delta_t u^n\trinl_{\pw}}\trinl\zeta^{n+1/2}\trinl_{\pw},$ $b=\sqrt{\constref{cdsj} \trinl  \delta_t u^n\trinl_{\pw}}\trinl\theta^{n+1/2}\trinl_{\pw}$, and $\epsilon= 1$ (resp. $a=\constref{cdsj} \trinl  \delta_t u^n\trinl_{\pw}\trinl\zeta^{n+1/2}\trinl_{\pw}$, $b=\trinl\theta^{n+1/2}\trinl_{\pw}$, and $\epsilon= 2$)  for $ 1\le n \le m-1$  (resp. for $n=m$) reveals
  \begin{align}\label{sp2}
\frac{1}{2} \constref{cdsj}k \trinl  \delta_t u^n\trinl_{\pw}\trinl\zeta^{n+1/2}\trinl_{\pw}\trinl\theta^{n+1/2}\trinl_{\pw} \le
 \begin{cases}
 \frac{1}{4}{\constref{cdsj}}k\trinl \delta_t u^{n}\trinl_{\pw}\big[\trinl \zeta^{n+1/2}\trinl_{\pw}^2+\trinl \theta^{n+1/2} \trinl_{\pw}^2\big]&\text{if } 1 \le n \le m-1,\\
 \frac{1}{2}\constref{cdsj}^2{k^2} \trinl\delta_t u^{m}\trinl^2_{\pw}\trinl \zeta^{m+1/2}\trinl_{\pw}^2+\frac{1}{4}\trinl \theta^{m+1/2} \trinl_{\pw}^2&\text{if } n= m .
\end{cases}
  \end{align}
Then, \eqref{varphi2}, a triangle inequality, \eqref{energy-c2}, and the smallness data assumption ${\rm M}_{1}(u_0,u_1,f) \le 1/2$   show 
  \begin{align}
\frac{1}{2}\constref{cdsj}^2{k^2} \trinl \delta_t u^{m}\trinl^2_{\pw} &\le\frac{1}{4} \constref{cdsj}^2 \big(\trinl u^{m+1}\trinl^2_{\pw}+\trinl u^{m-1}\trinl^2_{\pw}\big)\le \frac{1}{2}\big({\rm M}_{1}(u_0,u_1,f)\big)^2 \le 1/8.  \label{vsmal}
 \end{align}
 This, the bound similar to \eqref{23} (with $v$ replaced by $u$)
 in \eqref{sp2}, and a summation from $n=1$ to $n=m$   reveal
  \begin{align}
     & \frac{1}{2}\constref{cdsj}k \sum_{n=1}^m  \trinl  \delta_t u^n\trinl_{\pw} \trinl\zeta^{n+1/2} \trinl_{\pw} \trinl\theta^{n+1/2} \trinl_{\pw}\nonumber \\
      &\le  \frac{1}{4}{\constref{cdsj}}k\|u_t\|_{L^\infty(0,T;H^2(\Omega))} \sum_{n=0}^{m-1}\big[\trinl \zeta^{n+1/2}\trinl_{\pw}^2+\trinl \theta^{n+1/2} \trinl_{\pw}^2\big] +\frac{1}{8}\trinl \zeta^{m+1/2}\trinl_{\pw}^2+\frac{1}{8}\trinl \theta^{m+1/2} \trinl_{\pw}^2\label{s21}
  \end{align}
  with the addition of  $\frac{1}{4}{\constref{cdsj}}k\|u_t\|_{L^\infty(0,T;H^2(\Omega))} \big[\trinl \zeta^{1/2}\trinl_{\pw}^2+\trinl \theta^{1/2} \trinl_{\pw}^2\big]$ on the right-hand side  of the   last inequality.

  \medskip \noindent Next we bound the second term on the right-hand side of \eqref{s2}. 
  Young's inequality with $a= \trinl\zeta^{n+1/2}\trinl_{\pw},b=\trinl\theta^{n-1/2}\trinl_{\pw}$, and $\epsilon= 1$ followed by \eqref{23} (with $v$ replaced by $u$) for $ 1\le n \le m-1$ and \eqref{vsmal} for $n=m$
  reveals
  \begin{align*}
     \frac{1}{2}\constref{cdsj}k    \sum_{n=1}^m \trinl \delta_t u^n\trinl_{\pw}\trinl\zeta^{n+1/2}\trinl_{\pw}\trinl\theta^{n-1/2}\trinl_{\pw} \le \frac{1}{4} \constref{cdsj}k\sum_{n=1}^{m}\trinl \delta_t u^n\trinl_{\pw}\big[ \trinl\zeta^{n+1/2}\trinl_{\pw}^2+\trinl\theta^{n-1/2}\trinl_{\pw}^2\big]\\
      \le \frac{1}{8}\trinl\zeta^{m+1/2}\trinl_{\pw}^2+\frac{1}{4} \constref{cdsj}k\|u_t\|_{L^\infty(0,T;H^2(\Omega))}\sum_{n=0}^{m-1}\Big[\trinl\zeta^{n+1/2}\trinl_{\pw}^2+\trinl\theta^{n+1/2}\trinl_{\pw}^2\Big]
  \end{align*}
  with the addition of a non-negative term $\frac{1}{4} \constref{cdsj}k \|u_t\|_{L^\infty(0,T;H^2(\Omega))} \|\zeta^{1/2}\|^2$ on  the right-hand side.

  \medskip
  \noindent
  Interchanging the roles of $\zeta^{n}$ and $\theta^{n}$, we bound the third term on the right-hand side of \eqref{s2} as
\begin{align*}
     &\frac{1}{2}\constref{cdsj}k    \sum_{n=1}^m \trinl\delta_t u^n\trinl_{\pw}\trinl\zeta^{n-1/2}\trinl_{\pw}\trinl\theta^{n+1/2}\trinl_{\pw} \\
     &\quad
      \le \frac{1}{8}\trinl\theta^{m+1/2}\trinl_{\pw}^2+\frac{1}{4} \constref{cdsj}k\|u_t\|_{L^\infty(0,T;H^2(\Omega))}\sum_{n=0}^{m-1}\Big[ \trinl\zeta^{n+1/2}\trinl_{\pw}^2 + \trinl\theta^{n+1/2}\trinl_{\pw}^2\Big].
  \end{align*}
  Analogous arguments lead to the bounds of the last term on the right-hand side of \eqref{s2} as
  \begin{align*}
      \frac{1}{2}\constref{cdsj}k \sum_{n=1}^m  \trinl  \delta_t u^n\trinl_{\pw} \trinl\zeta^{n-1/2} \trinl_{\pw} \trinl\theta^{n-1/2} \trinl_{\pw} &\le  \frac{1}{4}{\constref{cdsj}}k\|u_t\|_{L^\infty(0,T;H^2(\Omega))} \sum_{n=0}^{m-1}\big[\trinl \zeta^{n+1/2}\trinl_{\pw}^2+\trinl \theta^{n+1/2} \trinl_{\pw}^2\big].
      \end{align*}
Putting together the last three displayed inequalities and \eqref{s21} in \eqref{s2} results in 
      \begin{align*}
          S_2 &\le \frac{1}{4}\big[\trinl \zeta^{m+1/2}\trinl_{\pw}^2+\trinl \theta^{m+1/2} \trinl_{\pw}^2\big]+{\constref{cdsj}}k\|u_t\|_{L^\infty(0,T;H^2(\Omega))} \sum_{n=0}^{m-1}\big[\trinl \zeta^{n+1/2}\trinl_{\pw}^2+\trinl \theta^{n+1/2} \trinl_{\pw}^2\big] . \qed
      \end{align*}

\medskip \noindent {\it Bound for $S_3$.}
The identity  $k\delta _t \zeta^n= \zeta^{n+1/2}-\zeta^{n-1/2}$ from \eqref{dtu} shows 
$$-k b_{\pw}(\rho^{n,1/4},\delta_t \zeta^n,\mathcal{R}_{\m}v^{n,1/4})=- b_{\pw}(\rho^{n,1/4},\zeta^{n+1/2}-\zeta^{n-1/2},\mathcal{R}_{\m}v^{n,1/4}).$$ 
This and elementary manipulations motivate a split for $S_3$ in
\begin{align}
  & \frac{1}{2} S_3:=-k\sum_{n=1}^m b_{\pw}(\rho^{n,1/4},\delta_t \zeta^n,\mathcal{R}_{\m}v^{n,1/4})=  -b_{\pw}(\rho^{m,1/4},\zeta^{m+1/2},\mathcal{R}_{\m}v^{m,1/4})+ b_{\pw}(\rho^{1,1/4},\zeta^{1/2},\mathcal{R}_{\m}v^{1,1/4})\nonumber\\
    &\; +\sum_{n=1}^{m-1} b_{\pw}(\rho^{n+1,1/4}-\rho^{n,1/4},\zeta^{n+1/2},\mathcal{R}_{\m}v^{n+1,1/4})+ \sum_{n=1}^{m-1}b_{\pw}(\rho^{n,1/4},\zeta^{n+1/2},\mathcal{R}_{\m}(v^{n+1,1/4}-v^{n,1/4})).\label{s4all}
\end{align}
The boundedness from \eqref{Bh-bdd}, \eqref{r-con}, Lemma~\ref{axul_ritz}$(b)$,  and Lemma~\ref{axul}$(b)$ imply
\begin{align*}
 |b_{\pw}(\rho^{m,1/4},\zeta^{m+1/2},\mathcal{R}_{\m}v^{m,1/4}) |
& \le \constref{car}\constref{cdsj}h^{\sigma}\|u\|_{C([0,T];H^{2+\sigma}(\Omega))}\|v\|_{C([0,T];H^2(\Omega))}\trinl\zeta^{m+1/2} \trinl_{\pw}. 
 \end{align*}
 A Young's inequality with $a=\constref{car}\constref{cdsj}h^{\sigma}\|u\|_{C([0,T];H^{2+\sigma}(\Omega))}\|v\|_{C([0,T];H^2(\Omega))},$ $ b=\trinl\zeta^{m+1/2} \trinl_{\pw},$ and $\epsilon=96$ controls the first term on the right-hand side of 
\eqref{s4all} viz. 
 \begin{align}
  |b_{\pw}(\rho^{m,1/4},\zeta^{m+1/2},\mathcal{R}_{\m}v^{m,1/4})| & 
   \le {\cal C}(u,v)h^{2\sigma}+\frac{1}{192}\trinl\zeta^{m+1/2} \trinl_{\pw}^2\label{S41}.
\end{align}
  Analogous arguments with $m=1$ 
  establish
  \begin{align}
      b_{\pw}(\rho^{1,1/4},\zeta^{1/2},\mathcal{R}_{\m}v^{1,1/4}) \le {\cal C}(u,v)h^{2\sigma}+\frac{1}{192}\trinl\zeta^{1/2} \trinl_{\pw}^2 .\label{S42}
  \end{align}
  Exactly similar arguments with Lemma~\ref{axul_ritz}$(d)$ and Lemma~\ref{axul}$(f)$ lead to 
     \begin{align*}
          &b_{\pw}(\rho^{n+1,1/4}-\rho^{n,1/4},\zeta^{n+1/2},\mathcal{R}_{\m}v^{n+1,1/4})
         \le \constref{car}\constref{cdsj}\sqrt{k/2} h^{\sigma}\| u_t\|_{L^2(t_{n-1},t_{n+2};H^{2+\sigma}(\Omega))}\| v\|_{C([0,T];H^{2}(\Omega))}\trinl \zeta^{n+1/2}\trinl_{\pw},\\
         &b_{\pw}(\rho^{n,1/4},\zeta^{n+1/2},\mathcal{R}_{\m}(v^{n+1,1/4}-v^{n,1/4}))
         \le \constref{car}\constref{cdsj}\sqrt{k/2} h^{\sigma}\| u\|_{C([0,T];H^{2+\sigma}(\Omega))}\| v_t\|_{L^2(t_{n-1},t_{n+2};H^2(\Omega))}\trinl \zeta^{n+1/2}\trinl_{\pw}.
     \end{align*}
 A Young's inequality with $b=\sqrt{k/2}\trinl \zeta^{n+1/2}\trinl_{\pw}$, $\epsilon=1$, $a= \constref{car}\constref{cdsj} h^{\sigma}\| u_t\|_{L^2(t_{n-1},t_{n+2};H^{2+\sigma}(\Omega))}\| v\|_{C([0,T];H^{2}(\Omega))}$ (resp. $a=\constref{car}\constref{cdsj} h^{\sigma}\| u\|_{C([0,T];H^{2+\sigma}(\Omega))}\| v_t\|_{L^2(t_{n-1},t_{n+2};H^2(\Omega))}$),  and the sum for $n=1$ to $m$, and some re-arrangements reveal
     \begin{align*}
        &\sum_{n=1}^{m-1} b_{\pw}(\rho^{n+1,1/4},\zeta^{n+1/2},\mathcal{R}_{\m}(v^{n+1,1/4}-v^{n,1/4}))\le  {\cal C}(u,v_t)h^{2 \sigma} +\frac{k}{4}\trinl \zeta^{n+1/2}\trinl_{\pw}^2\\
        &\Big(\text{resp. }\sum_{n=1}^{m-1} b_{\pw}(\rho^{n+1,1/4}-\rho^{n,1/4},\zeta^{n+1/2},\mathcal{R}_{\m}v^{n+1,1/4})
        \le  {\cal C}(u_t,v)h^{2 \sigma} +\frac{k}{4}\trinl \zeta^{n+1/2}\trinl_{\pw}^2\Big).
\end{align*}
A combination of \eqref{S41}-\eqref{S42} and the last two displayed inequalities  in \eqref{s4all} lead to
\begin{align*}
    S_3 \le  {\mathcal C}(u,v,u_t,v_t)  h^{2\sigma} + \frac{1}{96}\big[\trinl \zeta^{1/2}\trinl_{\pw}^2+\trinl \zeta^{m+1/2}\trinl_{\pw}^2\big]+k\sum_{n=0}^{m-1} \trinl \zeta^{n+1/2} \trinl_{\pw}^2
\end{align*}
with the addition of non-negative term $k\trinl \zeta^{1/2} \trinl_{\pw}^2$ on the right-hand side. \qed

\medskip \noindent {\it Bound for $S_4$.} The inequalities in  \eqref{Bh-bdd}, \eqref{r-con},  and  Lemma~\ref{axul_ritz}$(b)$ followed by the definition $k\delta_t u^{n}  =u^{n+1/2}-u^{n-1/2}$,  and  the bound from Lemma~\ref{axul}$(d)$ reveal
    \begin{align*}
       & kb_{\pw}(\rho^{n,1/4},\mathcal{R}_{\m}\delta_t u^n,\theta^{n,1/4}) \le  \constref{car}\constref{cdsj}\sqrt{k/2}h^{\sigma}\|u\|_{C([0,T];H^{2+\sigma}(\Omega))}\|u_t\|_{L^2(t_{n-1},t_{n+1};H^{2}(\Omega))}\trinl \theta^{n,1/4} \trinl_{\pw}.
    \end{align*} 
    The  
identity \eqref{un14}, a triangle inequality, and  some elementary manipulations lead to 
\begin{align}
  \frac{k}{T}\sum_{n=1}^m \trinl\theta^{n,1/4}\trinl_{\pw}^2& \le \frac{k}{2T}\sum_{n=1}^m\big[\trinl\theta^{n-1/2}\trinl_{\pw}^2+\trinl\theta^{n+1/2}\trinl_{\pw}^2 \big] \nonumber\\
 &\le \frac{k}{2T}\big[\trinl\theta^{1/2}\trinl_{\pw}^2+\trinl\theta^{m+1/2}\trinl_{\pw}^2\big]+\frac{k}{T}\sum_{n=1}^{m-1} \trinl\theta^{n+1/2}\trinl_{\pw}^2 \nonumber\\
  &\le \frac{1}{2}\big[\trinl\theta^{1/2}\trinl_{\pw}^2+\trinl\theta^{m+1/2}\trinl_{\pw}^2\big]+\frac{k}{T}\sum_{n=1}^{m-1} \trinl\theta^{n+1/2}\trinl_{\pw}^2
  \label{thetan14}
\end{align}
with $k \le T$ for the first two terms on the right-hand side of last inequality. This, a Young's inequality with $a=\constref{car}\constref{cdsj}h^{\sigma}\|u\|_{C([0,T];H^{2+\sigma}(\Omega))}\|u_t\|_{L^2(t_{n-1},t_{n+1};H^{2}(\Omega))}$, $b=\sqrt{k/2}\trinl\theta^{n,1/4}\trinl_{\pw}$, and $\epsilon=4/3T$ lead to
\begin{align*}
    S_4&:=2k\sum_{n=1}^{m-1} b_{\pw}(\rho^{n,1/4},\mathcal{R}_{\m}\delta_t u^n,\theta^{n,1/4}) 
    \le  {\cal C}(u,u_t)h^{2\sigma}+\frac{3}{16}\big[\trinl\theta^{1/2}\trinl_{\pw}^2+\trinl\theta^{m+1/2}\trinl_{\pw}^2\big]+\frac{3k}{8T}\sum_{n=0}^{m-1} \trinl\theta^{n+1/2}\trinl_{\pw}^2
\end{align*}
with an addition of non-negative term $\frac{3k}{8T} \trinl\theta^{1/2}\trinl_{\pw}^2$ on the right-hand side.
\qed
 
     \medskip \noindent {\it Bound for $S_5$.} 
     The details of the bound for $S_5:=-2k\sum_{n=1}^{m}b_{\pw}(u^{n,1/4},\delta _t\zeta^{n},\chi^{n,1/4})$ follows  verbatim from the bound for $S_3$ in   with $u^{n,1/4}$ replaced by $\rho^{n,1/4}$ and $\chi^{n,1/4}$ by $\mathcal{R}_{\m}v^{n,1/4}$. 
Hence the details are skipped and the final bound is 
\begin{align*}
    S_5 \le {\mathcal C}(u,v,u_t,v_t)  h^{2\sigma} + \frac{1}{96}\big[\trinl \zeta^{1/2}\trinl_{\pw}^2+\trinl\zeta^{m+1/2}\trinl_{\pw}^2\big]+ k\sum_{n=0}^{n-1}\trinl \zeta^{n+1/2}\trinl_{\pw}^2. \qed
\end{align*}
\noindent{\it Bound for $S_6$.}
The boundedness from \eqref{Bh-bdd}, Lemma~\ref{axul}$(b)$,  $k\delta_t \rho^{n}  =\rho^{n+1/2}-\rho^{n-1/2}$  from \eqref{dtu}, and   Lemma~\ref{axul_ritz}$(c)$  lead to 
 \begin{align*}
  kb_{\pw}(u^{n,1/4},\delta_t \rho^{n},\theta^{n,1/4}) 
& \le C_{\rm dS}\constref{car}\sqrt{k/2}h^{\sigma}\|u\|_{C([0,T];H^2(\Omega))}\|u_t\|_{L^2(t_{n-1},t_{n+2};H^{2+\sigma}(\Omega))}\trinl\theta^{n,1/4}\trinl_{\pw}.
 \end{align*}
 This and analogous arguments as for the bound of $S_4$ lead to
 \begin{align*}
   S_6
   \le {\cal C}(u,u_t)h^{2\sigma}+\frac{3}{16}\big[\trinl\theta^{1/2}\trinl_{\pw}^2+\trinl\theta^{m+1/2}\trinl_{\pw}^2\big]+\frac{3k}{8T}\sum_{n=0}^{m-1} \trinl\theta^{n+1/2}\trinl_{\pw}^2. \qquad \qed
 \end{align*}
\medskip \noindent {\it Bound for $S_7$.}
    The identity  $k\delta _t= \zeta^{n+1/2}-\zeta^{n-1/2}$ from  \eqref{dtu} and elementary manipulations analogous to $S_3$ reveal
\begin{align}
 \frac{1}{2} S_7& :=   k\sum_{n=1}^{m}b_{\pw}(u^{n,1/4},(I-J)\delta _t\zeta^{n},v^{n,1/4})=  b_{\pw}(u^{m,1/4},(I-J)\zeta^{m+1/2},v^{m,1/4}) \nonumber \\
 &\qquad + b_{\pw}(u^{1,1/4},(J-I)\zeta^{1/2},v^{1,1/4}) +\sum_{n=1}^{m-1} b_{\pw}(u^{n+1,1/4}-u^{n,1/4},(J-I)\zeta^{n+1/2},v^{n+1,1/4})\nonumber\\
   &\qquad+ \sum_{n=1}^{m-1}b_{\pw}(u^{n,1/4},(J-I)\zeta^{n+1/2},v^{n+1,1/4}-v^{n,1/4}).\label{s7}
    \end{align}
Lemma~\ref{IB}$(e)$ and a Young's inequality with $a=\constref{C3.4b} 
     h^\sigma \norm{u^{m,1/4}}_{ H^{2+\sigma}(\Omega) }\norm{v^{m,1/4}}_{H^2(\Omega)},$ $b=\trinl  \zeta^{m+1/2} \trinl _{\pw},$ and $\epsilon=96$ provide
    \begin{align}
       b_{\pw}(u^{m,1/4},(I-J)\zeta^{m+1/2},v^{m,1/4}) &\le   \constref{C3.4b} 
     h^\sigma\norm{u^{m,1/4}}_{H^{2+\sigma}(\Omega)} \trinl \zeta^{m+1/2} \trinl _{\pw}\norm{v^{m,1/4}}_{H^2(\Omega)}\nonumber\\
  & \le  {\cal C}(u,v)
     h^{2\sigma} +\frac{1}{192}\trinl \zeta^{m+1/2} \trinl _{\pw}^2.\label{34-0}    \end{align}
    Analogous arguments with $m=1$ show
    \begin{align}
       b_{\pw}(u^{1,1/4},(I-J)\zeta^{1/2},v^{1,1/4}) &\le  {\cal C}(u,v)
     h^{2\sigma}+\frac{1}{192}\trinl \zeta^{1/2} \trinl _{\pw}^2   \label{34-00}.
    \end{align}
Utilize Lemma~\ref{IB}$(e)$,  Lemma~\ref{axul}($b,f$), 
and a Young's inequality with $a= \constref{C3.4b}h^{\sigma}\| u_t\|_{L^2(t_{n-1},t_{n+2};H^{2+\sigma}(\Omega))}   \|v\|_{C(0,T;H^2(\Omega))}$ $(\text{resp. } a=\constref{C3.4b}h^{\sigma}\| u\|_{C(0,T;H^{2+\sigma}(\Omega))}$ $   \|v_t\|_{L^2(t_{n-1},t_{n+2};H^{2}(\Omega))}))$, $b=\sqrt{k/2}\trinl \zeta^{n+1/2} \trinl _{\pw}$, $\epsilon=1$, followed by some elementary manipulations to bound the third (resp. fourth) term on the right-hand side of \eqref{s7} as
    \begin{align*}
        \sum_{n=1}^{m-1}b_{\pw}(u^{n+1,1/4}-u^{n,1/4},(J-I)\zeta^{n+1/2},v^{n+1,1/4}) \le {\cal C}(u_t,v)
     h^{2\sigma}+\frac{k}{4}\sum_{n=1}^{m-1}\trinl \zeta^{n+1/2} \trinl _{\pw}^2\\
        \Big( \text{resp. }\sum_{n=1}^{m-1}b_{\pw}(u^{n,1/4},(J-I)\zeta^{n+1/2},v^{n+1,1/4}-v^{n,1/4}) \le{\cal C}(u,v_t)
     h^{2\sigma}+\frac{k}{4}\sum_{n=1}^{m-1}\trinl \zeta^{n+1/2} \trinl _{\pw}^2\Big).
    \end{align*}
Combining this estimate with \eqref{34-0}-\eqref{34-00} in \eqref{s7} leads to 
    \begin{align*}
      S_7 \le   {\mathcal C}(u,v,u_t,v_t)  h^{2\sigma}+\frac{1}{96}\big[\trinl \zeta^{1/2}\trinl_{\pw}^2+\trinl \zeta^{m+1/2}\trinl_{\pw}^2\big]+{k}\sum_{n=0}^{m-1}\trinl\zeta^{n+1/2} \trinl_{\pw}^2. \qed 
    \end{align*}
\medskip \noindent {\it Bound for $S_8$.}
   {Apply Lemma~\ref{IB}$(d)$, $k\delta_tu^n=u^{n+1/2}-u^{n-1/2}$  from \eqref{dtu}, and Lemma~\ref{axul}($b,d)$   to obtain}
\begin{align*}
kb_{\pw}(\delta_tu^n,u^{n,1/4},& (J-I)\theta^{n,1/4})  \le  \constref{C3.4a}h^{\sigma}\trinl k \delta_tu^n\trinl_{\pw}\trinl u^{n,1/4}\trinl_{\pw}\trinl \theta^{n,1/4}\trinl_{\pw}\\
& \le  \constref{C3.4a}\sqrt{k/2}h^{\sigma}\|u_t\|_{L^2(t_{n-1},t_{n+1};H^{2}(\Omega))}\| u\|_{C([0,T];H^2(\Omega))}\trinl \theta^{n,1/4}\trinl_{\pw}.
\end{align*}\noindent
This, a Young's inequality with $a=\constref{C3.4a}h^{\sigma}\|u_t\|_{L^2(t_{n-1},t_{n+1};H^{2}(\Omega))}\| u\|_{C([0,T];H^2(\Omega))},b=\sqrt{k/2}\trinl \theta^{n,1/4}\trinl_{\pw}$, and $\epsilon=4/3T$ followed by some elementary manipulations and \eqref{thetan14}  lead to
\begin{align*}
S_8:=2k\sum_{n=1}^{m}b_{\pw}(\delta_tu^n,u^{n,1/4},& (J-I)\theta^{n,1/4})  \le   {\cal C}(u,u_t)h^{2\sigma}+\frac{3}{16}\big[\trinl\theta^{1/2}\trinl_{\pw}^2+\trinl\theta^{m+1/2}\trinl_{\pw}^2\big]+\frac{3k}{8T}\sum_{n=0}^{m-1} \trinl\theta^{n+1/2}\trinl_{\pw}^2
\end{align*}
with an addition of the non-negative term  $\frac{3k}{8T} \trinl\theta^{1/2}\trinl_{\pw}^2$ on the right-hand side.
\medskip \noindent 
A combination of \eqref{TJ}, \eqref{sn}, and the bounds for $S_1, S_2, \cdots, S_8$ from  above conclude the proof.
\end{proof}
\noindent The subsequent lemma controls  the truncation error for the nonlinear term $T_4$ (in error equation \eqref{TJ}).
\begin{lemma}[Truncation error for the nonlinear term $T_4$]\label{lem-non-trunc}If $u_t,v_t \in H^1(0,T;H^2(\Omega))$, then for any $m$ with $1 \le m \le N-1$,
 the following bound holds:
\begin{align*}
      T_4&:=k^3\sum_{n=1}^m  \big[b(\delta_tu^{n},J\delta_t \zeta^n,\delta_tv^{n})+\frac{k^2}{8}b(\bar{\partial}^2_tu^{n},J\delta_t \zeta^n,\bar{\partial}^2_t v^n)  +\frac{1}{2}b(\bar{\partial}^2_tu^{n},\delta_tu^{n},J\theta^{n,1/4}) \big] \\
&\le {\cal C }(u_t,u_{tt},v_t,v_{tt})k^4+\frac{1}{16}\big[\trinl\zeta^{1/2}\trinl_{\pw}^2+\trinl\theta^{1/2}\trinl_{\pw}^2+\trinl\zeta^{m+1/2}\trinl_{\pw}^2+\trinl\theta^{m+1/2}\trinl_{\pw}^2\big]\\
&\quad +k\big(1+\frac{1}{16T}\big)\sum_{n=0}^{m-1}\trinl\zeta^{n+1/2}\trinl_{\pw}^2+\frac{k}{8T}\sum_{n=0}^{m-1}\trinl\theta^{n+1/2}\trinl_{\pw}^2.
\end{align*}
\end{lemma}
\begin{proof}
The identity $k\delta _t\zeta^{n}=\zeta^{n+1/2}-\zeta^{n-1/2}$ from  \eqref{dtu} and some elementary algebra lead to 
\begin{align}
   &k\sum_{n=1}^m  b(\delta_tu^{n},J\delta_t \zeta^n,\delta_tv^{n})=b(\delta_tu^{m},J \zeta^{m+1/2},\delta_tv^{m})-b(\delta_tu^{1},J\zeta^{1/2},\delta_tv^{1})\nonumber\\
  &\quad-\sum_{n=1}^{m-1}b(\delta_tu^{n+1}-\delta_tu^{n},J \zeta^{n+1/2},\delta_tv^{n+1})-\sum_{n=1}^{m-1}b(\delta_tu^{n},J \zeta^{n+1/2},\delta_tv^{n+1}-\delta_tv^{n}).\label{uist}
\end{align}
We establish bounds for the terms on the right-hand side of \eqref{uist}. 
Lemma~\ref{IB}$(a)$, \eqref{23} (and then with $v$ replaced by $u$) reveal 
 \begin{align}
   b(\delta_tu^{m},J \zeta^{m+1/2},\delta_tv^{m})&\le  \constref{ccj}   \trinl \delta_tu^{m} \trinl_{\pw}\trinl \zeta^{m+1/2} \trinl_{\pw}\trinl \delta_tv^{m} \trinl_{\pw} \nonumber\\
   &\le \constref{ccj} \|u_t\|_{L^\infty(0,T;H^2(\Omega))} \|v_t\|_{L^\infty(0,T;H^2(\Omega))}\trinl \zeta^{m+1/2}\trinl_{\pw}. \nonumber
 \end{align}
A  Young's inequality with $a=\constref{ccj} \|u_t\|_{L^\infty(0,T;H^2(\Omega))} \|v_t\|_{L^\infty(0,T;H^2(\Omega))}$, $b=\trinl\zeta^{m+1/2}\trinl_{\pw},$ and $\epsilon=16k^2$ provides
\begin{align}
      b(\delta_tu^{m},J \zeta^{m+1/2},\delta_tv^{m})  &\le {\cal C}(u_t,v_t)k^2+\frac{k^{-2}}{32}\trinl \zeta^{m+1/2}\trinl_{\pw}^2.\label{TR1}
\end{align}
Analogous arguments show
\begin{align}
   |b(\delta_tu^{1},J \zeta^{1/2},\delta_tv^{1})| \le {\cal C}(u_t,v_t)k^2+\frac{k^{-2}}{32}\trinl \zeta^{1/2}\trinl_{\pw}^2.\label{tr2}
 \end{align}
Utilize Lemma~\ref{IB}$(a)$, the identity $\delta_tu^{n+1}-\delta_tu ^{n}=\frac{k}{2}(\bar{\partial}_t^2u^{n+1}+\bar{\partial}_t^2u^{n})$ \big(resp. $\delta_tv^{n+1}-\delta_tv ^{n}=\frac{k}{2}(\bar{\partial}_t^2v^{n+1}+\bar{\partial}_t^2v^{n})$\big) obtained from \eqref{varphi2}, and   \eqref{23} (for $v$ and $u$) to obtain
\begin{align*}
    &|b(\delta_tu^{n+1}-\delta_tu^{n},J \zeta^{n+1/2},\delta_tv^{n+1})| \le  \frac{\constref{ccj}}{2}k\trinl \bar{\partial}_t^2u^{n+1}+\bar{\partial}_t^2u^{n} \trinl_{\pw}\|v_t\|_{L^\infty(0,T;H^2(\Omega))}\trinl\zeta^{n+1/2} \trinl_{\pw},\\
    &\big(\text{resp. }|b(\delta_tu^{n},J \zeta^{n+1/2},\delta_tv^{n+1}-\delta_tv^{n})| \le \frac{\constref{ccj}}{2}k\|u_t\|_{L^\infty(0,T;H^2(\Omega))}\trinl \bar{\partial}_t^2v^{n+1}+\bar{\partial}_t^2v^{n} \trinl_{\pw}\trinl\zeta^{n+1/2} \trinl_{\pw}\big).
\end{align*}
A Young's  inequality with $a=\frac{\constref{ccj}}{2}k\trinl \bar{\partial}_t^2u^{n+1}+\bar{\partial}_t^2u^{n} \trinl_{\pw}\|v_t\|_{L^\infty(0,T;H^2(\Omega))}$, $b=\trinl\zeta^{n+1/2} \trinl_{\pw}$, and $\epsilon=k$ (resp. $a=\frac{\constref{ccj}}{2}k\|u_t\|_{L^\infty(0,T;H^2(\Omega))}\trinl \bar{\partial}_t^2v^{n+1}+\bar{\partial}_t^2v^{n} \trinl_{\pw}$ $, b=\trinl\zeta^{n+1/2} \trinl_{\pw}$ and $\epsilon=k$) and the  triangle inequality lead to
\begin{align*}
        &|b(\delta_tu^{n+1}-\delta_tu^{n},J \zeta^{n+1/2},\delta_tv^{n+1})|  \le  \frac{{\constref{ccj}}^2}{4}k^3\big(\trinl \bar{\partial}_t^2u^{n+1}\trinl_{\pw}^2+\trinl \bar{\partial}_t^2u^{n} \trinl_{\pw}^2\big)\|v_t\|^2_{L^\infty(0,T;H^2(\Omega))}+\frac{1}{2k}\trinl\zeta^{n+1/2} \trinl_{\pw}^2,\\
    &| b(\delta_tu^{n},J \zeta^{n+1/2},\delta_tv^{n+1}-\delta_tv^{n}) |\le \frac{{\constref{ccj}}^2}{4}k^3\|u_t\|_{L^\infty(0,T;H^2(\Omega))}^2\big(\trinl \bar{\partial}_t^2v^{n+1}\trinl_{\pw}^2+\trinl \bar{\partial}_t^2v^{n} \trinl_{\pw}^2 \big)+\frac{1}{2k}\trinl\zeta^{n+1/2} \trinl_{\pw}^2.
\end{align*}
A combination of these two inequalities and  \eqref{TR1}-\eqref{tr2} in \eqref{uist}, the estimate $\sum_{n=1}^{m-1}\big(\trinl \bar{\partial}_t^2\varphi^{n+1}\trinl_{\pw}^2+\trinl \bar{\partial}_t^2\varphi^{n}\trinl_{\pw}^2\big) \le 2\sum_{n=1}^m\|\bar{\partial}_t^2\varphi^{n}\|_{H^2(\Omega)}$, and Lemma~\ref{axul}$(g)$ (applied twice)  reveals 
\begin{align}
    k^3\sum_{n=1}^m  b(\delta_tu^{n},J\delta_t \zeta^n,\delta_tv^{n}) \le {\cal{C}}(u_t,u_{tt},v_t,v_{tt})k^4+\frac{1}{32}\big[\trinl \zeta^{1/2}\trinl_{\pw}^2+\trinl \zeta^{m+1/2}\trinl_{\pw}^2\big]+k\sum_{n=1}^{m-1}\trinl\zeta^{n+1/2} \trinl_{\pw}^2.\label{Tc1}
\end{align}
Utilizing Lemma~\ref{IB}$(a)$,   the inequality $\|k\bar{\partial}^2_t v^n\|_{H^2(\Omega)}\le \|\bar{\partial}_t v^{n+1/2}\|_{H^2(\Omega)}+\|\bar{\partial}_t v^{n-1/2}\|_{H^2(\Omega)} \le  2\| v_t\|_{L^\infty(0,T;H^2(\Omega))}$ (obtained from triangle inequality and  Lemma~\ref{axul}$(c)$, and Young's inequality with $a= 2\constref{ccj}\| \bar{\partial}^2_tu^{n}\|_{H^2(\Omega)}\| v_t\|_{L^\infty(0,T;H^2(\Omega))}$, $b=\trinl\delta _t\zeta^{n}\trinl_{\pw},$ and $\epsilon=4Tk$, gives 
\begin{align*}
kb(\bar{\partial}^2_tu^{n},J\delta _t\zeta^{n},\bar{\partial}^2_t v^n)
  & \le \constref{ccj}\trinl \bar{\partial}^2_tu^{n}\trinl_{\pw}\trinl\delta_t\zeta^{n}\trinl_{\pw}\trinl k\bar{\partial}^2_t v^n \trinl_{\pw} \le 2\constref{ccj}\| \bar{\partial}^2_tu^{n}\|_{H^2(\Omega)}\| v_t\|_{L^\infty(0,T;H^2(\Omega))}\trinl\delta_t\zeta^{n}\trinl_{\pw} \\
  & \le 8\constref{ccj}^2Tk\| \bar{\partial}^2_tu^{n}\|^2_{H^2(\Omega)}\| v_t\|^2_{L^\infty(0,T;H^2(\Omega))}+\frac{1}{8Tk}\trinl\delta_t\zeta^{n}\trinl_{\pw}^2.
\end{align*}
A summation from $n=1$ to $m$ and  Lemma~\ref{axul}$(g)$  shows
\begin{align*}
 \frac{k^5}{8}\sum_{n=1}^mb(\bar{\partial}^2_tu^{n},J\delta _t\zeta^{n},\bar{\partial}^2_t v^n)   & \le {\cal C}(u_{tt},v_t)k^4+\frac{k^3}{64T}\sum_{n=1}^m\trinl\delta_t\zeta^{n}\trinl_{\pw}^2\nonumber.
\end{align*}
     The identity $k\delta_t\zeta^{n}=\zeta^{n+1/2}-\zeta^{n-1/2}$ from \eqref{dtu} followed by the  triangle inequality  leads to 
\begin{align*}
  \frac{k^3}{4T}\sum_{n=1}^m \trinl\delta_t\zeta^{n}\trinl_{\pw}^2 &\le \frac{k}{2T}\sum_{n=1}^m\big[\trinl\zeta^{n-1/2}\trinl_{\pw}^2+\trinl\zeta^{n+1/2}\trinl_{\pw}^2 \big]  \le \frac{k}{2T}\big[\trinl\zeta^{1/2}\trinl_{\pw}^2+\trinl\zeta^{m+1/2}\trinl_{\pw}^2\big] \nonumber\\
  &\qquad+\frac{k}{T}\sum_{n=1}^{m-1} \trinl\zeta^{n+1/2}\trinl_{\pw}^2\le \frac{1}{2}\big[\trinl\zeta^{1/2}\trinl_{\pw}^2+\trinl\zeta^{m+1/2}\trinl_{\pw}^2\big]+\frac{k}{T}\sum_{n=1}^{m-1} \trinl\zeta^{n+1/2}\trinl_{\pw}^2
\end{align*}
with $k \le T$ for the first two terms on the right-hand side. A combination of the last two inequalities reveal
\begin{align*}
 \frac{k^5}{8}\sum_{n=1}^mb(\bar{\partial}^2_tu^{n},J\delta _t\zeta^{n},\bar{\partial}^2_t v^n)   & \le {\cal C}(u_{tt},v_t)k^4+\frac{1}{32}\big[\trinl \zeta^{1/2}\trinl_{\pw}^2+\trinl \zeta^{m+1/2} \trinl_{\pw}^2\big]+\frac{k}{16T}\sum_{n=1}^{m-1}\trinl \zeta^{n+1/2}\trinl_{\pw}^2.
\end{align*}
 Lemma~\ref{IB}$(a)$ and Lemma~\ref{axul}$(e)$, 
a 
Young's inequality with $a=\constref{ccj}\|\bar{\partial}^2_tu^{n}\|_{H^2(\Omega)}\| u_t \|_{L^\infty(0,T;H^2(\Omega))} ,b=\trinl\theta^{n,1/4}\trinl_{\pw}$, and $\epsilon=2Tk^2$ reveal
    \begin{align*}
b(\bar{\partial}^2_tu^{n},\delta_tu^{n},J\theta^{n,1/4}) \le \constref{ccj}^2Tk^2\|\bar{\partial}^2_tu^{n}\|^2_{H^2(\Omega)}\| u_t \|^2_{L^\infty(0,T;H^2(\Omega))}+\frac{1}{4Tk^2}\trinl\theta^{n,1/4}\trinl_{\pw}^2.
    \end{align*}
    This,  together with the  estimate from Lemma~\ref{axul}$(g)$ and \eqref{thetan14} lead to
    \begin{align}
         \frac{k^3}{2}\sum_{n=1}^{m}b(\bar{\partial}^2_tu^{n},\delta_tu^{n},J\theta^{n,1/4}) \le  {\cal C}(u_t,u_{tt})k^4+\frac{1}{16}\big[\trinl\theta^{1/2}\trinl_{\pw}^2+\trinl\theta^{m+1/2}\trinl_{\pw}^2\big] +\frac{k}{8T}\sum_{n=1}^{m-1}\trinl\theta^{n+1/2}\trinl_{\pw}^2. \label{Tc3}
    \end{align}
The proof follows by  
    \eqref{Tc1}-\eqref{Tc3}, adding   $k\big(1+\frac{1}{16T}\big)\trinl\zeta^{1/2} \trinl_{\pw}^2+\frac{k}{8T}\trinl\theta^{1/2}\trinl_{\pw}^2$ on the right-hand side.
    \end{proof}
{We are in position to  present the proof of Theorem~\ref{TH-ERROR-FINAL}, which utilizes the bounds for $T_1$--$T_4$ from \eqref{eqn-ni} and 
Lemmas~\ref{f,u,r,rho-lem}--\ref{lem-non-trunc} into the error 
equation~\eqref{TJ} and concludes via a discrete Gronwall inequality.} \\

\medskip \noindent    {\textbf{{Proof of Theorem~\ref{TH-ERROR-FINAL}}}}.
The bounds for $T_2,T_3,$ and $T_4$ from  Lemmas \ref{f,u,r,rho-lem}-\ref{lem-non-trunc} in \eqref{TJ} lead to 
\begin{align*}
   &\frac{1}{8}\big[\|\bar{\partial}_t \zeta^{m+1/2}\|^2+\trinl{\zeta^{m+1/2}}\trinl_{\pw}^2+\trinl{\theta^{m+1/2}}\trinl_{\pw}^2\big] \le {\cal C}(u,u_t,u_{tt},u_{ttt},u_{tttt},v,v_t,v_{tt},f,f_t)(h^{2\sigma}+k^4)\\&+ \frac{15}{8}\norm{\bar{\partial}_t \zeta^{1/2}}^2+\frac{13}{8} \trinl {\zeta^{1/2}}\trinl_{\pw}^2  + \frac{13}{14} \trinl{\theta^{1/2}}\trinl_{\pw}^2+\frac{\mu^{*} k}{T} \sum_{n=0}^m\frac{1}{8}\big[\|\bar{\partial}_t \zeta^{n+1/2}\|^2+\trinl{\zeta^{n+1/2}}\trinl_{\pw}^2+\trinl{\theta^{n+1/2}}\trinl_{\pw}^2\big]
   \end{align*}
   with the constant 
   $$\mu^{*}=\max\left\{8,8T\big(\constref{cdsj}\big[\|u_t\|_{L^\infty(0,T;H^2(\Omega))}+\|v_t\|_{L^\infty(0,T;H^2(\Omega))}\big]+5+\frac{1}{16T}\big),8T\constref{cdsj}\|u_t\|_{L^\infty(0,T;H^2(\Omega))}+10\right\}.$$
   Next, we apply Lemma~\ref{P1 d-gronwall} with the following setting: $a_m= \frac{1}{8}\big[\|\bar{\partial}_t \zeta^{m+1/2}\|^2+\trinl{\zeta^{m+1/2}}\trinl_{\pw}^2+\trinl{\theta^{m+1/2}}\trinl_{\pw}^2\big]$, $b_m=0,$ $c_m={\cal C}(u,u_t,u_{tt},u_{ttt},u_{tttt},v,v_t,v_{tt},f,f_t)(h^{2\sigma}+k^4)+ \frac{15}{8}\norm{\bar{\partial}_t \zeta^{1/2}}^2+\frac{43}{32} \trinl {\zeta^{1/2}}\trinl_{\pw}^2  + \frac{13}{14} \trinl{\theta^{1/2}}\trinl_{\pw}^2$, and $\mu=\mu^{*}k/T$, and use that $mk \le T$ (which implies $e^{\frac{mk}{T}\mu^{*}} \le e^{\mu^{*}}$) and  the bounds from \eqref{eqn-ni} to obtain 
    \begin{align}
        \|\bar{\partial}_t \zeta^{m+1/2}\|^2+\trinl{\zeta^{m+1/2}}\trinl_{\pw}^2+\trinl{\theta^{m+1/2}}\trinl_{\pw}^2 &\lesssim h^{2\sigma}+k^4.\label{eqnm}
   \end{align}
  Inequality \eqref{eqnm} plus triangle inequality (applied thrice) show $\norm{\bar{\partial}_t(u^{m+1/2}-u_{\m}^{m+1/2})}+ \trinl u^{m+1/2}-u_{\m}^{m+1/2}\trinl_{\pw}+ \trinl v^{m+1/2}-v_{\m}^{m+1/2} \trinl_{\pw}\le  \norm{\bar{\partial}_t\rho^{m+1/2}}+ \trinl \rho^{m+1/2}\trinl_{\pw}+ \trinl \chi^{m+1/2} \trinl_{\pw} + \norm{\bar{\partial}_t\zeta^{m+1/2}}+ \trinl \zeta^{m+1/2}\trinl_{\pw}+ \trinl\theta^{m+1/2} \trinl_{\pw}$. Applying   Lemma~\ref{P1 ritz_lemma} for the first three terms and \eqref{eqnm} to the remaining terms, we conclude the proof.
\section{Numerical experiments}\label{num-sec}
In this section we present the outcomes of our numerical experiments performed using the fully-discrete
scheme described in Section~\ref{sec-main results}, with the goal to confirm the theoretical findings established 
in Theorems~\ref{LEM-ERROR-INI} and \ref{TH-ERROR-FINAL}. The numerical realization of the  fully-discrete formulation \eqref{P4fully}-\eqref{app_at_0} is carried out in an open-source finite element library \texttt{FreeFem++} \cite{MR3043640}. All simulations were performed on a workstation equipped with a 32-core processor, $2 \times 32$~GB RAM, a 2~TB enterprise HDD, and an NVIDIA RTX 4000 Ada 20~GB GDDR6 GPU, housed in a convertible tower cabinet. A Newton's iterative procedure is utilized to solve the nonlinear system arising at each time level, under a  user-specified tolerance of $\texttt{tol} := 10^{-7}$. {The details of Newton's iterations are given in the Table~\ref{tab-new}. Here $(\overset{(i)}{u_{\m}^n},\overset{(i)}{v_{\m}^n})$ denotes the $i^{\rm th}$ Newton iterate for the discrete solution at time level $t_n$, obtained by linearizing \eqref{P4fully} and \eqref{P1ic} about the previous iterate $(\overset{(i-1)}{u_{\m}^n},\overset{(i-1)}{v_{\m}^n})$. At each time step $t_n$ with $n\ge1$, the initial guess for the Newton loop is taken to be the converged solution $(u_{\m}^{n-1},v_{\m}^{n-1})$ from the previous time step, with the solution at $t=t_0$ itself obtained directly from the linear system \eqref{app_at_0} and hence requiring no Newton iteration. The loop terminates, i.e., the range of $i$ closes, as soon as the increment between successive iterates falls below $\texttt{tol}$.}
\begin{table}[h!]
\centering
\footnotesize
\begin{tabular}{@{}p{0.88\textwidth}@{}}
\toprule
\textbf{Newton iterations ($\texttt{tol} = 10^{-7}$)} \\
\midrule
\textit{Initial step corresponding to \eqref{P1ic}:}\\[0.5ex]
$\begin{aligned}[t]
&{2k^{-2}}( \overset{(i)}{u_{\m}^1}-{u_{\m}^0}-ku_1, \varphi_{{\m}} )_{L^2(\Omega)}+\frac{1}{2}a_{\pw}(\overset{(i)}{u_{\m}^{1}}+{u_{\m}^0},\varphi_{{\m}}) +\frac{1}{4}b_{\pw}(\overset{(i)}{u_{\m}^1},\varphi_{{\m}},\overset{(i-1)}{v_{\m}^{1}}+{v_{\m}^{0}})\\[4pt]
&+\frac{1}{4}b_{\pw}(\overset{(i-1)}{u_{\m}^{1}}+{u_{\m}^{0}},\varphi_{{\m}},\overset{(i)}{v_{\m}^{1}})=-\frac{1}{4}b_{\pw}({u_{\m}^0},\varphi_{{\m}},{v_{\m}^{0}}) +\frac{1}{4
}b_{\pw}(\overset{(i-1)}{u_{\m}^1},\varphi_{{\m}},\overset{(i-1)}{v_{\m}^{1}})+(f^{1/2},\varphi_{{\m}} )_{L^2(\Omega)},
\\[4pt]
&a_{\pw}(\overset{(i)}{v_{\m}^{1}},\psi_{{\m}}) -b_{\pw}(\overset{(i-1)}{u_{\m}^{1}},\overset{(i)}{u_{\m}^{1}},\psi_{{\m}})=-\frac{1}{2} b_{\pw}(\overset{(i-1)}{u_{\m}^{1}},\overset{(i-1)}{u_{\m}^{1}},\psi_{{\m}})+(g^{1},\psi_{{\m}} )_{L^2(\Omega)}.
\end{aligned}$\\[4pt]

\medskip
\noindent
\textit{General step ($n\geq 1$) corresponding to \eqref{P4fully}:}\\[0.5ex]

$\begin{aligned}[t]
&{k^{-2}}( \overset{(i)}{u_{\m}^{n+1}}-2{u_{\m}^{n}}+{u_{\m}^{n-1}}, \varphi_{{\m}} )_{L^2(\Omega)}+\frac{1}{4}a_{\pw}(\overset{(i)}{u_{\m}^{n+1}}+2{u_{\m}^{n}}+{u_{\m}^{n-1}},\varphi_{{\m}}) +\frac{1}{16}b_{\pw}(\overset{(i)}{u_{\m}^{n+1}},\varphi_{{\m}},\overset{(i-1)}{v_{\m}^{n+1}}+2{v_{\m}^{n}}+{v_{\m}^{n-1}})\\[4pt]
&+\frac{1}{16}b_{\pw}(\overset{(i-1)}{u_{\m}^{n+1}}+2{u_{\m}^{n}}+{u_{\m}^{n-1}},\varphi_{{\m}},\overset{(i)}{v_{\m}^{n+1}})=-\frac{1}{16}b_{\pw}(2{u_{\m}^n}+{u_{\m}^{n-1}},\varphi_{{\m}},2{v_{\m}^n}+{v_{\m}^{n-1}}) +\frac{1}{16}b_{\pw}(\overset{(i-1)}{u_{\m}^{n+1}},\varphi_{{\m}},\overset{(i-1)}{v_{\m}^{n+1}}) +(f^{n,1/4},\varphi_{{\m}} )_{L^2(\Omega)},\\[4pt]
&\frac{1}{2}a_{\pw}(\overset{(i)}{v_{\m}^{n+1}}+{v_{\m}^{n}},\psi_{{\m}}) -\frac{1}{4}b_{\pw}(\overset{(i)}{u_{\m}^{n+1}},\overset{(i-1)}{u_{\m}^{n+1}}+{u_{\m}^{n}},\psi_{{\m}})=\frac{1}{8}b_{\pw}({u_{\m}^{n}},{u_{\m}^{n}},\psi_{{\m}})-\frac{1}{8}b_{\pw}(\overset{(i-1)}{u_{\m}^{n+1}},\overset{(i-1)}{u_{\m}^{n+1}},\psi_{{\m}})+(g^{n+1/2},\psi_{{\m}} )_{L^2(\Omega)}.
\end{aligned}$\\
\bottomrule
\end{tabular}
\caption{Newton's iterative procedure for \eqref{P4fully} and  \eqref{P1ic}.}\label{tab-new}
\end{table}
For the tests in this article, Newton's method takes a maximum of four iterations per time step to converge.

\noindent The following notation is adopted for the convergence history tables in the rest of this subsection:
\begin{align*}
    &\texttt{Error}_0(u)=\underset{0 \le n \le N}{\max}\|u^{n}-u_{\m}^{n}\|, \quad  \texttt{Error}_0(v)=\underset{0 \le n \le N}{\max}\|v^{n}-v_{\m}^{n}\|,\\
    &  \texttt{Error}_2(u)=\underset{0 \le n \le N-1}{\max}\trinl u^{n+1/2}-u_{\m}^{n+1/2}\trinl_{\pw}, \quad  \texttt{Error}_2(v)=\underset{0 \le n \le N-1}{\max}\trinl v^{n+1/2}-v_{\m}^{n+1/2}\trinl_{\pw}.
    \end{align*}
    \textbf{Example 1.} (Convergence for smooth exact solution)
We choose the right-hand side load functions $(f,g)$ such that 
\begin{align*}
     f(x,y,t)&=u_{tt} + \Delta^2 u -[u,v], \\
g(x,y,t)&=\Delta^2 v +\frac{1}{2}[u,u],
\end{align*}
is satisfied on the domain $\Omega \times [0,T ] =(0,1)^2
\times [0,1 ]$ with the exact solution $(u,v)$ given
 by 
 $$u(x,y,t)=(1/10\sin(2\pi t)\sin(\pi x)\sin(\pi y))^2 \quad \text{ and }\quad v(x,y,t)=\exp(t)(x(x-1)y(y-1))^2 .$$
To compute the numerical convergence rates both in spatial and time discretization, we apply $k=h/4$ at each discretization step.  
 Table \ref{tab-smooth convergence} shows the errors and experimental convergence rates for the variables $u^n_{\m}$ and $v^n_{\m}$. The computational order of convergences in $L^2$ and energy norms are quasi-optimal and verify the
 theoretical results obtained in Theorem~\ref{LEM-ERROR-INI} and \ref{TH-ERROR-FINAL}.\\
\begin{table}[]
    \centering
    \begin{tabular}{|c|c|c|c|c|c|c|c|c|}
\hline
$h$&\texttt{Error}$_0(u$)&\texttt{Rate}&\texttt{Error}$_2(u$)&\texttt{Rate}&\texttt{Error}$_0(v$)&\texttt{Rate}&\texttt{Error}$_2(v$) &\texttt{Rate} \\ \hline
0.707 & 1.47e-02 &  $\star$ & 2.89e-01 &  $\star$& 6.94e-03 &  $\star$& 1.90e-01 &  $\star$\\ \hline
0.354 & 3.37e-03 & 2.124 & 1.14e-01 & 1.346 & 3.82e-03 & 0.859 & 1.39e-01 & 0.452\\ \hline
0.177 & 8.24e-04 & 2.032 & 6.01e-02 & 0.921 & 1.14e-03 & 1.749 & 7.72e-02 & 0.849\\ \hline
0.088 & 2.14e-04 & 1.948 & 3.08e-02 & 0.965 & 3.00e-04 & 1.925 & 3.99e-02 & 0.952\\ \hline
0.044 & 5.18e-05 & 2.042 & 1.55e-02 & 0.992 & 7.60e-05 & 1.979 & 2.02e-02 & 0.984\\ \hline
0.022 & 1.28e-05 & 2.020 & 7.76e-03 & 0.996 & 1.91e-05 & 1.994 & 1.01e-02 & 0.995\\ \hline
0.011 & 3.18e-06 & 2.006 & 3.88e-03 & 0.999 & 4.77e-06 & 1.999 & 5.07e-03 & 0.998\\ \hline
0.006 & 7.94e-07 & 2.003 & 1.94e-03 & 1.000 & 1.19e-06 & 2.000 & 2.54e-03 & 0.999\\ \hline

    \end{tabular}
    \caption{Convergence history in $L^2$ and energy norms with smooth exact solution $(u,v)$.}
    \label{tab-smooth convergence}
\end{table}

\noindent \textbf{Example 2.}  (2D-3D plate deflection comparison) Let $\Omega \subset \mathbb{R}^2$ be the  mid-surface (assumed to lie in alignment with the $xy$-axis) of a thin, isotropic, flat plate $\hat{\Omega}=[0,1] \times [0,1] \times [-d/2, d/2]$ with a uniform thickness $  d $.  Building upon \cite{MR569597}, given a space-time dependent loading $(\hat{f}_1(t),\hat{f}_2(t),\hat{f}_3(t))=\boldsymbol{f} (t): \widehat{\Omega}  \to \mathbb{R}^3$, the dynamic 3D von K\'arm\'an model seeks the displacement field ${\bf u}=(\hat{u}_i)_{i=1}^3$ and the Piola--Kirchhoff   stress tensor $\boldsymbol{\sigma}=(\hat{\sigma}_{ij})_{i,j=1}^3$ such that 
\begin{subequations}\label{3Deqn}
\begin{align}
  & {\bf u}_{tt} -{\partial_j}\left(\hat{\sigma}_{ij}+\hat{\sigma}_{kj}{\partial_k}\hat{u}_i\right)={\bf f}(\hat{\bx},t) \qquad\text{ for } (\hat{\bx},t) \in \hat{\Omega}\times (0,T],\\
 &  \left(\frac{1+\nu}{E}\right)\hat{\sigma}_{ij}-\frac{\nu}{E}\hat{\sigma}_{ll}\delta_{ij}=\frac{1}{2}\left({\partial_i}\hat{u}_j+{\partial_j}\hat{u}_i+{\partial_i}\hat{u}_l{\partial_j}\hat{u}_l\right) \quad \text{ for } (\hat{\bx},t) \in   \hat{\Omega}\times [0,T],\\
 &\text{with the boundary conditions, }\hat{u}_3=0, \text{ and }\hat{u}_1,\;\hat{u}_2 \text{ are independent  of }z \text{ on }\partial \Omega \times [\-d/2,d/2] \times (0,T] ,\\
  &\left(\hat{\sigma}_{i3}+\hat{\sigma}_{l3}{\partial_l}\hat{u}_i\right)=0\text{ on } \Omega\times\{\pm \; d/2\}\times (0,T], \quad \frac{1}{d}\int_{-d/2}^{d/2}\left(\hat{\sigma}_{rs}+\hat{\sigma}_{ls}{\partial_l}\hat{u}_s\right){n}_{r}=0 \text{ on }\partial \Omega \times (0,T] 
\end{align}  
\end{subequations}
with $i,j,l \in \{1,2,3\}$ and $r,s \in \{1,2\}$.  Here, $E$ is Young's modulus, $\nu$ is the Poisson's ratio, and ${\bf n}=(n_1,n_2)$ is the unit outer normal to  $\partial \Omega.$ \\
A dimensional reduction analysis in \cite{MR569597} for static case derives a 2D model  \eqref{2deqn} from the 3D model \eqref{3Deqn} as: given the load function $f=\frac{1}{  d }\int_{ - d /2}^{  d /2} \hat{f}_3\;{\rm d}z$, seek
   transverse displacement averaged over thickness, 
  $
   u=\frac{1}{ d }\int_{ - d /2}^{  d /2} \hat{u}_3 \;{\rm d}z
  $, and the Airy stress function $v$ such that \cite[eqn. 6.11-6.15]{MR569597}
\begin{subequations}\label{2deqn}
\begin{align}
&u_{tt} + \frac{d^2 E}{12(1-\nu^2)}\Delta^2 u -[u,v] = f(\bx,t) \quad \text{ for }  (\bx,t) \in \Omega \times \left(0,T \right]\\
& \frac{1}{E}\Delta^2 v +\frac{1}{2}[u,u]= 0\quad\quad\quad (\text{ for } \bx,t) \in \Omega \times \left[0,T \right]\\
 & u=\frac{\partial u}{\partial {\bf n}} =0,\; v=\frac{\partial v}{\partial  {\bf n}} =0 \text{ on }\partial \Omega \times \left(0,T \right].
\end{align}
\end{subequations}
Our objective is to illustrate that the 2D von K\'arm\'an  model~\eqref{2deqn}, effectively approximates the 3D von K\'arm\'an  model \eqref{3Deqn} in the sense that if $U^n$ is the approximation of the displacement $ {u}$ of the 2D model \eqref{2deqn} at time $t=t_n$ computed with the discrete formulation \eqref{P4fully} (with the coefficients $\frac{d^2 E}{12(1-\nu^2)}$ (resp. ${1}/{E})$ in bilinear forms  of \eqref{P4 fully_discrete1} (resp.\eqref{P4 fully_discrete2}) and $\hat{\mathbf{U}}^n$ is the discrete solution of \eqref{3Deqn} at $t=t_n$, then $U^n$ approximates $\frac{1}{d}\int_{-d/2}^{d/2}\hat{U}_3^n\; {\rm d}z$ with $\hat{\mathbf{U}}^n=(\hat{U}_1^n,\hat{U}_2^n,\hat{U}_3^n)$. 
 \begin{figure}
     \centering
     \includegraphics[width=0.5\linewidth]{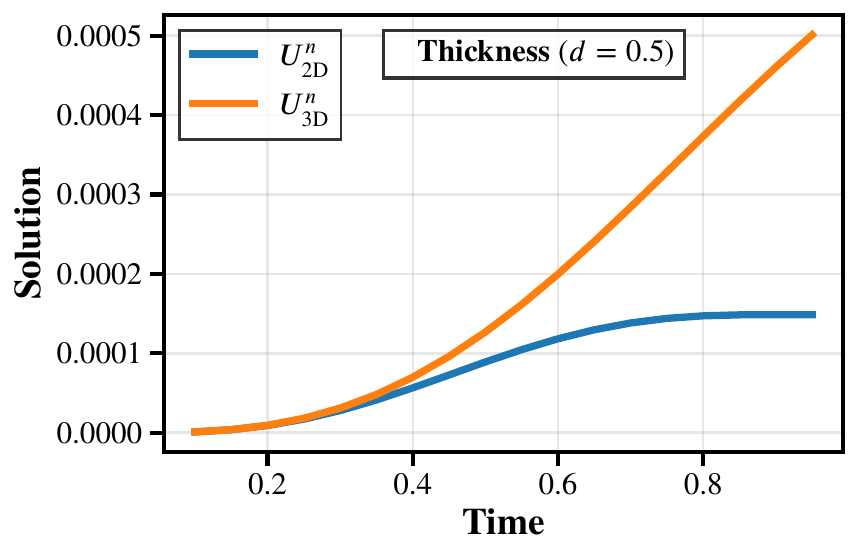}\includegraphics[width=0.5\linewidth]{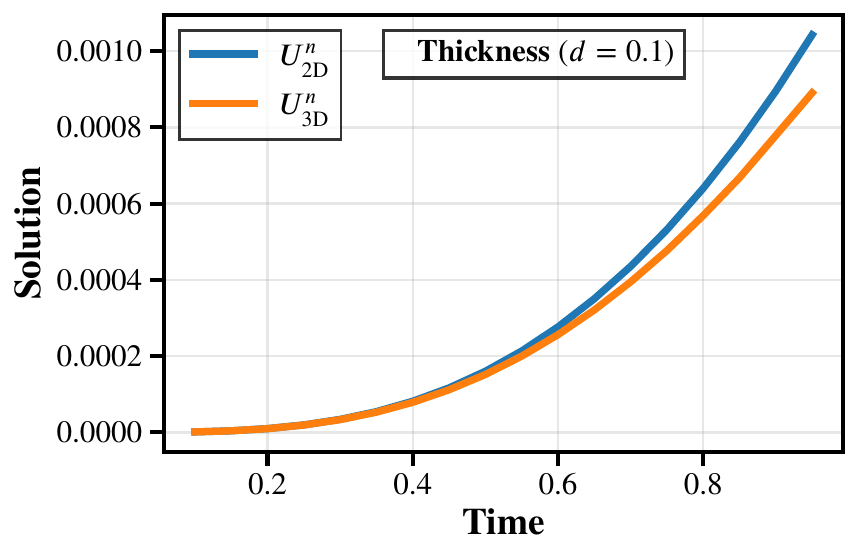}\\\includegraphics[width=0.5\linewidth]{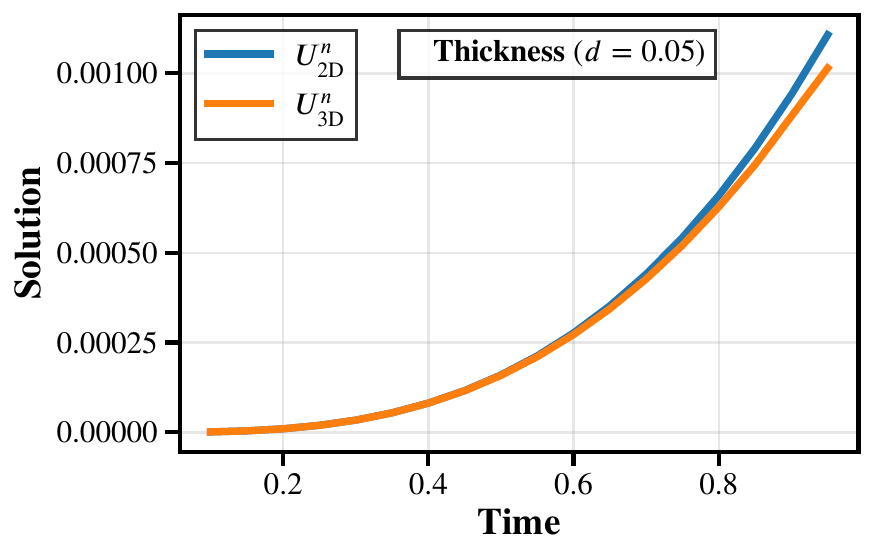}\includegraphics[width=0.5\linewidth]{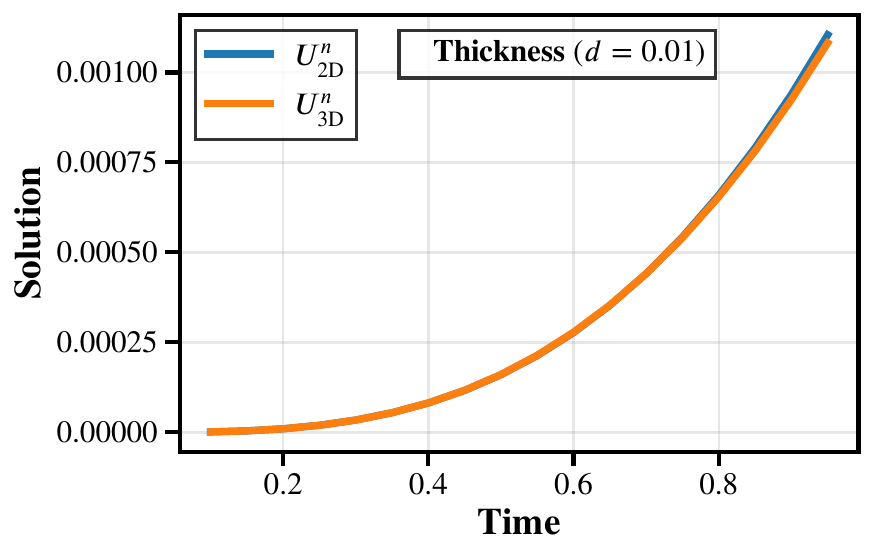}
     \caption{Example 2. 2D displacement  $U^n_{2\rm D}$ and 3D displacement  ${U}^n_{3\rm D}$  vs time $t_n$ with plate thickness $d=0.5,0.1,0.05,0.01$.}
    \label{fig2d3d}
 \end{figure}
To achieve this, we solve the 3D system \eqref{3Deqn} using continuous $H^1-$ conforming Lagrange finite element space: \( (\mathcal{P}_1(\mathcal{T}))^3  \) for displacement \( \boldsymbol{u} \), and \( (\mathcal{P}_0(\mathcal{T}))^9 \)-dG for stress. The temporal discretization is done using  the modified Newmark scheme. Homogeneous Dirichlet boundary conditions are set on all the sides except the surfaces $z=-d/2,d/2$ where the plate is assumed  traction free. The 3D load function reads
\begin{align}
\mathbf{f} = \left(0, 0, \frac{1}{100}t\sin(\pi x) \sin(\pi y)\right).\label{load-exp}
\end{align}
For 2D, we utilize $f=\frac{1}{d}\int_{-d/2}^{d/2}\hat{f}_3^n$ to obtain $f= {\rm d}z=\frac{1}{100}t\sin(\pi x) \sin(\pi y)$. The initial conditions are set to zero in both 2D and 3D cases. 
The elastic parameters take typical values for copper plates from~\cite{sherief2005half}. Let $T=1$, $\Delta t=1/20$ and consider the cells   $\hat{\Omega}_c=[5/64,6/64]\times [5/64,6/64]\times  [-d/2,d/2]$ and ${\Omega_c}=[5/64,6/64]\times [5/64,6/64]$.  At time $t=t_n$, we will use the following output quantities   
\begin{align*}
& U^n_{3\rm D}:=\frac{1}{|\hat{\Omega}_c|} \int_{\hat{\Omega}_c}   \hat{U}_3\;  d\hat{\bx}\; \text{ for } \hat{\bx}=(x,y,z) \text{ and }\;U^n_{2\rm D}:=\frac{1}{|{\Omega}_c|} \int_{{\Omega}_c}   {U}^n  \;d\bx\; \text{ for } \bx=(x,y).
\end{align*}
The simulations in Figure~\ref{fig2d3d} reveal that as the plate thickness \( d \) decreases  the results of the 2D model approximate those of the 3D model. As expected, the computational efficiency is significantly improved: The 2D model takes approximately  $25\%$ of time than that of 3D one.

\subsection*{Acknowledgments}
This work has been supported by the J.C. Bose grant ANRF/JBG/2025/000209/HAA and by the Australian Research Council through the Future Fellowship grant FT220100496.
\bibliographystyle{siam}
\bibliography{References}
\end{document}